\documentclass{article}
\usepackage{titlesec}
\usepackage{amsmath, amssymb, latexsym, amsthm, tikz, caption, subcaption}
\usepackage[margin=3.6cm]{geometry}
\usetikzlibrary{cd}
\usepackage{hyperref}

\hypersetup{
    colorlinks=true,
     citecolor=blue
}

\theoremstyle{plain}
\newtheorem{thm}{Theorem}[subsection]
\newtheorem{cor}[thm]{Corollary}
\newtheorem{lem}[thm]{Lemma}
\newtheorem{prop}[thm]{Proposition}

\newtheorem*{thm*}{Theorem}
\newtheorem*{cor*}{Corollary}
\newtheorem*{lem*}{Lemma}
\newtheorem*{prop*}{Proposition}
\newtheorem*{conj*}{Conjecture}

\theoremstyle{definition}
\newtheorem{defn}[thm]{Definition}
\newtheorem{ex}[thm]{Example}
\newtheorem{rem}[thm]{Remark}

\newtheorem*{defn*}{Definition}
\newtheorem*{ex*}{Example}
\newtheorem*{rem*}{Remark}
\newtheorem*{alg*}{Algorithm}
\newtheorem*{asi*}{Aside}

\theoremstyle{remark}
\newtheorem*{claim}{Claim}
\newtheorem*{ackno}{Acknowledgements}

\makeatletter
\def\blfootnote{\xdef\@thefnmark{}\@footnotetext}
\makeatother

\title{Ideal Polytopes for Representations of $GL_n(\mathbb{C})$}
\author{Teresa L\"udenbach}
\date{}

\begin{document}

\maketitle

\begin{abstract}
\blfootnote{
\emph{Date:} 25th August 2022 }
\blfootnote{
This work was supported by the Engineering and Physical Sciences Research Council {[EP/S021590/1]}. \\
The EPSRC Centre for Doctoral Training in Geometry and Number Theory {(The London School of Geometry and Number Theory)}, University College London, and King's College London.
}
In this paper we use the superpotential $\mathcal{W}_q$ for the flag variety $GL_n/B$ and particular coordinate systems that we call ‘ideal coordinates for $\mathbf{i}$’, to construct polytopes $\mathcal{P}^{\mathbf{i}}_{\lambda}$ inside $\mathbb{R}^{R_+}$, associated to highest weight representations $V_\lambda$ of $GL_n(\mathbb{C})$.
Here $\mathbf{i}$ is a reduced expression of the longest element of the Weyl group and $R_+$ is the set of positive roots of $GL_n$.
The lattice points of $\mathcal{P}^{\mathbf{i}}_{\lambda}$ can be used to encode a basis of the representation $V_{\lambda}$. In particular, for a specific choice of $\mathbf{i}$, the polytope $\mathcal{P}^{\mathbf{i}}_{\lambda}$ is unimodularly equivalent to a Gelfand-Tsetlin polytope.
The construction of the polytopes involves tropicalisation of the superpotential, namely $\mathcal{P}^{\mathbf{i}}_{\lambda}$ is $\left\{ \mathrm{Trop}\left(\mathcal{W}^{\mathbf{i}}_{t^\lambda}\right) \geq 0 \right\}$ written in a tropical version of the ideal coordinates for $\mathbf{i}$.

Using work of Judd (\cite{Judd2018}) we have that there is a unique `positive’ critical point of $\mathcal{W}_{t^\lambda}$ over the field of Puiseux series.
Its coordinates, in terms of the ideal coordinates for $\mathbf{i}$, are positive in the sense of having positive leading term.
The remarkable property of our new polytopes relates to the tropical version of this critical point, which, for every choice of $\mathbf{i}$, gives a point in $\mathbb{R}^{R_+}$ that lies in the interior of the polytope $\mathcal{P}^{\mathbf{i}}_{\lambda}$.
We prove that this tropical critical point is independent of the reduced expression $\mathbf{i}$, and that it is given by a pattern called the ‘ideal filling’ for $\lambda$ that was introduced by Judd.

Finally, combining these results with work of Rietsch (\cite{Rietsch2008}) relating critical points of the superpotential with Toeplitz matrices, we show that for a totally positive lower-triangular Toeplitz matrix over the field of Puiseux series factorised into simple root subgroups, the valuations of the factors give an ideal filling.
\end{abstract}

\tableofcontents
\listoffigures
\listoftables

\section{Introduction}

Representations of Lie groups are often described in terms of their weights - the characters arising in the action of a maximal torus. A standard way to depict these weights is by embedding the character lattice into a real vector space and viewing the weights as lattice points in their convex hull, the so-called `weight polytope’ of the representation. For an irreducible representation the weights along the boundary of the weight polytope (including the highest weight) all have one-dimensional weight spaces. The weight spaces corresponding to interior points can be higher-dimensional. Accordingly, a better `picture’ of the representation may be given by a higher-dimensional polytope that projects onto the weight polytope, such that the lattice points in a fibre parametrise a basis of the corresponding weight space.

A famous example of such a construction is given by the Gelfand-Tsetlin polytope of a representation of $GL_n$, given first in \cite{GelfandTsetlin1950} (e.g. Figure \ref{fig A Gelfand-Tsetlin polytope}).
More recent examples relate to Lusztig's canonical basis and its combinatorial and geometric construction (\cite{Luszig1990}, \cite{Lusztig1990_2}), as well as Kashiwara's crystal basis operators (\cite{Kashiwara1991}). Of particular interest on the crystal basis side are the string polytopes introduced by Littelmann in \cite{Littelmann1998}. On the canonical basis side there is another parametrisation due to Lusztig (\cite{Lusztig1994}). His parametrisation uses coordinate charts on the Langlands dual flag variety and the notion of tropicalisation that he introduced.
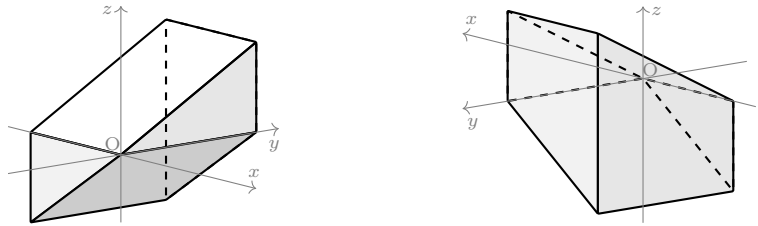
\begin{figure}[ht]
\centering
\begin{tikzpicture}[scale=0.6]
    \filldraw[fill=black!10, thick, rounded corners=0.5] (0,0) -- (3,2.5) -- (3,0.5) -- cycle; % $x=0$
    \filldraw[fill=black!5, thick, rounded corners=0.5] (0,0) -- (-2,-1.5) -- (-2,0.5) -- cycle; % $y=0$
    \filldraw[fill=black!20, thick, rounded corners=0.5] (0,0) -- (3,0.5) -- (1,-1) -- (-2,-1.5) -- cycle; % $x=z$
    \draw[thick, dashed, rounded corners=0.5] (3,0.5) -- (3,2.5) -- (1,3) -- (1,-1) -- cycle; % $y=1$
    \draw[thick, rounded corners=0.5] (0,0) -- (-2,0.5) -- (1,3) -- (3,2.5) -- cycle; % $y=z$
    \draw[black!50, ->] (-2.5,0.625) -- (3,-0.75); % x axis
        \node[black!50] at (2.95,-0.4) {\scriptsize{$x$}};
    \draw[black!50, ->] (-2.5,-0.417) -- (3.5,0.583); % y axis
        \node[black!50] at (3.4,0.2) {\scriptsize{$y$}};
    \draw[black!50, ->] (0,-1.5) -- (0,3.3); % z axis
        \node[black!50] at (-0.3,3.2) {\scriptsize{$z$}};
        \node[black!50] at (-0.18, 0.25) {\scriptsize{O}};
\end{tikzpicture}
\hspace{2cm}
\begin{tikzpicture}[scale=0.6]
    \filldraw[fill=black!5, thick, rounded corners=0.5] (0,-2) -- (0,2) -- (-2,2.5) -- (-2,0.5) -- cycle; % $y=1$
    \filldraw[fill=black!10, thick, rounded corners=0.5] (0,-2) -- (0,2) -- (3,0.5) -- (3,-1.5) -- cycle; % $x=-1$
    \draw[thick, dashed, rounded corners=0.5] (3,0.5) -- (1,1) -- (3,-1.5) -- cycle; % $y=0$
    \draw[thick, dashed, rounded corners=0.5] (1,1) -- (-2,0.5) -- (-2,2.5) -- cycle; % $x=0$
    \draw[black!50, ->] (3.5,0.375) -- (-3,2); % x axis
        \node[black!50] at (-2.8,2.3) {\scriptsize{$x$}};
    \draw[black!50, ->] (3.3,1.38) -- (-3,0.33); % y axis
        \node[black!50] at (-2.77,0) {\scriptsize{$y$}};
    \draw[black!50, ->] (1,-2.2) -- (1,2.6); % z axis
        \node[black!50] at (1.3,2.5) {\scriptsize{$z$}};
        \node[black!50] at (1.16, 1.2) {\scriptsize{O}};
\end{tikzpicture}
\caption{A Gelfand-Tsetlin polytope} \label{fig A Gelfand-Tsetlin polytope}
\end{figure}

Building on the work of Lusztig and Kashiwara, Berenstein and Kazhdan in \cite{BerensteinKazhdan2000}, \cite{BerensteinKazhdan2007}, `geometrised' such polytopes via their theory of geometric crystals.
Their construction includes a function that encodes all of the walls of these polytopes.
This function turns out to agree with the superpotential of a flag variety, which was later independently constructed by Rietsch in the the context of mirror symmetry (\cite{Rietsch2008}, see also Chhaibi's work in \cite{ChhaibiThesis13}).
Another example of this are the polytopes constructed using the mirror symmetry of Grassmannians by Rietsch and Williams in \cite{RietschWilliams2019}, which relate to fundamental representations of $GL_n$ and cluster duality (\cite{FockGoncharov2009}).

Rietsch showed in \cite{Rietsch2008} that the critical points of the superpotential describe the spectrum of the quantum cohomology ring of $G/P$ in a presentation given by Dale Peterson.
The idea to consider the critical points of the superpotential, $\mathcal{W}_q$, was transported to the setting of geometric crystals and polytopes by Judd in \cite{Judd2018}, working in the case of full flag varieties $SL_n/B$. He considered an analogue, $\mathcal{W}_{t^{\lambda}}$, of $\mathcal{W}_q$ over the field of Puiseux series, associated to a dominant weight $\lambda$ of $SL_n$. He showed that there is a unique critical point of $\mathcal{W}_{t^{\lambda}}$ with positive leading term coefficients. By tropicalisation he showed that this special point gives rise to a canonical point in the interior of the string polytope for the representation of $SL_n$ with highest weight $\lambda$. We will refer to this point as the tropical critical point of the superpotential.

As part of Judd's study of these objects, he investigated when the tropical critical point is integral. To do so, he introduced a combinatorial object called an ideal filling for a dominant weight $\lambda$ and formed a close connection to the tropical critical point. Such a ideal filling for $\lambda$ is given by assigning non-negative real numbers $n_{ij}$ to boxes in an upper triangular form such that $n_{ij}=\max\{n_{i+1 \; j}, n_{i \; j-1} \} $ for $j-i\geq 2$, together with a constraint on the $n_{ij}$ determined by $\lambda$. For example if $n=4$ the arrangement looks like Figure \ref{fig Ideal filling arrangement for n=4}.
\begin{figure}[ht!]
\centering
\begin{tikzpicture}
    %box lines
    %horizontal
    \draw (0,0) -- (2.4,0);
    \draw (0,-0.8) -- (2.4,-0.8);
    \draw (0.8,-1.6) -- (2.4,-1.6);
    \draw (1.6,-2.4) -- (2.4,-2.4);

    %vertical
    \draw (0,0) -- (0,-0.8);
    \draw (0.8,0) -- (0.8,-1.6);
    \draw (1.6,0) -- (1.6,-2.4);
    \draw (2.4,0) -- (2.4,-2.4);

    %nij labels
    \node at (0.4,-0.4) {\small{$n_{12}$}};

    \node at (1.2,-0.4) {\small{$n_{13}$}};
    \node at (1.2,-1.2) {\small{$n_{23}$}};

    \node at (2,-0.4) {\small{$n_{14}$}};
    \node at (2,-1.2) {\small{$n_{24}$}};
    \node at (2,-2) {\small{$n_{34}$}};
\end{tikzpicture}
\caption{Ideal filling arrangement for $n=4$} \label{fig Ideal filling arrangement for n=4}
\end{figure}

Since the ideal filling for $\lambda$ is a combinatorial reformulation of the tropical critical point, we should expect it to also have a polytope in some special toric chart. Unfortunately, building a polytope in these `ideal filling coordinates' was not part of Judd's work. This poses the following question: can we find a toric chart for which the tropical critical point is exactly the ideal filling for $\lambda$? In answer, the main result of our work is the construction of such a coordinate system.

Initially our `ideal filling coordinates' are indexed by a specific choice of reduced expression for $w_0$, the longest element of the Weyl group. Namely our coordinates correspond to factorising a lower-triangular unipotent matrix into simple root subgroups according to the reduced expression $s_{i_1} \cdots s_{i_N}$ given by
    $$\mathbf{i}_0 = (i_1, \ldots, i_N) :=
    (1,2, \ldots, n-1, 1,2 \ldots, n-2, \ldots, 1, 2, 1)
    $$
where we refer to Sections \ref{subsec Notation and definitions} and \ref{sec The ideal coordinates} for the precise definitions. However we later generalise our `ideal coordinates' to arbitrary reduced expressions. Remarkably this gives rise to family of polytopes that all have the same tropical critical point (see Proposition \ref{prop trop crit pt independent of i}).

As a corollary we have the following theorem which gives an interpretation of ideal fillings using Toeplitz matrices over the field of Puiseux series:
    $$\mathcal{K}= \bigcup_{n\in \mathbb{Z}_{>0}} \mathbb{C}((t^{\frac{1}{n}}))
    $$
\begin{thm} \label{thm Toeplitz mat and ideal fillings}
Let $\phi_i: SL_2 \to GL_n$ be the homomorphism corresponding to the $i$-th simple root of $GL_n$ and take
    $$\mathbf{y}_i(z) = \phi_i\begin{pmatrix} 1 & 0 \\ z & 1 \end{pmatrix}.
    $$
Let $\mathbf{i}=(i_1, \ldots, i_N)$ stand for an arbitrary reduced expression $s_{i_1} \cdots s_{i_N}$ for $w_0$, where $N=\binom{n}{2}$. Then we have an ordering on the set of positive roots $R_+=\{\alpha_{ij}=\epsilon_i-\epsilon_j \ | \ 1\leq i<j \leq n \}$ given by
    $$\alpha^{\mathbf{i}}_j=\begin{cases}
    \alpha_{i_1, i_1+1} & \text{for } j=1, \\
    s_{i_1}\cdots s_{i_{j-1}}\alpha_{i_j, i_j+1} & \text{for } j=2, \ldots N.
    \end{cases}
    $$
Now take $m_{\alpha^{\mathbf{i}}_1},\ldots, m_{\alpha^{\mathbf{i}}_N}$ to be Puiseux series with positive leading coefficients and non-negative valuations $\mu_{\alpha}=\mathrm{Val}_{\mathbf{K}}(m_{\alpha})$ (defined in Section \ref{subsec The basics of tropicalisation} as the exponent of the first non-zero term). If the product
$\mathbf{y}_{i_1}\left(m_{\alpha^{\mathbf{i}}_1}^{-1}\right)\cdots  \mathbf{y}_{i_N}\left(m_{\alpha^{\mathbf{i}}_N}^{-1}\right)$
is a Toeplitz matrix, then the valuations $\mu_{\alpha}$ form an ideal filling:
\begin{center}
    \begin{tikzpicture}[scale=1.35]
        %box lines
        %horizontal
        \draw (0,0) -- (1.7,0);
        \draw (2.3,0) -- (3.2,0);
        \draw (0,-0.8) -- (1.7,-0.8);
        \draw (2.3,-0.8) -- (3.2,-0.8);
        \draw (0.8,-1.6) -- (1.7,-1.6);
        \draw (2.3,-1.6) -- (3.2,-1.6);
        \draw (2.3,-2.4) -- (3.2,-2.4);
        \draw (2.4,-3.2) -- (3.2,-3.2);

        %vertical
        \draw (0,0) -- (0,-0.8);
        \draw (0.8,0) -- (0.8,-1.6);
        \draw (1.6,0) -- (1.6,-1.7);
        \draw (2.4,0) -- (2.4,-1.7);
        \draw (2.4,-2.3) -- (2.4,-3.2);
        \draw (3.2,0) -- (3.2,-1.7);
        \draw (3.2,-2.3) -- (3.2,-3.2);

        %nij labels
        \node at (0.42,-0.4) {\scriptsize{$\mu_{\alpha_{12}}$}};
        \node at (1.22,-0.4) {\scriptsize{$\mu_{\alpha_{13}}$}};
        \node at (2.82,-0.4) {\scriptsize{$\mu_{\alpha_{1n}}$}};
        \node at (2.02,-0.4) {\scriptsize{$\cdots$}};

        \node at (1.22,-1.2) {\scriptsize{$\mu_{\alpha_{23}}$}};
        \node at (2.02,-1.2) {\scriptsize{$\cdots$}};
        \node at (2.82,-1.2) {\scriptsize{$\mu_{\alpha_{2n}}$}};

        \node at (2.82,-1.9) {\scriptsize{$\vdots$}};
        \node at (2,-1.9) {\scriptsize{$\ddots$}};

        \node at (2.82,-2.8) {\scriptsize{$\mu_{\alpha_{n-1,n}}$}};
    \end{tikzpicture}
\end{center}
Moreover every ideal filling with rational entries arises in this way.
\end{thm}

This result follows by combining the theorem of Rietsch that the critical points of the superpotential are given by Toeplitz matrices \cite[non-T-equivariant case of Theorem 4.1]{Rietsch2008}, with our result that the factorisations of the positive critical point lead to ideal fillings (combination of Proposition \ref{prop GLn version of Jamie's 6.2}, Corollary \ref{cor nu'_i coords ideal filling for lambda} and Proposition \ref{prop trop crit pt independent of i}).
Additionally, the result extends from $\mathbb{Q}$-valued to $\mathbb{R}$-valued ideal fillings if we instead work over the field of generalised Puiseux series (defined in Section \ref{subsec The basics of tropicalisation}).

Our work is structured as follows: in Section \ref{sec Mirror symmetry for G/B applied to representation theory} we introduce the mirror to the flag variety $G/B$ and present the first of two coordinate systems. We do this since the key to constructing a polytope using the superpotential, is to express the superpotential in some torus coordinate chart. It will later transpire that this first system gives rise to string polytopes. The second coordinate system, the `ideal' coordinates, will be best suited to the tropical critical point. We begin Section \ref{sec The ideal coordinates} by defining this system for the reduced expression given by $\mathbf{i}_0$ and devote much of the section to constructing the ideal coordinates from the string coordinates.

Then, following Judd (\cite{Judd2018}), in Section \ref{sec Givental-type quivers and critical points} we introduce quivers as defined by Givental in \cite{Givental1997}. These will provide a framework within which to neatly hold the information of the toric charts, highest weight and superpotential. They will also give a helpful description of the critical point conditions for a given highest weight.

Finally, in Section \ref{sec The tropical viewpoint} we consider everything we have developed up until this point through the lens of tropical geometry. This is where we will discuss the polytopes mentioned above and prove, for a given highest weight $\lambda$, that the ideal filling and tropical critical point coincide.
We conclude by generalising the ideal coordinates to arbitrary reduced expressions and presenting our new family of polytopes, as well as the conjecture mentioned above.

\begin{ackno}
I wish to thank my supervisor Konni Rietsch for suggesting this idea to me, and particularly for her helpful explanations and patience throughout.
\end{ackno}

\section[Mirror symmetry for G/B applied to representation theory]{Mirror symmetry for $G/B$ applied to representation theory} \label{sec Mirror symmetry for G/B applied to representation theory}

\subsection{Notation and definitions} \label{subsec Notation and definitions}

Let $\mathbb{K}$ be a field of characteristic $0$, containing a positive semifield. Unless otherwise stated we take $G=GL_n(\mathbb{K})$ with $B$, $B_-$ the Borel subgroups of upper and lower triangular matrices. Let $U$, $U_-$ be their respective unipotent radicals, that is the subgroups of upper and lower triangular matrices with all diagonal entries equal to $1$, and let $T=B \cap B_-$ be the diagonal matrices in $G$. The Langlands dual group to $G$ is denoted $G^{\vee}=GL_n(\mathbb{K})$, and may be taken together with the corresponding subgroups $B^{\vee}$, $B^{\vee}_-$, $U^{\vee}$, $U^{\vee}_-$ and $T^{\vee}$ in $G^{\vee}$.

For $i=1,\ldots , n$, we write $\epsilon_i$, $\epsilon^{\vee}_i$ for the standard characters and cocharacters of $T$. Let $X^*(T) = \mathrm{Hom}(T,\mathbb{K}^*)$, $X_*(T)= \mathrm{Hom}(\mathbb{K}^*, T)$ the respective character and cocharacter lattices, dually paired in the standard way by
    \begin{equation*} \label{eqn bilinear form pair lattices}
    \langle \ , \ \rangle : X^*(T) \times X_*(T) \to \mathrm{Hom}(\mathbb{K}^*, \mathbb{K}^*) \cong \mathbb{Z}.
    \end{equation*}
We take $\alpha_{ij} = \epsilon_i - \epsilon_j \in X^*(T)$, $\alpha_{ij}^{\vee} = \epsilon^{\vee}_i - \epsilon^{\vee}_j \in X^*(T^{\vee})= X_*(T)$. Additionally for each $i \in I=\{1, \ldots, n-1\}$ we write $\alpha_i=\alpha_{i, i+1}$, $\alpha_i^{\vee}=\alpha_{i, i+1}^{\vee}$. Then the roots and positive roots of $G$ are
    $$R = \{\alpha_{ij} \ \vert \ i \neq j\} \text{ and } R_+ = \{\alpha_{ij} \ \vert \ i < j\}
    $$
respectively and the simple roots of $G$ are $\{ \alpha_i \ \vert \ i \in I\}$. The Cartan matrix is $A=(a_{ij})$ defined by $a_{ij}=\langle \alpha_j, \alpha_i^{\vee}\rangle$.

The fundamental weights of $G$ are given by $\omega_i = \epsilon_1 + \cdots + \epsilon_i$. Additionally we denote the set of dominant integral weights by
    $${X^*(T)}^+=\{\lambda \in X^*(T) \ \vert \ \langle \lambda, \alpha_{ij}^{\vee} \rangle \geq 0 \ \forall i <j\}.
    $$
For $\lambda \in {X^*(T)}^+$, let $V_{\lambda}$ denote the irreducible representation with highest weight $\lambda$.

Note that we may identify $X^*(T)\otimes \mathbb{R} = X_*(T^{\vee})\otimes \mathbb{R}$ with the dual, $\mathfrak{h}^*_{\mathbb{R}}$, of the Lie algebra of the split real torus of $G$.

The Weyl group of $G$ is the symmetric group, $W=N_G(T)/T = S_n$, generated by the simple reflections $s_i$ for $i \in I$. It acts on $X^*(T)$ by permuting the roots and we denote the action of $w \in W$ on $\alpha \in X^*(T)$ by $w\alpha$.

Associated to each simple root $\alpha_i$ there is a homomorphism $\phi_i: SL_2 \to G$, explicitly
    $$\phi_i : \begin{pmatrix} a & b \\ c & d \end{pmatrix} \mapsto
    \begin{pmatrix} 1 & & & & & \\ & \ddots & & & & \\ & & a & b & & \\ & & c & d & & \\ & & & & \ddots & \\ & & & & & 1\end{pmatrix}
    \quad \text{with } a \text{ in position } (i,i),
    $$
and we have a number of $1$-parameter subgroups of $G$ respectively defined by
    $$\begin{aligned}
    & \mathbf{x}_i(z) = \phi_i\begin{pmatrix} 1 & z \\ 0 & 1 \end{pmatrix}, \quad &&
        \mathbf{y}_i(z) = \phi_i\begin{pmatrix} 1 & 0 \\ z & 1 \end{pmatrix}, \\
    & \mathbf{x}_{-i}(z) = \phi_i\begin{pmatrix} z^{-1} & 0 \\ 1 & z \end{pmatrix}, \quad &&
        \mathbf{t}_i(t) = \phi_i\begin{pmatrix} t & 0 \\ 0 & t^{-1} \end{pmatrix}.
    \end{aligned}$$
for $z \in \mathbb{K}$, $t \in \mathbb{K}^*$ and $i \in I$.
The simple reflections in the Weyl group, $s_i \in W$, are given explicitly by $s_i=\bar{s}_i T$ where
    $$\bar{s}_i = \mathbf{x}_i(-1)\mathbf{y}_i(1)\mathbf{x}_i(-1) = \phi_i\begin{pmatrix} 0 & -1 \\ 1 & 0 \end{pmatrix}.
    $$
More generally we may write each $w \in W$ as a product with a minimal number of factors, $w=s_{i_1}\cdots s_{i_m}$. We get a representative of $w$ in $N_G(T)$ by taking $\bar{w}=\bar{s}_{i_1}\cdots \bar{s}_{i_m}$. Here $m$ is called the length of $w$, denoted $l(w)$, and the choice of expression $s_{i_1}\cdots s_{i_m}$ is said to be reduced. In particular it is well known that $\bar{w}$ is independent of this choice. For ease of notation we will often let $\mathbf{i}=(i_1, \ldots, i_m)$ stand for the reduced expression ${s}_{i_1}\cdots {s}_{i_m}$.

In a similar way for $G^{\vee}$ we have, for each $i \in I$, a homomorphism $\phi_i^{\vee}: PGL_2 \to G^{\vee}$ and $\mathbf{x}^{\vee}_i(z)$, $\mathbf{y}^{\vee}_i(z)$, $\mathbf{x}^{\vee}_{-i}(z)$, $\mathbf{t}^{\vee}_i(t)$ defined analogously. The Weyl group of $G^{\vee}$ is again the symmetric group and we use the same notation as above.

With this in mind, we make an observation which will be used frequently; given a reduced expression, say $s_{i_1}\cdots s_{i_m}$, we can construct matrices in $G^{\vee}$ which are indexed by $\mathbf{i}=(i_1,\ldots,i_m)$. We do this by taking products of the matrices defined above. An explicit example is given by the following map:
    $$\mathbf{x}^{\vee}_{\mathbf{i}} : (\mathbb{K}^*)^N \to U^{\vee} \cap B_-^{\vee} \bar{w}_0 B_-^{\vee}, \quad (z_1, \ldots, z_N)\mapsto \mathbf{x}^{\vee}_{i_1}(z_1) \cdots \mathbf{x}^{\vee}_{i_N}(z_N).
     $$

\subsection{Landau-Ginzburg models} \label{subsec Landau-Ginzburg models}

The mirror to the flag variety $G/B$ is a pair $(Z,\mathcal{W})$, called a Landau-Ginzburg model, where $Z \subset G^{\vee}$ is an affine variety and $\mathcal{W}:Z\to \mathbb{K}^*$ is a holomorphic function called the superpotential. In order to give a more precise description we first recall Bruhat decomposition, namely that $G$ may be written as a disjoint union of Bruhat cells $B \bar{w} B$ (see \cite[Theorem 8.3.8]{SpringerLAG}):
    $$G= \bigsqcup_{w \in W} B \bar{w} B.
    $$
Similarly we may write $G/B$ as
    $$G/B= \bigsqcup_{w \in W} B \bar{w} B/B \quad \text{with} \quad \mathrm{dim}(B \bar{w} B/B)= l(w).
    $$
We note that the cells $B \bar{w} B$ do not depend on the choice of representative $\bar{w}$.

These Bruhat cells give rise to a partial ordering of Weyl group elements, known as the Bruhat order (see \cite[Theorem 8.5.4]{SpringerLAG}); for $v,w \in W$ we say $v\leq w$ if $B \bar{v} B \subseteq \overline{B \bar{w} B}$. With respect to this ordering there is a unique maximal element $w_0 \in W$ and we set $N=l(w_0)$.

Additionally, we use the Bruhat order to define open Richardson varieties. These are given by intersecting opposite Bruhat cells; for $v,w \in W$ such that $v\leq w$ we have
    $$\mathcal{R}_{v,w}:=(B_- \bar{v}B \cap B \bar{w} B)/B \subset G/B.
    $$
It is well known that $\mathcal{R}_{v,w}$ is smooth, irreducible and has dimension $l(w)-l(v)$, \cite{KazhdanLustig1980}. On the dual side we have
    $$\mathcal{R}^{\vee}_{v,w}:= (B^{\vee}_- \bar{v}B^{\vee} \cap B^{\vee} \bar{w} B^{\vee})/B^{\vee} \subset G^{\vee}/B^{\vee}.
    $$

We now return to  $(Z, \mathcal{W})$, the Landau-Ginzburg model for $G/B$, and define the subvariety
    $$Z:= B_-^{\vee}\cap B^{\vee}\bar{w}_0B^{\vee}\subset G^{\vee}.
    $$
In order to define the superpotential $\mathcal{W}$, we let $\chi:U^{\vee} \to \mathbb{K}$ be the sum of above-diagonal elements
    $$\chi(u) := \sum_{i=1}^{n-1} u_{i\, i+1}, \quad u=(u_{ij})\in U^{\vee}.
    $$
Then the superpotential is given by
    $$\mathcal{W}:Z \to \mathbb{K}^*, \quad u_1d\bar{w}_0u_2 \mapsto \chi(u_1)+ \chi(u_2)
    $$
where $u_1, u_2 \in U^{\vee}$ and $d \in T^{\vee}$.
This map will appear frequently in subsequent sections.

The motivation for introducing the Landau-Ginzburg model is to study the representation theory of $G$ using the mirror to $G/B$. It is natural then to equip $Z$ with highest weight and weight maps. The highest weight map recovers the original torus factor, $d$, as follows:
    $$\mathrm{hw}:Z \to T^{\vee}, \quad u_1d\bar{w}_0u_2 \mapsto d.
    $$
For the weight map we first note that each element $b \in Z$ may be written as $b=[b]_-[b]_0$ with $[b]_- \in U_-^{\vee}$, $[b]_0\in T^{\vee}$. Then the weight map is given by the projection
    $$
    \mathrm{wt}:Z \to T^{\vee}, \quad b=u_1d\bar{w}_0u_2 \mapsto [b]_0.
    $$
We will often write the above decomposition of $b$ as $b=[b]_-t_R$ to remind us that the torus factor is taken on the right.

\subsection{The string coordinates} \label{subsec The string coordinates}

In order to make the connection with representation theory we restrict our attention to various toric charts $T^{\vee} \times (\mathbb{K}^*)^N \to Z$, indexed by reduced expressions for $w_0$. The first chart we want to consider is useful for reconstructing the string polytope via the superpotential.
These `string coordinates', which we introduce in this section, were used by Judd \cite{Judd2018} following Chhaibi \cite{ChhaibiThesis13}, who was in turn inspired by the work of Berenstein and Kazhdan \cite{BerensteinKazhdan2000}, \cite{BerensteinKazhdan2007}.

The toric chart in question is defined by the composition of a number of maps which we write here for overview and then define in detail:
\begin{center}
\begin{tikzcd}[column sep=1.25cm]
    T^{\vee} \times (\mathbb{K}^*)^N \arrow[r, " {\left(id, \, \mathbf{x}_{-\mathbf{i}}^{\vee}\right)} "] & T^{\vee}\times (B_-^{\vee} \cap U^{\vee} \bar{w}_0 U^{\vee}) \arrow[r, " \tau "] & T^{\vee}\times (U^{\vee} \cap B_-^{\vee} \bar{w}_0 B_-^{\vee}) \arrow[r, " \Phi "] & Z \\
\end{tikzcd}
\end{center}

\vspace{-0.7cm}
If we let $\mathbf{i}=(i_1, \ldots, i_N)$ stand for a reduced expression $s_{i_1} \cdots s_{i_N}$ for $w_0$, then first map constructs a matrix indexed by $\mathbf{i}$ as follows:
    $$\mathbf{x}_{-\mathbf{i}}^{\vee} : (\mathbb{K}^*)^N \to B_-^{\vee}\cap U^{\vee}\bar{w}_0 U^{\vee}\,, \quad (z_1, \ldots, z_N)\mapsto \mathbf{x}_{-i_1}^{\vee}(z_1) \cdots \mathbf{x}_{-i_N}^{\vee}(z_N).
    $$

The second map, $\tau$, may be written as the composition of a twist map $\eta^{w_0,e}$ and an involution $\iota$.  We present $\tau$ in this way since the involution will be helpful later. The twist map is defined to be
    $$\eta^{w_0,e}:B_-^{\vee}\cap U^{\vee}\bar{w}_0 U^{\vee} \to U^{\vee} \cap B_-^{\vee} \bar{w}_0 B_-^{\vee}\,, \quad b \mapsto [(\bar{w}_0b^T)^{-1}]_+.
    $$
Here $b^T$ is the transpose of $b$ and $[g]_+$ is given by decomposing $g=[g]_- [g]_0 [g]_+ $ such that $[g]_- \in U_-^{\vee}$, $[g]_0\in T^{\vee} $, $[g]_+ \in U^{\vee}$. The involution is given by
    $$\iota:G^{\vee} \to G^{\vee}, \quad g\mapsto (\bar{w}_0g^{-1}\bar{w}_0^{-1})^T.
    $$
We note that this map preserves $U^{\vee}$. It remains to define $\tau$ by applying the composition $\iota \circ \eta^{w_0,e}$ to the second factor:
    $$\tau : T^{\vee}\times (B_-^{\vee} \cap U^{\vee} \bar{w}_0 U^{\vee}) \to T^{\vee}\times (U^{\vee} \cap B_-^{\vee} \bar{w}_0 B_-^{\vee}), \ (d,u) \mapsto \left(d,\iota(\eta^{w_0,e}(u))\right).
    $$

The final map in the definition of the string toric chart is an isomorphism which allows us to factorise elements of $Z$:
    $$\Phi :T^{\vee}\times (U^{\vee} \cap B_-^{\vee} \bar{w}_0 B_-^{\vee}) \to Z, \quad (d,u_1) \mapsto u_1 d \bar{w}_0 u_2.$$
Here $u_2 \in U^{\vee}$ is the unique element such that $u_1 d\bar{w}_0 u_2 \in Z$. Then, the string toric chart on $Z$ (corresponding to $\mathbf{i}$) is defined to be
    \begin{equation} \label{eqn string coord chart defn}
    \varphi_{\mathbf{i}}:T^{\vee} \times (\mathbb{K}^*)^N \to Z, \quad \varphi_{\mathbf{i}}(d,\boldsymbol{z}) = \Phi\circ \tau\left(d, \mathbf{x}_{-\mathbf{i}}^{\vee}(\boldsymbol{z})\right) = \Phi\left(d, \iota \left(\eta^{w_0,e}\left(\mathbf{x}_{-\mathbf{i}}^{\vee}(\boldsymbol{z})\right)\right)\right).
    \end{equation}

Later in this work we will need the specific toric chart corresponding to
    $$\mathbf{i}_0 = (i_1, \ldots, i_N) := (1,2, \ldots, n-1, 1,2 \ldots, n-2, \ldots, 1, 2, 1)
    $$
so, unless otherwise stated, from now on we will take
    $$
    u:=\mathbf{x}_{-\mathbf{i}_0}^{\vee}(\boldsymbol{z}), \quad u_1 := \iota(\eta^{w_0,e}(u)), \quad b:= u_1 d \bar{w}_0 u_2.
    $$
With this notation the composition of maps defining $\varphi_{\mathbf{i}_0}$, the string toric chart corresponding to $\mathbf{i}_0$, may be visualised as follows:
\vspace{-0.3cm} \begin{center}
\begin{tikzcd}[row sep=0.3em, column sep=1.35cm]
    T^{\vee} \times (\mathbb{K}^*)^N \arrow[r, " {\left(id, \, \mathbf{x}_{-\mathbf{i}_0}^{\vee}\right)} "] & T^{\vee}\times (B_-^{\vee} \cap U^{\vee} \bar{w}_0 U^{\vee}) \arrow[r, " \tau "] & T^{\vee}\times (U^{\vee} \cap B_-^{\vee} \bar{w}_0 B_-^{\vee}) \arrow[r, " \Phi "] & Z \\
    (d,\boldsymbol{z}) \arrow[r, mapsto] & (d,u) \arrow[r, mapsto] & (d,u_1) \arrow[r, mapsto] & b \\
\end{tikzcd}
\end{center}

\begin{ex}[Dimension 3] \label{ex dim 3 z coords} The reduced expression is $\mathbf{i}_0=(1,2,1)$.
We will start with $\left(d, (z_1, z_2, z_3 )\right)$ and apply the sequence of maps defined above.

Applying $\mathbf{x}_{-\mathbf{i}_0}^{\vee}$ to $(z_1,z_2,z_3)$ gives the matrix $u$:
    $$u=\mathbf{x}_{-\mathbf{i}_0}^{\vee}(z_1, z_2, z_3) = \begin{pmatrix}
        \frac{1}{z_1z_3} & & \\ \frac{1}{z_3}+\frac{z_1}{z_2} & \frac{z_1z_3}{z_2} & \\ 1 & z_3 & z_2
    \end{pmatrix}, \quad
    \bar{w}_0=\begin{pmatrix} & & 1 \\ & -1 & \\ 1 & & \end{pmatrix}.
    $$
We recall the definition of second map: $\tau\left(d, (z_1, z_2, z_3 )\right) = \left(d, \iota \circ \eta^{w_0,e}(u)\right)$. To see this in action we first evaluate $\eta^{w_0,e}(u)$ and then apply $\iota$ to obtain the matrix $u_1$:
    $$\begin{aligned}
    \eta^{w_0,e}(u)
        &= \left[ \left( \bar{w}_0 \begin{pmatrix} \frac{1}{z_1z_3} & \frac{z_1}{z_2}+\frac{1}{z_3} & 1 \\ & \frac{z_1z_3}{z_2} & z_3 \\ & & z_2  \end{pmatrix} \right)^{-1} \right]_+
        = \left[ \begin{pmatrix} 1 & z_1 +\frac{z_2}{z_3} & z_1z_3 \\ -\frac{1}{z_1} & -\frac{z_2}{z_1z_3} & \\ \frac{1}{z_2} & & \end{pmatrix} \right]_+ \\
        &= \begin{pmatrix} 1 & z_1 +\frac{z_2}{z_3} & z_1z_3 \\ 0 & 1 & z_3 \\ 0 & 0 & 1 \end{pmatrix}
    \end{aligned}
    $$
    $$u_1=\iota\left(\eta^{w_0,e}(\mathbf{x}_{-\mathbf{i}_0}^{\vee}(\boldsymbol{z}))\right) = \left(\bar{w}_0 \begin{pmatrix} 1 & z_1 +\frac{z_2}{z_3} & z_1z_3 \\ 0 & 1 & z_3 \\ 0 & 0 & 1 \end{pmatrix}^{-1} \bar{w}_0^{-1}\right)^T
        = \begin{pmatrix} 1 & z_3 & z_2 \\ & 1 & z_1 +\frac{z_2}{z_3} \\ & & 1\end{pmatrix}
    $$
Of note, we can factorise this matrix using a different reduced expression; $\mathbf{i}'_0=(2,1,2)$. We obtain
    \begin{equation*}
    \tau\left(d, \mathbf{x}_{-\mathbf{i}_0}^{\vee}\left(z_1, z_2, z_3 \right)\right)
    = \left(d, u_1\right) = \left(d, \mathbf{x}_{\mathbf{i}'_0}^{\vee}\left(z_1, z_3, \frac{z_2}{z_3}\right)\right).
    \end{equation*}

It remains to apply the map $\Phi$. We recall that $u_2 \in U^{\vee}$ will be the unique element such that $b=u_1 d\bar{w}_0 u_2 \in Z$.
$$\begin{aligned}
    b&= \Phi\left(d, u_1\right) = u_1 d \bar{w}_0u_2 \\
    &= \begin{pmatrix} 1 & z_3 & z_2 \\ & 1 & z_1 +\frac{z_2}{z_3} \\ & & 1\end{pmatrix}
        \begin{pmatrix} d_1& &  \\ & d_2 & \\ & & d_3 \end{pmatrix}
        \begin{pmatrix} & & 1 \\ & -1 & \\ 1 & & \end{pmatrix}
        \begin{pmatrix} 1 & \frac{d_2}{d_3} \frac{z_3}{z_2} & \frac{d_1}{d_3}\frac{1}{z_1z_3} \\ & 1 & \frac{d_1}{d_2}\left(\frac{1}{z_3}+\frac{z_2}{z_1z_3^2}\right) \\ & & 1\end{pmatrix} \\
    &=\begin{pmatrix} d_3 z_2 & & \\ d_3 \left(z_1+\frac{z_2}{z_3} \right) & d_2 \frac{z_1z_3}{z_2} & \\ d_3 & d_2 \frac{z_3}{z_2} & d_1 \frac{1}{z_1z_3} \end{pmatrix} \\
    \end{aligned}$$
From our construction of $b$ we see that the superpotential is
    $$\mathcal{W}(d,\boldsymbol{z}) = z_1 +\frac{z_2}{z_3} + z_3 + \frac{d_1}{d_2}\left(\frac{1}{z_3} + \frac{z_2}{z_1z_3^2}\right) + \frac{d_2}{d_3}\frac{z_3}{z_2}
    $$
and the weight matrix is
    \begin{equation} \label{eqn dim 3 wt matrix z coords}
    \mathrm{wt}(b)=\begin{pmatrix} d_3z_2 & & \\ & d_2\frac{z_1z_3}{z_2} & \\ & & d_1\frac{1}{z_1z_3} \end{pmatrix}.
    \end{equation}
\end{ex}

\subsection{The form of the weight matrix} \label{subsec The form of the weight matrix}

In this section we generalise the formula for the weight matrix $\mathrm{wt}(b)=t_R$ in terms of the string coordinates, $(d, \boldsymbol{z})$, as in the above example, Equation (\ref{eqn dim 3 wt matrix z coords}). We begin with the two factorisations of $b$ that we have already seen:
    $$b=u_1d \bar{w}_0 u_2 = [b]_- t_R \quad \text{where} \quad [b]_- \in U^{\vee}_-.
    $$
Now recalling the involution $\iota$, we see this acts on elements $b\in Z$ as
    $$\iota(b)=\iota(u_1) \bar{w}_0 d^{-1} \iota(u_2).
    $$
Defining $\tilde{u}_i:=\iota(u_i)$ and $\tilde{b}:=\iota(b)$ gives $\tilde{b}=\tilde{u}_1 \bar{w}_0 d^{-1} \tilde{u}_2$. Moreover we can write
    $$\tilde{b}= \left[\tilde{b}\right]_-\tilde{t}_R \quad \text{where} \quad \left[\tilde{b}\right]_- \in U^{\vee}_-, \ \tilde{t}_R \in T^{\vee}.
    $$
Then
    $$\tilde{b}=\iota(b) = \left(\bar{w}_0 b^{-1}\bar{w}_0^{-1}\right)^T
        = \left(\bar{w}_0 \left([b]_- t_R \right)^{-1}\bar{w}_0^{-1}\right)^T
        = \left(\bar{w}_0[b]_-^{-1}\bar{w}_0^{-1}\right)^T \bar{w}_0t_R^{-1}\bar{w}_0^{-1}.
    $$
Thus $\tilde{t}_R = \bar{w}_0t_R^{-1}\bar{w}_0^{-1}$ or equivalently, $t_R = \bar{w}_0\tilde{t}_R^{-1}\bar{w}_0^{-1}$.

\begin{prop} \label{prop form of t_R using d,u}
Let $N=\binom{n}{2}$ and recall the definition $u:=\mathbf{x}_{-\mathbf{i}_0}^{\vee}(\boldsymbol{z})$ where
    $$\mathbf{i}_0 = (i_1, \ldots, i_N) := (1,2, \ldots, n-1, 1,2 \ldots, n-2, \ldots, 1, 2, 1).
    $$
Then $t_R = \bar{w}_0d[u]_0\bar{w}_0^{-1}$.
\end{prop}

The proof of Proposition \ref{prop form of t_R using d,u} will require the following lemma:
\begin{lem} \label{lem [(bar w_0u^T)^-1]_0=I}
Let $\mathbf{i}=(i_1, \ldots, i_N)$ stand for a reduced expression $s_{i_1} \cdots s_{i_N}$ for $w_0$ and take $\boldsymbol{z}\in \mathbb{K}^N$. If $A= \mathbf{x}_{-i_1}^{\vee}(z_1) \cdots \mathbf{x}_{-i_N}^{\vee}(z_N)$ and $I$ denotes the identity matrix, then
    $$\left[\left(\bar{w}_0A^T\right)^{-1}\right]_0=I.$$
\end{lem}

\begin{proof}[Proof of Proposition \ref{prop form of t_R using d,u}]

Since $t_R = \bar{w}_0\tilde{t}_R^{-1}\bar{w}_0^{-1}$ we will in fact prove that $\tilde{t}_R= d^{-1}[u]_0^{-1}$.
Recalling the definition
        $$u_1 := \iota(\eta^{w_0,e}(u))
        $$
we see that
        $$\tilde{u}_1 := \iota(u_1)=\eta^{w_0,e}(u)=\left[\left(\bar{w}_0u^T\right)^{-1}\right]_+.$$
We have
    $$\left(u^T\right)^{-1}  = \left(\bar{w}_0u^T\right)^{-1}\bar{w}_0= \left[\left(\bar{w}_0 u^T\right)^{-1}\right]_- \left[\left(\bar{w}_0u^T\right)^{-1}\right]_0 \tilde{u}_1 \bar{w}_0.
    $$
In particular by Lemma \ref{lem [(bar w_0u^T)^-1]_0=I} we see that
    $$\left(u^T\right)^{-1} \in U^{\vee}_- \tilde{u}_1\bar{w}_0 \quad
    \Rightarrow \quad u^{-1} \in \bar{w}_0^{-1} \tilde{u}_1^T U^{\vee}_+.
    $$

Thus if $\lambda\in {X^*(T)}^+$ is a dominant integral weight and we denote a corresponding highest weight vector by $v^+_{\lambda} \in V_{\lambda}$, then
    $$ u^{-1} \cdot v^+_{\lambda} = \bar{w}_0^{-1} \tilde{u}_1^T \cdot v^+_{\lambda}.
    $$
This expression allows for two computations of the coefficient of the highest weight vector in $u^{-1} \cdot v^+_{\lambda}$. Firstly, since $u=[u]_0[u]_-$ we have
    $$\left\langle u^{-1}\cdot v^+_{\lambda}, v^+_{\lambda} \right\rangle = \lambda\left([u]_0^{-1}\right).
    $$
Secondly, we rewrite $\bar{w}_0^{-1} \tilde{u}_1^T$ using $\tilde{b}=\tilde{u}_1 \bar{w}_0 d^{-1} \tilde{u}_2=\left[\tilde{b}\right]_-\tilde{t}_R$;
    $$\bar{w}_0^{-1} \tilde{u}_1^T = \bar{w}_0^{-1} \left(\tilde{b}\tilde{u}_2^{-1}d\bar{w}_0^{-1}\right)^T = \bar{w}_0^{-1} \bar{w}_0 d^T \left(\tilde{u}_2^{-1}\right)^T \tilde{b}^T = d \left(\tilde{u}_2^{-1}\right)^T \tilde{t}_R \left[\tilde{b}\right]_-^T.
    $$
Then we see that the result follows from the second computation;
    $$\begin{aligned} \lambda\left([u]_0^{-1}\right)=\left\langle u^{-1}\cdot v^+_{\lambda}, v^+_{\lambda} \right\rangle =
    \left\langle \bar{w}_0^{-1}\tilde{u}_1^T\cdot v^+_{\lambda}, \ v^+_{\lambda} \right\rangle
        &=\big\langle d \left(\tilde{u}_2^{-1}\right)^T \tilde{t}_R \left[\tilde{b}\right]_-^T  \cdot v^+_{\lambda}, \ v^+_{\lambda} \big\rangle \\
        &=\lambda(\tilde{t}_R) \lambda(d) \big\langle \left(\tilde{u}_2^{-1}\right)^T \left[\tilde{b}\right]_-^T\cdot v^+_{\lambda}, \ v^+_{\lambda} \big\rangle =\lambda(\tilde{t}_R) \lambda(d).
    \end{aligned}
    $$
\end{proof}

\begin{proof}[Proof of Lemma \ref{lem [(bar w_0u^T)^-1]_0=I}]
Taking the decomposition
    $$\left(\bar{w}_0A^T\right)^{-1} = \left[\left(\bar{w}_0 A^T\right)^{-1}\right]_- \left[\left(\bar{w}_0A^T\right)^{-1}\right]_0 \left[\left(\bar{w}_0 A^T\right)^{-1}\right]_+
    $$
and defining
    $$U:=\left[\left(\bar{w}_0 A^T\right)^{-1}\right]_+^{-1} \in U^{\vee}_+ \ , \quad D:= \left[\left(\bar{w}_0A^T\right)^{-1}\right]_0^{-1} \in T^{\vee} \ , \quad L:= \left[\left(\bar{w}_0 A^T\right)^{-1}\right]_-^{-1} \in U^{\vee}_-
    $$
we see that
    $$
    \left(\bar{w}_0A^T\right)^{-1} = L^{-1} D^{-1} U^{-1} \quad \Rightarrow \quad A^T = \bar{w}_0^{-1} UDL.
    $$
We will show that $D=I$.

We recall
    $$\mathbf{x}_{-i}^{\vee}(z)^T = \phi^{\vee}_i\begin{pmatrix} z^{-1} & 1 \\ 0 & z \end{pmatrix} \ , \quad \bar{s}_i^{-1} = \phi_i\begin{pmatrix} 0 & 1 \\ -1 & 0 \end{pmatrix}.
    $$

For any fundamental representation $V_{\omega_i}=\bigwedge^i\mathbb{C}^n$, applying $\bar{w}_0^{-1}$ to a lowest weight vector $v^-_{\omega_i} = v_{n-i+1}\wedge \cdots \wedge v_n $ gives the corresponding highest weight vector $v^+_{\omega_i} = v_1\wedge \cdots \wedge v_i$, so
    $$\left\langle \bar{w}_0^{-1} \cdot v^-_{\omega_i} , \ v^+_{\omega_i} \right\rangle = 1.
    $$
Similarly applying $A^T=\mathbf{x}_{-i_N}^{\vee}(z_N)^T \cdots \mathbf{x}_{-i_1}^{\vee}(z_1)^T $ gives
    $$\left\langle A^T \cdot v^-_{\omega_i} , \ v^+_{\omega_i} \right\rangle = \left\langle \mathbf{x}_{-i_N}^{\vee}(z_N)^T \cdots \mathbf{x}_{-i_1}^{\vee}(z_1)^T \cdot v^-_{\omega_i} , \ v^+_{\omega_i} \right\rangle = 1.
    $$

We may also evaluate the coefficient of the highest weight vector in $A^T \cdot v^-_{\lambda}$ using the expression $A^T = \bar{w}_0^{-1} UDL$;
    $$\begin{aligned}
    \left\langle A^T \cdot v^-_{\omega_i} , \ v^+_{\omega_i} \right\rangle
    = \left\langle \bar{w}_0^{-1} UDL \cdot v^-_{\omega_i} , \ v^+_{\omega_i} \right\rangle
    = \left\langle \bar{w}_0^{-1} UD \cdot v^-_{\omega_i} , \ v^+_{\omega_i} \right\rangle
        &= \omega_i(D) \left\langle \bar{w}_0^{-1} U \cdot v^-_{\omega_i} , \ v^+_{\omega_i} \right\rangle \\
        &= \omega_i(D) \left\langle \bar{w}_0^{-1} \cdot v^-_{\omega_i} , \ v^+_{\omega_i} \right\rangle = \omega_i(D).
    \end{aligned}
    $$
Thus $\omega_i(D)=1$ for all $i$, so $D=I$.
\end{proof}

Now that we better understand the weight matrix $\mathrm{wt}(b) = t_R$, we complete this section by expressing it directly in terms of $(d,\boldsymbol{z})$.

\begin{cor} \label{cor wt matrix in z string coords}
The weight matrix $t_R$, given in terms of the string coordinates $(d,\boldsymbol{z})$, is the diagonal matrix with entries
\begin{equation} \label{eqn general formula wt matrix z coords}
    \left( t_R \right)_{n-j+1, n-j+1} =  \frac{d_j \prod_{\substack{1\leq m \leq N \\ i_m=j-1}} z_m}{\prod_{\substack{1\leq m \leq N \\ i_m=j}} z_m}, \quad j=1, \ldots, n.
    \end{equation}
\end{cor}

\begin{proof}
We recall the matrix $u = \mathbf{x}_{-i_1}^{\vee}(z_1) \cdots \mathbf{x}_{-i_N}^{\vee}(z_N)$ where
    $$\mathbf{i}_0 = (i_1, \ldots, i_N) := (1,2, \ldots, n-1, 1,2 \ldots, n-2, \ldots, 1, 2, 1)
    $$
and
    $$\mathbf{x}_{-i}^{\vee}(z) = \phi_i\begin{pmatrix} z^{-1} & 0 \\ 1 & z \end{pmatrix}.
    $$
We see that
    \begin{align}
    \left( [u]_0 \right)_{jk}
    & = \begin{cases}
        \prod_{m=1}^{N} \left( \mathbf{x}_{-i_m}^{\vee}(z_m) \right)_{jj} & \text{if} \ j=k\\
        0 & \text{if} \ j\neq k
    \end{cases} \nonumber \\
    & = \begin{cases}
        \frac{\prod_{\substack{1\leq m \leq N \\ i_m=j-1}} z_m}{\prod_{\substack{1\leq m \leq N \\ i_m=j}} z_m} & \text{if} \ j=k\\
        0 & \text{if} \ j\neq k \label{eqn wt matrix z coords}
    \end{cases}
    \end{align}
noting that $\prod_{\substack{1\leq m \leq N \\ i_m=0}} z_m = 1 = \prod_{\substack{1\leq m \leq N \\ i_m=n}} z_m$.
Thus since $t_R = \bar{w}_0d[u]_0\bar{w}_0^{-1}$ we see that it is the diagonal matrix with entries
    $$\left( t_R \right)_{n-j+1, n-j+1} =  \frac{d_j \prod_{\substack{1\leq m \leq N \\ i_m=j-1}} z_m}{\prod_{\substack{1\leq m \leq N \\ i_m=j}} z_m}, \quad j=1, \ldots, n.
    $$
\end{proof}

\section{The ideal coordinates} \label{sec The ideal coordinates}

Our preferred coordinate system on $Z$, which we will call the ideal coordinate system, is far more natural than the string toric chart since it is easier to define. We begin by recalling the reduced expression
    $$\mathbf{i}_0 = (i_1, \ldots, i_N) := (1,2, \ldots, n-1, 1,2 \ldots, n-2, \ldots, 1, 2, 1)$$
and consider the map
    \begin{equation*}
    \left(\mathbb{K}^*\right)^N \times T^{\vee} \to Z\,, \quad \left((m_1, \ldots, m_N), t_R\right) \mapsto \mathbf{y}_{i_1}^{\vee}\left(\frac{1}{m_1}\right) \cdots \mathbf{y}_{i_N}^{\vee}\left(\frac{1}{m_N}\right) t_R.
    \end{equation*}
We recall that $Z$ has two projections to $T^{\vee}$, given by the highest weight and weight maps. In the previous coordinate system the highest weight map was obvious due to the form of $b=u_1d\bar{w}_0 u_2$, whereas the weight required more effort to compute. In this new system the weight is much more straightforward.

We could consider this system with coordinates $(\boldsymbol{m},t_R)$, but instead, for ease of later application, we wish to work with coordinates $(d, \boldsymbol{m})$ which we call the ideal coordinates\footnote{
    The choice to work with the inverted coordinates $\frac{1}{m_i}$ is motivated by the main theorem of Section \ref{sec Givental-type quivers and critical points}, Theorem \ref{thm crit points, sum at vertex is nu_i}, which also gives rise to the name for this coordinate system.
}:
    $$\psi : T^{\vee} \times \left(\mathbb{K}^*\right)^N \to Z\,, \quad \left(d, (m_1, \ldots, m_N)\right) \mapsto \mathbf{y}_{i_1}^{\vee}\left(\frac{1}{m_1}\right) \cdots \mathbf{y}_{i_N}^{\vee}\left(\frac{1}{m_N}\right) t_R(d,\boldsymbol{m}).
    $$
Here $t_R(d,\boldsymbol{m})$ is the weight matrix given now in terms of the coordinates $(d, \boldsymbol{m})$. An explicit description of the map
    $$T^{\vee} \times \left(\mathbb{K}^*\right)^N \to T^{\vee} \,, \quad \left(d,\boldsymbol{m}\right) \mapsto t_R(d,\boldsymbol{m})
    $$
will be given shortly.

First, we present the main theorem of this section. For a given highest weight matrix, $d$, this theorem describes the coordinate change from the string coordinates to the ideal coordinates, allowing us to move freely between the two systems:
\begin{thm} \label{thm coord change}
To change from the string coordinates, $(d,\boldsymbol{z})$, to the ideal coordinates, $(d,\boldsymbol{m})$, we first let
    $$s_k := \sum_{j=1}^{k-1}(n-j)
    $$
then for $k=1,\ldots, n-1$, $a=1,\ldots, n-k$ we have
    $$m_{s_k+a} = \begin{cases}
            z_{1+s_{n-a}} &\text{if } k=1, \\
            \frac{z_{k+s_{n-k-a+1}}}{z_{k-1+s_{n-k-a+1}}} &\text{otherwise}.
            \end{cases}
    $$
\end{thm}

The string and ideal coordinate systems are related by repeated application of a theorem known as the Chamber Ansatz, which we will discuss in Section \ref{subsec The Chamber Ansatz}. We will then develop our understanding of the relation between the string and ideal coordinates in Sections \ref{subsec Chamber Ansatz minors} and \ref{subsec The coordinate change}, culminating in the proof of Theorem \ref{thm coord change}. To complete the current section we present an example followed by a further application of this theorem.

\begin{ex} \label{ex b and maps in ideal coords dim 3}
In dimension 3 the the coordinate change is
    $$m_1=z_3, \quad m_2 = z_1, \quad m_3= \frac{z_2}{z_1}   \quad \qquad z_1=m_2, \quad z_2=m_2m_3, \quad z_3=m_1.
    $$
In Example \ref{ex dim 3 z coords} we saw that the matrix $b$ was given by
    $$b= \Phi\left(d, \mathbf{x}_{\mathbf{i}'_0}^{\vee}\left(z_1, z_3, \frac{z_2}{z_3}\right)\right) = \begin{pmatrix} d_3 z_2 & & \\ d_3 \left(z_1+\frac{z_2}{z_3} \right) & d_2 \frac{z_1z_3}{z_2} & \\ d_3 & d_2 \frac{z_3}{z_2} & d_1 \frac{1}{z_1z_3} \end{pmatrix}.
    $$
With the new coordinates we have
    $$b =\begin{pmatrix} d_3 m_2m_3 & & \\ d_3 \left(m_2+\frac{m_2m_3}{m_1} \right) & d_2 \frac{m_1}{m_3} & \\ d_3  & d_2 \frac{m_1}{m_2m_3} & d_1 \frac{1}{m_1m_2} \end{pmatrix}
    = \mathbf{y}_{\mathbf{i}_0}^{\vee}\left(\frac{1}{m_1}, \frac{1}{m_2}, \frac{1}{m_3} \right)
        \begin{pmatrix} d_3 m_2m_3 & & \\ & d_2 \frac{m_1}{m_3} & \\ & & d_1 \frac{1}{m_1m_2} \end{pmatrix}.
    $$
Additionally the superpotential is now given by
    $$\mathcal{W}(d,\boldsymbol{m}) = m_1 + m_2 + \frac{m_2m_3}{m_1} + \frac{d_2}{d_3} \frac{m_1}{m_2m_3} + \frac{d_1}{d_2} \left(\frac{m_3}{{m_1}^2}+\frac{1}{m_1}\right).
    $$
We again also give the weight matrix:
    $$\mathrm{wt}(b) = \begin{pmatrix} d_3 m_2m_3 & & \\ & d_2 \frac{m_1}{m_3} & \\ & & d_1 \frac{1}{m_1m_2} \end{pmatrix}.
    $$
\end{ex}

We now present an application of Theorem \ref{thm coord change} on the weight matrix, namely we describe $\mathrm{wt}(b)=t_R$ in terms of the ideal coordinates.

\begin{cor} \label{cor wt matrix in m ideal coords}
The weight matrix $t_R$, given in terms of the ideal coordinates $(d, \boldsymbol{m})$, is the diagonal matrix with entries
\begin{equation} \label{eqn general formula wt matrix m coords}
    \left( t_R \right)_{n-j+1, n-j+1} =  \frac{d_j \prod\limits_{k=1, \ldots,j-1} m_{s_k+(j-k)}}{\prod\limits_{r=1, \ldots, n-j} m_{s_j+r}}, \quad j=1, \ldots, n, \ \text{with} \ r=k-j.
    \end{equation}
\end{cor}

\begin{proof}
Recalling Corollary \ref{cor wt matrix in z string coords}, we see that we need to show
    \begin{equation} \label{eqn t_R quotients of z and nu'}
    \frac{\prod_{\substack{1\leq m \leq N \\ i_m=j-1}} z_m}{\prod_{\substack{1\leq m \leq N \\ i_m=j}} z_m}
        = \frac{\prod\limits_{k=1, \ldots,j-1} m_{s_k+(j-k)}}{\prod\limits_{r=1, \ldots, n-j} m_{s_j+r}}, \quad \text{where} \ r=k-j.
    \end{equation}

The denominator of the left hand side of (\ref{eqn t_R quotients of z and nu'}) is
    $$\prod_{\substack{1\leq m \leq N \\ i_m=j}} z_m = \prod_{\substack{m=j+s_r \\ r=1, \ldots, n-j}} z_m = \prod_{r=1, \ldots, n-j} z_{j+s_{n-j-r+1}}.
    $$
Similarly the numerator is
    $$\begin{aligned}
    \prod_{\substack{1\leq m \leq N \\ i_m=j-1}} z_m
        &= \begin{cases}
            1 & \text{if } j=1 \\
            \prod\limits_{\substack{m=j-1+s_r \\ r=1, \ldots, n-j+1}} z_m & \text{otherwise}
        \end{cases} \\
        &= \begin{cases}
            1 & \text{if } j=1 \\
            \prod\limits_{r=0, \ldots, n-j} z_{j-1+s_{n-j-r+1}} & \text{otherwise}.
        \end{cases}
    \end{aligned}
    $$
Consequently if $j=1$ the left hand side of (\ref{eqn t_R quotients of z and nu'}) becomes
    $$\frac{\prod_{\substack{1\leq m \leq N \\ i_m=j-1}} z_m}{\prod_{\substack{1\leq m \leq N \\ i_m=j}} z_m} = \frac{1}{\prod\limits_{r=1, \ldots, n-j} z_{j+s_{n-j-r+1}}} = \frac{1}{\prod\limits_{r=1, \ldots, n-j} m_{s_j+r}}
    $$
as desired.

If $j\geq 2$ then the left hand side of (\ref{eqn t_R quotients of z and nu'}) is
    $$\frac{\prod_{\substack{1\leq m \leq N \\ i_m=j-1}} z_m}{\prod_{\substack{1\leq m \leq N \\ i_m=j}} z_m} = \frac{\prod\limits_{r=0, \ldots, n-j} z_{j-1+s_{n-j-r+1}}}{\prod\limits_{r=1, \ldots, n-j} z_{j+s_{n-j-r+1}}} \\
        = z_{j-1+s_{n-j+1}} \prod_{r=1, \ldots, n-j} \frac{1}{m_{s_j+r}}.
    $$
It remains to show that
    $$\prod\limits_{k=1, \ldots,j-1} m_{s_k+(j-k)} = z_{j-1+s_{n-j+1}}.
    $$
Indeed from the coordinate change formula in Theorem \ref{thm coord change}, for $j\geq 2$ we see
    $$m_{s_k+(j-k)} = \frac{z_{k+s_{n-k-(j-k)+1}}}{z_{k-1+s_{n-k-(j-k)+1}}} = \frac{z_{k+s_{n-j+1}}}{z_{k-1+s_{n-j+1}}}.
    $$
So the product becomes telescopic and, as desired, we obtain
    $$\prod\limits_{k=1, \ldots,j-1} m_{s_k+(j-k)} = z_{1+s_{n-j+1}}\prod\limits_{k=2, \ldots,j-1} \frac{z_{k+s_{n-j+1}}}{z_{k-1+s_{n-j+1}}}
    = z_{j-1+s_{n-j+1}}.
    $$
\end{proof}

\subsection{The Chamber Ansatz} \label{subsec The Chamber Ansatz}

In order to prove Theorem \ref{thm coord change} we require a sequence of lemmas, the first of which (Lemma \ref{lem form of u_1 and b factorisations}) makes use of the afore mentioned Chamber Ansatz. In this section we introduce the Chamber Ansatz and then state and prove Lemma \ref{lem form of u_1 and b factorisations}.

\begin{defn}
Let $J=\{j_1 < \cdots < j_l\}\subseteq [1,n]$ and $K=\{k_1 < \cdots < k_l\}\subseteq [1,n]$. The pair $(J,K)$ is called admissible if $j_s\leq k_s$ for $s=1, \ldots, l$.
\end{defn}

For such an admissible pair $(J,K)$, we denote by $\Delta^{J}_K$ the $l \times l$ minor with row and column sets defined by $J$ and $K$ respectively.

We now state a specific case of the Generalised Chamber Ansatz presented by Marsh and Rietsch in \cite[Theorem 7.1]{MarshRietsch2004}:
\begin{thm}[Chamber Ansatz]
Let $B=z\bar{w}_0 B^{\vee}_+\in \mathcal{R}^{\vee}_{e,w_0}$, where $z \in U^{\vee}_+$. Let $\mathbf{w}=(w_{(0)}, w_{(1)}, \ldots, w_{(n)})$ be a sequence of partial products for $w_0$ defined by its sequence of factors
    $$\left(w_{(1)}, w_{(1)}^{-1}w_{(2)}, \ldots, w_{(N-1)}^{-1}w_{(N)}\right) = (s_{i_1}, \ldots, s_{i_N}).
    $$
Then there is an element
    $$g= \mathbf{y}_{i_1}(t_1) \mathbf{y}_{i_2}(t_2)\cdots \mathbf{y}_{i_N}(t_N) \in U^{\vee}_- \cap B^{\vee}_+ \bar{w}_0 B^{\vee}_+$$
such that $B = g B^{\vee}_+$. Moreover for $k=1, \ldots, N$ we have
    $$t_k = \frac{\prod_{j\neq i_k} \Delta^{\omega^{\vee}_j}_{w_{(k)}\omega^{\vee}_j}(z)^{-a_{j,i_k}}}{\Delta^{\omega^{\vee}_{i_k}}_{w_{(k)}\omega^{\vee}_{i_k}}(z) \Delta^{\omega^{\vee}_{i_k}}_{w_{(k-1)}\omega^{\vee}_{i_k}}(z)}.
    $$
\end{thm}

Each $\Delta^{\omega^{\vee}_{i_k}}_{w_{(k)}\omega^{\vee}_{i_k}}$, where $\omega^{\vee}_{i_k}$ ranges through the set of fundamental weights, is called a (standard) chamber minor. As above, it is given by the $i_k\times i_k$ minor with $\omega^{\vee}_{i_k}$ encoding the row set and $w_{(k)}\omega^{\vee}_{i_k}$ encoding the column set. We note that these row and column sets form admissible pairs.

Much of the information in the Chamber Ansatz may be read from an associated pseudoline arrangement; it may be viewed as a singular braid diagram and is called an ansatz arrangement. In dimension $n$, for the case we are considering, the ansatz arrangement consists of $n$ pseudolines. These are numbered from bottom to top on the left side of the arrangement.

Each factor $g_k = \mathbf{y}_{i_k}^{\vee}(t_k)$ of $g$ gives rise to a crossing of the pseudolines at level $i_k$. We label each chamber in the diagram with the labels of the strands passing below it and associate to the chamber with label $S$ the flag minor $\Delta^{\left[1,|S|\right]}_{S}$.

If $A_k$, $B_k$, $C_k$ and $D_k$ are the minors corresponding to the chambers surrounding the $k$-th singular point, with $A_k$ and $D_k$ above and below it and $B_k$ and $C_k$ to the left and right, then the Chamber Ansatz gives
    $$
\begin{tikzpicture}[baseline=0.43cm]
    %dots and stars
    \node at (1,0.5) {$\bullet$};

    \draw (0,1) -- (0.75,1) -- (1.25,0) -- (2,0);
    \draw (0,0) -- (0.75,0) -- (1.25,1) -- (2,1);
    %\labels
    \node at (1,-0.25) {\scriptsize{$D_k$}};
    \node at (0.3,0.5) {\scriptsize{$B_k$}};
    \node at (1.7,0.5) {\scriptsize{$C_k$}};
    \node at (1,1.25) {\scriptsize{$A_k$}};
\end{tikzpicture}
    \qquad \quad
    t_k = \frac{A_k D_k}{B_k C_k}.
    $$

Let $N=\binom{n}{2}$ and again take
    $$\begin{gathered}
    u:=\mathbf{x}_{-\mathbf{i}_0}^{\vee}(\boldsymbol{z}), \quad u_1 := \iota(\eta^{w_0,e}(u)), \quad b:= u_1 d \bar{w}_0 u_2, \\
    \mathbf{i}_0 = (i_1, \ldots, i_N) := (1,2, \ldots, n-1, 1,2 \ldots, n-2, \ldots, 1, 2, 1).
    \end{gathered}
    $$
We will also need
    $$\mathbf{i}'_0 = (i'_1, \ldots, i'_N) := (n-1, n-2, \ldots, 1, n-1, n-2, \ldots, 2, \ldots, n-1, n-2, n-1)
    $$
and we use the superscript `$\mathrm{op}$' to denote taking such an expression in reverse, for example
    $$ {\mathbf{i}'_0}^{\mathrm{op}}:=(i'_N, \ldots, i'_1) = (n-1, n-2, n-1, \ldots, n-3, n-2, n-1).
    $$

\begin{ex}[Ansatz arrangements for {$\mathbf{i}_0$}, {${\mathbf{i}'_0}^{\mathrm{op}}$} in dimension {$4$}] Since $n=4$, we have $N=6$. For $\mathbf{i}_0 = (1,2,3,1,2,1)$, the sequence of partial products for $w_0$ is given by
    $$\mathbf{w} = (w_{(0)}, w_{(1)}, \ldots, w_{(6)}) = (e, s_1, s_1s_2, s_1s_2s_3, s_1s_2s_3s_1, s_1s_2s_3s_1s_2, s_1s_2s_3s_1s_2s_1).
    $$
The ansatz arrangement for $\mathbf{i}_0$ is given in Figure \ref{fig ansatz arrangement i_0 dim 4}.

For ${\mathbf{i}'_0}^{\mathrm{op}} = (3,2,3,1,2,3)$, the sequence of partial products for $w_0$ is
    $$\mathbf{w} = (w_{(0)}, w_{(1)}, \ldots, w_{(6)}) = (e, s_3, s_3s_2, s_3s_2s_3, s_3s_2s_3s_1, s_3s_2s_3s_1s_2, s_3s_2s_3s_1s_2s_3).
    $$
The ansatz arrangement for ${\mathbf{i}'_0}^{\mathrm{op}}$ is given in Figure \ref{fig ansatz arrangement i'_0^op dim 4}.

\begin{figure}[ht!]
\centering
\begin{minipage}[b]{0.47\textwidth}
    \centering
\begin{tikzpicture}[scale=0.85]
    % pseudolines
    \draw (0,4) -- (2.75,4) -- (3.25,3) -- (4.75,3) -- (5.25,2) -- (5.75,2) -- (6.25,1) -- (7,1);
    \draw (0,3) -- (1.75,3) -- (2.25,2) -- (3.75,2) -- (4.25,1) -- (5.75,1) -- (6.25,2) -- (7,2);
    \draw (0,2) -- (0.75,2) -- (1.25,1) -- (3.75,1) -- (4.25,2) -- (4.75,2) -- (5.25,3) -- (7,3);
    \draw (0,1) -- (0.75,1) -- (1.25,2) -- (1.75,2) -- (2.25,3) -- (2.75,3) -- (3.25,4) -- (7,4);

    % dots and stars
    \node at (1,1.5) {$\bullet$};
    \node at (2,2.5) {$\bullet$};
    \node at (3,3.5) {$\bullet$};
    \node at (4,1.5) {$\bullet$};
    \node at (5,2.5) {$\bullet$};
    \node at (6,1.5) {$\bullet$};

    % pseudoline labels
    \node at (-0.3,4) {$4$};
    \node at (-0.3,3) {$3$};
    \node at (-0.3,2) {$2$};
    \node at (-0.3,1) {$1$};

    \node at (7.3,4) {$1$};
    \node at (7.3,3) {$2$};
    \node at (7.3,2) {$3$};
    \node at (7.3,1) {$4$};

    % chamber labels
    \node at (1.5,3.5) {$123$};
    \node at (5,3.5) {$234$};

    \node at (1,2.5) {$12$};
    \node at (3.5,2.5) {$23$};
    \node at (6,2.5) {$34$};

    \node at (0.5,1.5) {$1$};
    \node at (2.5,1.5) {$2$};
    \node at (5,1.5) {$3$};
    \node at (6.5,1.5) {$4$};
\end{tikzpicture}
\caption{The ansatz arrangement for $\mathbf{i}_0$ in dimension 4} \label{fig ansatz arrangement i_0 dim 4}
\end{minipage}
    \hfill
\begin{minipage}[b]{0.47\textwidth}
    \centering
\begin{tikzpicture}[scale=0.85]
    % pseudolines
    \draw (0,4) -- (0.75,4) -- (1.25,3) -- (1.75,3) -- (2.25,2) -- (3.75,2) -- (4.25,1) -- (7,1);
    \draw (0,3) -- (0.75,3) -- (1.25,4) -- (2.75,4) -- (3.25,3) -- (4.75,3) -- (5.25,2) -- (7,2);
    \draw (0,2) -- (1.75,2) -- (2.25,3) -- (2.75,3) -- (3.25,4) -- (5.75,4) -- (6.25,3) -- (7,3);
    \draw (0,1) -- (3.75,1) -- (4.25,2) -- (4.75,2) -- (5.25,3) -- (5.75,3) -- (6.25,4) -- (7,4);

    % dots and stars
    \node at (1,3.5) {$\bullet$};
    \node at (2,2.5) {$\bullet$};
    \node at (3,3.5) {$\bullet$};
    \node at (4,1.5) {$\bullet$};
    \node at (5,2.5) {$\bullet$};
    \node at (6,3.5) {$\bullet$};

    % pseudoline labels
    \node at (-0.3,4) {$4$};
    \node at (-0.3,3) {$3$};
    \node at (-0.3,2) {$2$};
    \node at (-0.3,1) {$1$};

    \node at (7.3,4) {$1$};
    \node at (7.3,3) {$2$};
    \node at (7.3,2) {$3$};
    \node at (7.3,1) {$4$};

    % chamber labels
    \node at (0.5,3.5) {$123$};
    \node at (2,3.5) {$124$};
    \node at (4.5,3.5) {$134$};
    \node at (6.5,3.5) {$234$};

    \node at (1,2.5) {$12$};
    \node at (3.5,2.5) {$14$};
    \node at (6,2.5) {$34$};

    \node at (2,1.5) {$1$};
    \node at (5.5,1.5) {$4$};
\end{tikzpicture}
\caption{The ansatz arrangement for ${\mathbf{i}'_0}^{\mathrm{op}}$ in dimension 4} \label{fig ansatz arrangement i'_0^op dim 4}
\end{minipage}
\end{figure}

\end{ex}

We are now ready to state the first lemma needed for the proof of Theorem \ref{thm coord change}:

\begin{lem} \label{lem form of u_1 and b factorisations}
We can factorise $u_1$ and $b$ as follows:
        \begin{align}
        u_1 &= \mathbf{x}_{i'_1}^{\vee}(p_1) \cdots \mathbf{x}_{i'_N}^{\vee}(p_N) \label{eqn form of u_1}, \\
        b &=\mathbf{y}_{i_1}^{\vee}\left(\frac{1}{m_1}\right) \cdots \mathbf{y}_{i_N}^{\vee}\left(\frac{1}{m_N}\right) t_R. \label{eqn form of b}
        \end{align}
The $p_i$ and $m_i$ are given by the Chamber Ansatz in terms of chamber minors:
    \begin{align}
    p_{N-k+1} &= \frac{\prod_{j\neq i'_{N-k+1}} \Delta^{\omega^{\vee}_j}_{w_{(k)}\omega^{\vee}_j}(u^T)^{-a_{j,i'_{N-k+1}}}}{\Delta^{\omega^{\vee}_{i'_{N-k+1}}}_{w_{(k)}\omega^{\vee}_{i'_{N-k+1}}}(u^T) \Delta^{\omega^{\vee}_{i'_{N-k+1}}}_{w_{(k-1)}\omega^{\vee}_{i'_{N-k+1}}}(u^T)}, \quad k=1, \ldots, N, \label{eqn p_i in terms of chamber minors} \\
    \frac{1}{m_k} &= \frac{\prod_{j\neq i_k} \Delta^{\omega^{\vee}_j}_{w_{(k)}\omega^{\vee}_j}(u_1)^{-a_{j,i_k}}}{\Delta^{\omega^{\vee}_{i_k}}_{w_{(k)}\omega^{\vee}_{i_k}}(u_1) \Delta^{\omega^{\vee}_{i_k}}_{w_{(k-1)}\omega^{\vee}_{i_k}}(u_1)}, \quad k=1, \ldots, N. \label{eqn 1/m_i in terms of chamber minors}
    \end{align}
\end{lem}

\begin{proof}
We will use the Chamber Ansatz to prove that $u_1$ and $b$ have the factorisations given in (\ref{eqn form of u_1}) and (\ref{eqn form of b}) respectively.

To show (\ref{eqn form of u_1}) we first note that if $x= \mathbf{x}_{\mathbf{i}}^{\vee}(z_1, \ldots, z_N)$ then $x^T= \mathbf{y}_{\mathbf{i}^{\mathrm{op}}}^{\vee}(z_N, \ldots, z_1)$. Thus to apply the Chamber Ansatz we need a matrix $A \in U^{\vee}_+$ such that $u_1^T B^{\vee}_+ =  A \bar{w}_0 B^{\vee}_+$. We will extract this matrix from the definition of $u_1$:
    $$u_1 = \iota(\eta^{w_0,e}(u)) = \iota\left([(\bar{w}_0u^T)^{-1}]_+\right).
    $$
After applying the involution $\iota$ we see that
    $$B^{\vee}_-\left(\bar{w}_0u^T\right)^{-1} = B^{\vee}_- \eta^{w_0,e}(u) = B^{\vee}_- \iota(u_1) = B^{\vee}_- \left(\bar{w}_0u_1^{-1}\bar{w}_0^{-1}\right)^T = B^{\vee}_- \left(\bar{w}_0u_1^T\bar{w}_0^{-1} \right)^{-1}.
    $$
Taking the inverse gives
    $$\bar{w}_0u_1^T\bar{w}_0^{-1} B^{\vee}_- = \bar{w}_0 u^T B^{\vee}_- \quad \Rightarrow \quad u_1^T\bar{w}_0^{-1} B^{\vee}_- =  u^T B^{\vee}_-.
    $$
Then using the relation $\bar{w}_0 B^{\vee}_+ \bar{w}_0^{-1} = B^{\vee}_-$ we obtain the desired form:
    $$ u_1^T\bar{w}_0^{-1} \bar{w}_0 B^{\vee}_+ \bar{w}_0^{-1} =  u^T \bar{w}_0 B^{\vee}_+ \bar{w}_0^{-1} \quad \Rightarrow \quad u_1^T B^{\vee}_+ =  u^T \bar{w}_0 B^{\vee}_+.
    $$

We take the reduced expression for $w_0$ defined by ${\mathbf{i}'_0}^{\mathrm{op}}$ and let $\mathbf{w} = (w_{(0)}, w_{(1)}, \ldots, w_{(N)}) $ be the sequence of partial products for $w_0$ given by its sequence of factors
    $$\left(w_{(1)}, w_{(1)}^{-1}w_{(2)}, \ldots, w_{(N-1)}^{-1}w_{(N)}\right) = (s_{i'_N}, \ldots, s_{i'_1}).
    $$
Then by the Chamber Ansatz we have $u_1^T B^{\vee}_+ = u^T \bar{w}_0 B^{\vee}_+ = \mathbf{y}_{i'_N}^{\vee}(p_N) \cdots \mathbf{y}_{i'_1}^{\vee}(p_1) B^{\vee}_+$ with
    $$p_{N-k+1} = \frac{\prod_{j\neq i'_{N-k+1}} \Delta^{\omega^{\vee}_j}_{w_{(k)}\omega^{\vee}_j}(u^T)^{-a_{j,i'_{N-k+1}}}}{\Delta^{\omega^{\vee}_{i'_{N-k+1}}}_{w_{(k)}\omega^{\vee}_{i'_{N-k+1}}}(u^T) \Delta^{\omega^{\vee}_{i'_{N-k+1}}}_{w_{(k-1)}\omega^{\vee}_{i'_{N-k+1}}}(u^T)}, \quad k=1, \ldots, N
    $$
which is exactly in the form of (\ref{eqn p_i in terms of chamber minors}). Moreover since both $u_1^T \in U^{\vee}_- $ and $\mathbf{y}_{i'_N}^{\vee}(t'_1) \cdots \mathbf{y}_{i'_1}^{\vee}(t'_N) \in U^{\vee}_-$ this determines $u_1$ completely; $u_1 = \mathbf{x}_{i'_1}^{\vee}(p_1) \cdots \mathbf{x}_{i'_N}^{\vee}(p_N)$ as desired.

To give the factorisation in (\ref{eqn form of b}) we note that by definition $b B^{\vee}_+ = u_1 \bar{w}_0 B^{\vee}_+$ with $u_1 \in U^{\vee}_+$. For this second application of the Chamber Ansatz we let $w_0$ be described by $\mathbf{i}_0$. We again take $\mathbf{w} = (w_{(0)}, w_{(1)}, \ldots, w_{(N)})$ to be the respective sequence of partial products for $w_0$ defined by its sequence of factors
    $$\left(w_{(1)}, w_{(1)}^{-1}w_{(2)}, \ldots, w_{(N-1)}^{-1}w_{(N)}\right) = (s_{i_1}, \ldots, s_{i_N}).
    $$
Then by the Chamber Ansatz we have $b B^{\vee}_+ = u_1 \bar{w}_0 B^{\vee}_+ = \mathbf{y}_{i_1}^{\vee}\left(\frac{1}{m_1}\right) \cdots \mathbf{y}_{i_N}^{\vee}\left(\frac{1}{m_N}\right)B^{\vee}_+$ with
    $$\frac{1}{m_k} = \frac{\prod_{j\neq i_k} \Delta^{\omega^{\vee}_j}_{w_{(k)}\omega^{\vee}_j}(u_1)^{-a_{j,i_k}}}{\Delta^{\omega^{\vee}_{i_k}}_{w_{(k)}\omega^{\vee}_{i_k}}(u_1) \Delta^{\omega^{\vee}_{i_k}}_{w_{(k-1)}\omega^{\vee}_{i_k}}(u_1)}, \quad k=1, \ldots, N
    $$
which proves (\ref{eqn 1/m_i in terms of chamber minors}).
\end{proof}

\subsection{Chamber Ansatz minors} \label{subsec Chamber Ansatz minors}

We wish to further describe the coordinate changes defined by our two applications of the Chamber Ansatz. In this section we show they are monomial by considering the required minors of $u^T$ and $u_1$. Similar to how the ansatz arrangement tells us which quotients of minors to take when applying Chamber Ansatz, we may use a planar acyclic directed graph to easily compute these minors, and in particular to confirm that they are all monomial (see \cite[Proposition 4.2]{FominZelevinsky1999}, generalising \cite[Theorem 2.4.4]{BerensteinFominZelevinsky1996}). Note that the following description is slightly different to that given by Fomin and Zelevinsky in \cite{FominZelevinsky1999}, since we do not need the same level of generality.

Let $\mathbf{i}= (i_1, \ldots, i_N)$ define some reduced expression for $w_0$ and consider
    $$\mathbf{x}_{\mathbf{i}}^{\vee}(\boldsymbol{z}) = \mathbf{x}_{i_1}^{\vee}(z_1) \cdots \mathbf{x}_{i_N}^{\vee}(z_N).
    $$
For particular choices of admissible pairs $(J,K)$, we wish to compute the %(flag)
minors
    $$\Delta^J_K\left(\mathbf{x}_{\mathbf{i}}^{\vee}(\boldsymbol{z})\right) \text{ or } \ \Delta^J_K\left(\mathbf{x}_{-\mathbf{i}}^{\vee}(\boldsymbol{z})^T\right).
    $$
In the second case it will be helpful to express $\mathbf{x}_{-\mathbf{i}}^{\vee}(\boldsymbol{z})^T$ as a product of matrices $\mathbf{x}_i$ and $\mathbf{t}_i$. To do this we notice
    $$\mathbf{x}_{-i}^{\vee}(z)^T = \phi_i^{\vee}\begin{pmatrix} z^{-1} & 1 \\ 0 & z \end{pmatrix} = \phi_i^{\vee} \left(\begin{pmatrix} z^{-1} & 0 \\ 0 & z \end{pmatrix} \begin{pmatrix} 1 & z \\ 0 & 1 \end{pmatrix} \right)
    = \mathbf{t}_{i}^{\vee}(z^{-1}) \mathbf{x}_{i}^{\vee}(z).
    $$
In particular we may rewrite $\mathbf{x}_{-\mathbf{i}}^{\vee}(\boldsymbol{z})^T $ as follows:
    $$\mathbf{x}_{-\mathbf{i}}^{\vee}(\boldsymbol{z})^T =  \mathbf{x}_{-i_N}^{\vee}(z_N)^T \cdots \mathbf{x}_{-i_1}^{\vee}(z_1)^T
    = \mathbf{t}_{i_N}^{\vee}(z_N^{-1}) \mathbf{x}_{i_N}^{\vee}(z_N) \cdots \mathbf{t}_{i_1}^{\vee}(z_1^{-1}) \mathbf{x}_{i_1}^{\vee}(z_1).
    $$

To construct the graph, $\Gamma$, corresponding to a matrix $x= \mathbf{x}_{\mathbf{i}}^{\vee}(\boldsymbol{z})$ or $\mathbf{x}_{-\mathbf{i}}^{\vee}(\boldsymbol{z})^T$ we begin with $n$ parallel horizontal lines. We add vertices to the ends of each line and number them from bottom to top on both sides. Then for each factor $\mathbf{t}_{i_k}^{\vee}(z_k^{-1})$, $\mathbf{x}_{i_k}^{\vee}(z_k)$ of $x$ we include a labelled line segment and vertices at height $i_k$ defined in Figure \ref{Labelled line segments in graphs for computing Chamber Ansatz minors}.
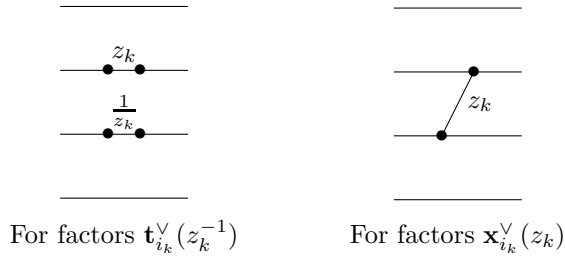
\begin{figure}[hb!]
\centering
\begin{minipage}[b]{0.3\textwidth}
\centering
    \begin{tikzpicture}[scale=0.85]
        \node at (0.75,2) {$\bullet$};
        \node at (1.25,2) {$\bullet$};
        \node at (0.75,1) {$\bullet$};
        \node at (1.25,1) {$\bullet$};
        \draw (0,3) -- (2,3);
        \draw (0,2) -- (2,2);
        \draw (0,1) -- (2,1);
        \draw (0,0) -- (2,0);
        \node at (1,1.35) {$\frac{1}{z_k}$};
        \node at (1,2.25) {$z_k$};
        \node at (1,-0.6) {For factors $\mathbf{t}_{i_k}^{\vee}(z_k^{-1})$};
    \end{tikzpicture}
\end{minipage}
\begin{minipage}[b]{0.3\textwidth}
\centering
    \begin{tikzpicture}[scale=0.85]
        \node at (0.75,1) {$\bullet$};
        \node at (1.25,2) {$\bullet$};
        \draw (0,3) -- (2,3);
        \draw (0,2) -- (2,2);
        \draw (0,1) -- (2,1);
        \draw (0,0) -- (2,0);
        \draw (0.75,1) -- (1.25,2);
        \node at (1.35,1.5) {$z_k$};
        \node at (1,-0.6) {For factors $\mathbf{x}_{i_k}^{\vee}(z_k)$};
    \end{tikzpicture}
\end{minipage}
\caption{Labelled line segments in graphs for computing Chamber Ansatz minors} \label{Labelled line segments in graphs for computing Chamber Ansatz minors}
\end{figure}

Each line segment is viewed as a labelled edge of $\Gamma$, oriented left to right. For an edge $e$, the labelling, called the weight of $e$ and denoted $w(e)$, is given by the diagrams above and taken to be $1$ if left unspecified. The weight $w(\pi)$ of an oriented path $\pi$ is defined to be the product of weights $w(e)$ taken over all edges $e$ in $\pi$.

The set of vertices of the graph $\Gamma$ is given by the endpoints of all line segments. Those vertices appearing as the leftmost (resp. rightmost) endpoints of the horizontal lines are the sources (resp. sinks) of $\Gamma$.

With this notation, \cite[Proposition 4.2]{FominZelevinsky1999} becomes the following:
\begin{thm} \label{thm minors from paths in graph}
For an admissible pair $(J,K)$ of size $l$
    $$ \Delta^J_K\left(\mathbf{x}_{\mathbf{i}'}^{\vee}(\boldsymbol{z})\right) = \sum_{\pi_1,\ldots, \pi_l} w(\pi_1) \cdots w(\pi_l)
    $$
where the sum is taken over all families of l vertex-disjoint paths $\{\pi_1,\ldots, \pi_l\}$ connecting the sources labelled by $J$ with the sinks labelled by $K$.
\end{thm}

To prove that the minors of $u^T$, $u_1$ appearing in our applications of the Chamber Ansatz are monomial we must show that in each case there is only one possible family of paths $\{\pi_1,\ldots, \pi_l\}$. Before showing this we give two examples to clarify the above construction.

\begin{ex}
We take $n=4$, so $N=6$, and we wish to compute the minors $\Delta^J_K\left(u_1\right)$ where
    $$u_1=\mathbf{x}_{\mathbf{i}'_0}^{\vee}\left(z_1, z_4, z_6, \frac{z_2}{z_4}, \frac{z_5}{z_6}, \frac{z_3}{z_5}\right) = \begin{pmatrix}
        1 & z_6 & z_5 & z_3 \\
         & 1 & z_4+\frac{z_5}{z_6} & z_2+\frac{z_3z_4}{z_5}+\frac{z_3}{z_6} \\
         &  & 1 & z_1+\frac{z_2}{z_4}+\frac{z_3}{z_5} \\
         &  &  & 1
    \end{pmatrix}
    $$
with $\mathbf{i}'_0 = (3,2,1,3,2,3)$. The graph for $u_1$ is given in Figure \ref{The graph for u_1 when n=4}.
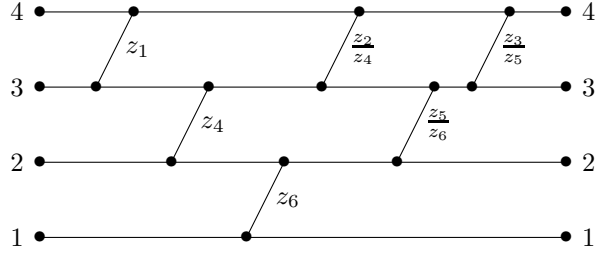
\begin{figure}[ht!]
\centering
\begin{tikzpicture}
    % pseudolines
    \draw (0,4) -- (7,4);
    \draw (0,3) -- (7,3);
    \draw (0,2) -- (7,2);
    \draw (0,1) -- (7,1);

    \draw (0.75,3) -- (1.25,4);
    \draw (1.75,2) -- (2.25,3);
    \draw (2.75,1) -- (3.25,2);
    \draw (3.75,3) -- (4.25,4);
    \draw (4.75,2) -- (5.25,3);
    \draw (5.75,3) -- (6.25,4);
    % dots and stars
    \node at (0,1) {$\bullet$};
    \node at (0,2) {$\bullet$};
    \node at (0,3) {$\bullet$};
    \node at (0,4) {$\bullet$};
    \node at (7,1) {$\bullet$};
    \node at (7,2) {$\bullet$};
    \node at (7,3) {$\bullet$};
    \node at (7,4) {$\bullet$};

    \node at (0.75,3) {$\bullet$};
    \node at (1.25,4) {$\bullet$};
    \node at (1.75,2) {$\bullet$};
    \node at (2.25,3) {$\bullet$};
    \node at (2.75,1) {$\bullet$};
    \node at (3.25,2) {$\bullet$};

    \node at (3.75,3) {$\bullet$};
    \node at (4.25,4) {$\bullet$};
    \node at (4.75,2) {$\bullet$};
    \node at (5.25,3) {$\bullet$};
    \node at (5.75,3) {$\bullet$};
    \node at (6.25,4) {$\bullet$};

    % pseudoline labels
    \node at (-0.3,4) {$4$};
    \node at (-0.3,3) {$3$};
    \node at (-0.3,2) {$2$};
    \node at (-0.3,1) {$1$};

    \node at (7.3,4) {$4$};
    \node at (7.3,3) {$3$};
    \node at (7.3,2) {$2$};
    \node at (7.3,1) {$1$};

    % weight labels
    \node at (1.3,3.5) {$z_1$};
    \node at (2.3,2.5) {$z_4$};
    \node at (3.3,1.5) {$z_6$};
    \node at (4.3,3.5) {$\frac{z_2}{z_4}$};
    \node at (5.3,2.5) {$\frac{z_5}{z_6}$};
    \node at (6.3,3.5) {$\frac{z_3}{z_5}$};
\end{tikzpicture}
\caption{The graph for $u_1$ when $n=4$} \label{The graph for u_1 when n=4}
\end{figure}

Computing the minor $\Delta^1_4(u_1)$ directly from the matrix $u_1$ is trivial. If we use the graph instead, we see that it is given by the weight of one path with three diagonal edges:
    $$\Delta^1_4(u_1) = z_6 \frac{z_5}{z_6} \frac{z_3}{z_5} = z_3.
    $$
The minor $\Delta^{\{1,2,3\}}_{\{2,3,4\}}(u_1)$ takes far more effort to compute directly, but using the graph makes the computation simple. This minor is the product of weights of three paths, each with one diagonal edge. We highlight the paths with parentheses:
    $$\Delta^{\{1,2,3\}}_{\{2,3,4\}}(u_1) = (z_6) (z_4) (z_1) = z_1z_4z_6.
    $$
\end{ex}

\begin{ex}
Again we take $n=4$, so $N=6$ and we wish to compute the minors $\Delta^J_K\left(u^T\right)$ where
    $$u^T= \left(\mathbf{x}_{-\mathbf{i}_0}^{\vee}(z_1, \ldots, z_6)\right)^T
        = \begin{pmatrix} \frac{1}{z_1z_4z_6} & \frac{z_1z_4(z_4z_6+z_5) + z_2z_5}{z_2z_4z_5z_6} & \frac{z_3z_5 +z_6(z_2z_5+z_3z_4)}{z_3z_5z_6} & 1 \\  & \frac{z_1z_4z_6}{z_2z_5} & \frac{z_6(z_2z_5+z_3z_4)}{z_3z_5} & z_6 \\ & & \frac{z_2z_5}{z_3} & z_5 \\ & & & z_3  \end{pmatrix}
    $$
with $\mathbf{i}_0 = (1,2,3,1,2,1)$. The graph for
    $$u^T = \mathbf{t}_{1}^{\vee}(z_6^{-1})\mathbf{x}_{1}^{\vee}(z_6) \mathbf{t}_{2}^{\vee}(z_5^{-1})\mathbf{x}_{2}^{\vee}(z_5) \mathbf{t}_{1}^{\vee}(z_4^{-1})\mathbf{x}_{1}^{\vee}(z_4) \mathbf{t}_{3}^{\vee}(z_3^{-1})\mathbf{x}_{3}^{\vee}(z_3) \mathbf{t}_{2}^{\vee}(z_2^{-1})\mathbf{x}_{2}^{\vee}(z_2) \mathbf{t}_{1}^{\vee}(z_1^{-1}) \mathbf{x}_{1}^{\vee}(z_1)
    $$
is given in Figure \ref{The graph for u^T when n=4}.

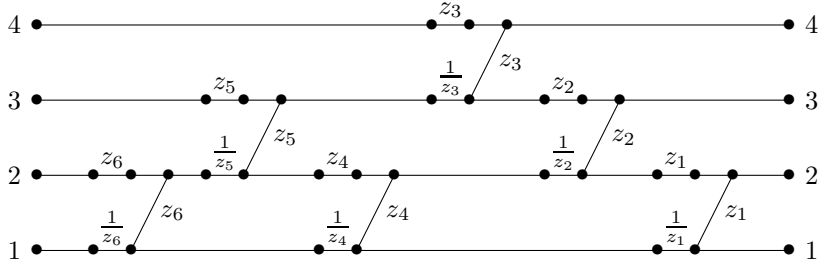
\begin{figure}[ht!]
\centering
\begin{tikzpicture}
    \draw (0,4) -- (10,4);
    \draw (0,3) -- (10,3);
    \draw (0,2) -- (10,2);
    \draw (0,1) -- (10,1);

    \draw (1.25,1) -- (1.75,2);
    \draw (2.75,2) -- (3.25,3);
    \draw (4.25,1) -- (4.75,2);
    \draw (5.75,3) -- (6.25,4);
    \draw (7.25,2) -- (7.75,3);
    \draw (8.75,1) -- (9.25,2);
    % dots and stars
    \node at (0,1) {$\bullet$};
    \node at (0,2) {$\bullet$};
    \node at (0,3) {$\bullet$};
    \node at (0,4) {$\bullet$};
    \node at (10,1) {$\bullet$};
    \node at (10,2) {$\bullet$};
    \node at (10,3) {$\bullet$};
    \node at (10,4) {$\bullet$};

    \node at (0.75,1) {$\bullet$};
    \node at (1.25,1) {$\bullet$};
    \node at (0.75,2) {$\bullet$};
    \node at (1.25,2) {$\bullet$};
    \node at (1.75,2) {$\bullet$};
    \node at (2.25,2) {$\bullet$};
    \node at (2.75,2) {$\bullet$};
    \node at (2.25,3) {$\bullet$};
    \node at (2.75,3) {$\bullet$};
    \node at (3.25,3) {$\bullet$};

    \node at (3.75,1) {$\bullet$};
    \node at (4.25,1) {$\bullet$};
    \node at (3.75,2) {$\bullet$};
    \node at (4.25,2) {$\bullet$};
    \node at (4.75,2) {$\bullet$};
    \node at (5.25,3) {$\bullet$};
    \node at (5.75,3) {$\bullet$};
    \node at (5.25,4) {$\bullet$};
    \node at (5.75,4) {$\bullet$};
    \node at (6.25,4) {$\bullet$};

    \node at (6.75,2) {$\bullet$};
    \node at (7.25,2) {$\bullet$};
    \node at (6.75,3) {$\bullet$};
    \node at (7.25,3) {$\bullet$};
    \node at (7.75,3) {$\bullet$};
    \node at (8.25,1) {$\bullet$};
    \node at (8.75,1) {$\bullet$};
    \node at (8.25,2) {$\bullet$};
    \node at (8.75,2) {$\bullet$};
    \node at (9.25,2) {$\bullet$};

    % pseudoline labels
    \node at (-0.3,4) {$4$};
    \node at (-0.3,3) {$3$};
    \node at (-0.3,2) {$2$};
    \node at (-0.3,1) {$1$};

    \node at (10.3,4) {$4$};
    \node at (10.3,3) {$3$};
    \node at (10.3,2) {$2$};
    \node at (10.3,1) {$1$};

    % weight labels
    \node at (1,1.3) {$\frac{1}{z_6}$};
    \node at (1,2.2) {$z_6$};
    \node at (1.8,1.5) {$z_6$};
    \node at (2.5,2.3) {$\frac{1}{z_5}$};
    \node at (2.5,3.2) {$z_5$};
    \node at (3.3,2.5) {$z_5$};
    \node at (4,1.3) {$\frac{1}{z_4}$};
    \node at (4,2.2) {$z_4$};
    \node at (4.8,1.5) {$z_4$};
    \node at (5.5,3.3) {$\frac{1}{z_3}$};
    \node at (5.5,4.2) {$z_3$};
    \node at (6.3,3.5) {$z_3$};
    \node at (7,2.3) {$\frac{1}{z_2}$};
    \node at (7,3.2) {$z_2$};
    \node at (7.8,2.5) {$z_2$};
    \node at (8.5,1.3) {$\frac{1}{z_1}$};
    \node at (8.5,2.2) {$z_1$};
    \node at (9.3,1.5) {$z_1$};
\end{tikzpicture}
\caption{The graph for $u^T$ when $n=4$} \label{The graph for u^T when n=4}
\end{figure}

We again give example computations, with parentheses highlighting products of multiple paths:
    $$\begin{aligned}
    &\Delta^{\{1,2\}}_{\{1,4\}}(u^T) = \left( \frac{1}{z_6} \frac{1}{z_4} \frac{1}{z_1} \right) \left( z_6 \frac{1}{z_5} z_5 \frac{1}{z_3} z_3 \right) = \frac{1}{z_1z_4},  \\
    &\Delta^{\{1,2\}}_{\{3,4\}}(u^T) = \left( \frac{1}{z_6} \frac{1}{z_4} z_4 \frac{1}{z_2} z_2 \right) \left( z_6 \frac{1}{z_5} z_5 \frac{1}{z_3} z_3 \right) = 1.
    \end{aligned}
    $$
\end{ex}

\begin{lem} \label{lem our Chamber Ansatz minors are monomial}
All minors in both applications of the Chamber Ansatz in the proof of Lemma \ref{lem form of u_1 and b factorisations} are monomial and consequently the resulting coordinate changes are monomial.
\end{lem}

\begin{proof}
In each application of the Chamber Ansatz, the relevant minors are those flag minors with column sets given by the chamber labels of the corresponding ansatz arrangements.
\begin{claim}
Let
    $$\begin{gathered}
    \mathbf{i}_0 = (i_1, \ldots, i_N) := (1,2, \ldots, n-1, 1,2 \ldots, n-2, \ldots, 1, 2, 1), \\
    \mathbf{i}'_0 = (i'_1, \ldots, i'_N) := (n-1, n-2, \ldots, 1, n-1, n-2, \ldots, 2, \ldots, n-1, n-2, n-1), \\
    {\mathbf{i}'_0}^{\text{op}} := (i'_N, \ldots, i'_1).
    \end{gathered}
    $$
Then
\begin{enumerate}
    \item Chamber labels of the ansatz arrangement for $\mathbf{i}_0$ are of the form $\{a, \ldots, b\}$.

    \item Chamber labels of the ansatz arrangement for ${\mathbf{i}'_0}^{\text{op}}$ are of the form $\{1, \ldots, a\}$, $\{b, \ldots, n\}$ or $\{1, \ldots, a\} \cup \{b, \ldots, n\}$.
\end{enumerate}
\end{claim}

Note that flag minors of $u_1$ and $u^T$ correspond to chamber labels of the ansatz arrangement for $\mathbf{i}_0$ and ${\mathbf{i}'_0}^{\text{op}}$ respectively.

\begin{proof}[Proof of Claim]
We may construct the reduced expression ${\mathbf{i}'_0}^{\text{op}}$ from $\mathbf{i}_0$ in two steps; first replace each $i_k$ in $\mathbf{i}_0$ with $n-i_k$ (this gives ${\mathbf{i}'_0}$) and then reverse the order. In terms of the ansatz arrangement, both of these operations result in reflections. Viewing the ansatz arrangement for $\mathbf{i}_0$ in the the plane, with the origin at the centre of the arrangement, we see that the first step above reflects the ansatz arrangement in the horizontal axis. Note that this causes each chamber label $S$ to change in the following way
    $$S \mapsto \{1, \ldots, n\}\setminus S.
    $$
In the second step we reverse the order of the reduced expression which gives a refection of the arrangement in the vertical axis and in particular there is no further change to the chamber labels. It follows that the two statements in the claim are equivalent and so we will consider only the $\mathbf{i}_0$ case.

Given the ansatz arrangement for $\mathbf{i}_0$ in dimension $n$, if we ignore the first $n-1$ crossings and the top pseudoline after this point (i.e. we remove the 1-string), the remaining graph has the form of the ansatz arrangement in dimension $n-1$, with labelling $2, \ldots n$ rather than $1, \ldots, n-1$. This is a consequence of the form of $\mathbf{i}_0$, namely that the reduced expression $\mathbf{i}_0$ in dimension $n-1$ is given by the last $\binom{n-1}{2}$ entries of the expression $\mathbf{i}_0$ in dimension $n$.

Since, by definition of the ansatz arrangement, the leftmost chamber labels are always given by sets of consecutive integers, it follows by induction that all chamber labels are of this form.
\end{proof}

We now use the graphs for $u_1$ and $u^T$ to see that the relevant flag minors are all monomial, namely by using Theorem \ref{thm minors from paths in graph} and showing that there is only one possible family of paths in each case.
\begin{enumerate}
    \item Minors of $u_1$:
    \begin{enumerate}
        \item Column sets of the form $\{1, \ldots, b\}$: Since $u_1 \in U^{\vee}_+$ these minors always equal $1$. We can see this from the graph for $u_1$ since the paths must be horizontal and the lack of non-trivial torus factors means that all horizontal edges have weight $1$.

        \item Column sets of the form $\{a, \ldots, b\}$ with $a>1$: Note that there is only one edge connecting the bottom two horizontal lines. After travelling up this edge there is only one possible path to the third line and so on. Thus there is only one path from the source $1$ to the sink $a$.

        In order for the paths in our family to remain vertex disjoint, the path from the source $2$ to the sink $a+1$ must take the first opportunity to travel upwards and indeed every possible opportunity to travel upwards until it reaches the line at height $a+1$. This imposes the same restriction on the path from the source $3$ to sink $a+2$ and so on for all paths in this family. In particular there is only one possible family of paths, thus these minors are monomial.
        For example see Figure \ref{Example family of paths for the proof of Lem CA minors monomial, u_1 case}.
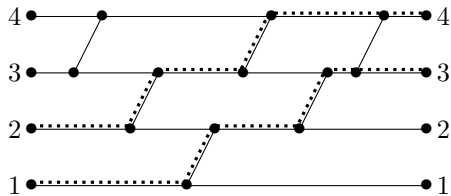
\begin{figure}[ht!]
\centering
\begin{tikzpicture}[scale=0.75]
    % paths
    \draw[very thick, dotted] (0,2.06) -- (1.71,2.06) -- (2.21,3.06) -- (3.71,3.06) -- (4.21,4.06) -- (7,4.06);
    \draw[very thick, dotted] (0,1.06) -- (2.71,1.06) -- (3.21,2.06) -- (4.71,2.06) -- (5.21,3.06) -- (7,3.06);

    % pseudolines
    \draw (0,4) -- (7,4);
    \draw (0,3) -- (7,3);
    \draw (0,2) -- (7,2);
    \draw (0,1) -- (7,1);

    \draw (0.75,3) -- (1.25,4);
    \draw (1.75,2) -- (2.25,3);
    \draw (2.75,1) -- (3.25,2);
    \draw (3.75,3) -- (4.25,4);
    \draw (4.75,2) -- (5.25,3);
    \draw (5.75,3) -- (6.25,4);
    % dots and stars
    \node at (0,1) {$\bullet$};
    \node at (0,2) {$\bullet$};
    \node at (0,3) {$\bullet$};
    \node at (0,4) {$\bullet$};
    \node at (7,1) {$\bullet$};
    \node at (7,2) {$\bullet$};
    \node at (7,3) {$\bullet$};
    \node at (7,4) {$\bullet$};

    \node at (0.75,3) {$\bullet$};
    \node at (1.25,4) {$\bullet$};
    \node at (1.75,2) {$\bullet$};
    \node at (2.25,3) {$\bullet$};
    \node at (2.75,1) {$\bullet$};
    \node at (3.25,2) {$\bullet$};

    \node at (3.75,3) {$\bullet$};
    \node at (4.25,4) {$\bullet$};
    \node at (4.75,2) {$\bullet$};
    \node at (5.25,3) {$\bullet$};
    \node at (5.75,3) {$\bullet$};
    \node at (6.25,4) {$\bullet$};

    % pseudoline labels
    \node at (-0.3,4) {$4$};
    \node at (-0.3,3) {$3$};
    \node at (-0.3,2) {$2$};
    \node at (-0.3,1) {$1$};

    \node at (7.3,4) {$4$};
    \node at (7.3,3) {$3$};
    \node at (7.3,2) {$2$};
    \node at (7.3,1) {$1$};

    % \node at (3.5,4.6) {$\vdots$}
\end{tikzpicture}
\caption{Example family of paths for the proof of Lemma \ref{lem our Chamber Ansatz minors are monomial}, $u_1$ case} \label{Example family of paths for the proof of Lem CA minors monomial, u_1 case}
\end{figure}
    \end{enumerate}

    \item Minors of $u^T$:
    \begin{enumerate}
        \item Column sets of the form $\{1, \ldots, a\}$: These minors are monomial since they correspond to horizontal paths in the graph for $u^T$.

        \item Column sets of the form $\{b, \ldots, n\}$:
        We first note that there is only one edge from the line at height $n-1$ to the $n$-th horizontal line. Before this point there is only one edge from the line at height $n-2$ to the line at height $n-1$. Working backwards in this way we see there is only one possible path from each source which ends at the sink $n$, and in particular only one such path from the source $n-b+1$.

        Similarly, in order to have vertex distinct paths there is now only one possible way to reach the sink $n-1$ from the source $n-b$. Continuing in this way we see that there is only one possible family of paths and so these minors are monomial.
        For example see Figure \ref{Example family of paths for the proof of Lem CA minors monomial, u^T case}.
\begin{figure}[ht!]
\centering
\begin{tikzpicture}[scale=0.75]
    % paths
    \draw[very thick, dotted] (0,2.06) -- (2.71,2.06) -- (3.21,3.06) -- (5.71,3.06) -- (6.21,4.06) -- (10,4.06);
    \draw[very thick, dotted] (0,1.06) -- (4.21,1.06) -- (4.71,2.06) -- (7.21,2.06) -- (7.71,3.06) -- (10,3.06);

    % pseudolines
    \draw (0,4) -- (10,4);
    \draw (0,3) -- (10,3);
    \draw (0,2) -- (10,2);
    \draw (0,1) -- (10,1);

    \draw (1.25,1) -- (1.75,2);
    \draw (2.75,2) -- (3.25,3);
    \draw (4.25,1) -- (4.75,2);
    \draw (5.75,3) -- (6.25,4);
    \draw (7.25,2) -- (7.75,3);
    \draw (8.75,1) -- (9.25,2);
    % dots and stars
    \node at (0,1) {$\bullet$};
    \node at (0,2) {$\bullet$};
    \node at (0,3) {$\bullet$};
    \node at (0,4) {$\bullet$};
    \node at (10,1) {$\bullet$};
    \node at (10,2) {$\bullet$};
    \node at (10,3) {$\bullet$};
    \node at (10,4) {$\bullet$};

    \node at (0.75,1) {$\bullet$};
    \node at (1.25,1) {$\bullet$};
    \node at (0.75,2) {$\bullet$};
    \node at (1.25,2) {$\bullet$};
    \node at (1.75,2) {$\bullet$};
    \node at (2.25,2) {$\bullet$};
    \node at (2.75,2) {$\bullet$};
    \node at (2.25,3) {$\bullet$};
    \node at (2.75,3) {$\bullet$};
    \node at (3.25,3) {$\bullet$};

    \node at (3.75,1) {$\bullet$};
    \node at (4.25,1) {$\bullet$};
    \node at (3.75,2) {$\bullet$};
    \node at (4.25,2) {$\bullet$};
    \node at (4.75,2) {$\bullet$};
    \node at (5.25,3) {$\bullet$};
    \node at (5.75,3) {$\bullet$};
    \node at (5.25,4) {$\bullet$};
    \node at (5.75,4) {$\bullet$};
    \node at (6.25,4) {$\bullet$};

    \node at (6.75,2) {$\bullet$};
    \node at (7.25,2) {$\bullet$};
    \node at (6.75,3) {$\bullet$};
    \node at (7.25,3) {$\bullet$};
    \node at (7.75,3) {$\bullet$};
    \node at (8.25,1) {$\bullet$};
    \node at (8.75,1) {$\bullet$};
    \node at (8.25,2) {$\bullet$};
    \node at (8.75,2) {$\bullet$};
    \node at (9.25,2) {$\bullet$};

    % pseudoline labels
    \node at (-0.3,4) {$4$};
    \node at (-0.3,3) {$3$};
    \node at (-0.3,2) {$2$};
    \node at (-0.3,1) {$1$};

    \node at (10.3,4) {$4$};
    \node at (10.3,3) {$3$};
    \node at (10.3,2) {$2$};
    \node at (10.3,1) {$1$};
\end{tikzpicture}
\caption{Example family of paths for the proof of Lemma \ref{lem our Chamber Ansatz minors are monomial}, $u^T$ case} \label{Example family of paths for the proof of Lem CA minors monomial, u^T case}
\end{figure}
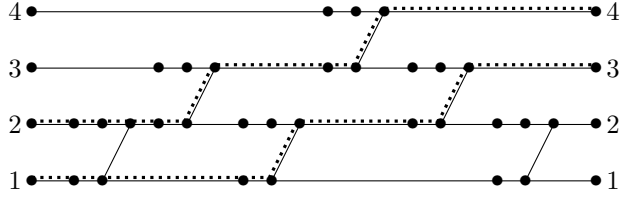

        \item Column sets of the form $\{1, \ldots, a\} \cup \{b, \ldots, n\}$: We see that these minors are all monomial by combining the previous two cases.
    \end{enumerate}
\end{enumerate}
\end{proof}

\subsection{The coordinate change} \label{subsec The coordinate change}

In this section we address two final lemmas needed for the proof of Theorem \ref{thm coord change}, both detailing coordinate changes. We then recall and prove this theorem.

\begin{lem} \label{lem coord change for u_1 p_i in terms of z_i}
For $k=1,\ldots, n-1$, $a=1,\ldots, n-k$ we have
    $$p_{s_k+a} =\begin{cases}
        z_{1+s_{a}} & \text{if } k=1 \\
        \frac{z_{k+s_{a}}}{z_{k-1+s_{a+1}}} & \text{otherwise}
        \end{cases}
    \qquad \text{where} \ \
    s_k := \sum_{j=1}^{k-1}(n-j).
    $$
\end{lem}

\begin{proof}
In the proof of Lemma \ref{lem form of u_1 and b factorisations} we used the Chamber Ansatz with ${\mathbf{i}'_0}^{\mathrm{op}}$ to show that $u_1 = \mathbf{x}_{i'_1}^{\vee}(p_1) \cdots \mathbf{x}_{i'_N}^{\vee}(p_N)$ where
    $$
    p_k = \frac{\prod_{j\neq i'_k} \Delta^{\omega^{\vee}_j}_{w_{(N-k+1)}\omega^{\vee}_j}(u^T)^{-a_{j,i'_k}}}{\Delta^{\omega^{\vee}_{i_k}}_{w_{(N-k+1)}\omega^{\vee}_{i'_k}}(u^T) \Delta^{\omega^{\vee}_{i'_k}}_{w_{(N-k)}\omega^{\vee}_{i'_k}}(u^T)}, \quad k=1, \ldots, N.
    $$
Since this expression is quite unpleasant, we instead work diagrammatically. In particular, we will use the ansatz arrangement for ${\mathbf{i}'_0}^{\mathrm{op}}$ and the graph for $u^T$.

To give a visual aid we recall Figure \ref{fig ansatz arrangement i'_0^op dim 4}, the dimension $4$ example of the ansatz arrangement for ${\mathbf{i}'_0}^{\mathrm{op}} = (i'_N, \ldots, i'_1)=(3,2,3,1,2,3)$:
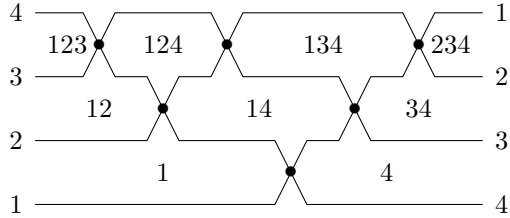
\begin{center}
    \begin{tikzpicture}[scale=0.85]
            % pseudolines
            \draw (0,4) -- (0.75,4) -- (1.25,3) -- (1.75,3) -- (2.25,2) -- (3.75,2) -- (4.25,1) -- (7,1);
            \draw (0,3) -- (0.75,3) -- (1.25,4) -- (2.75,4) -- (3.25,3) -- (4.75,3) -- (5.25,2) -- (7,2);
            \draw (0,2) -- (1.75,2) -- (2.25,3) -- (2.75,3) -- (3.25,4) -- (5.75,4) -- (6.25,3) -- (7,3);
            \draw (0,1) -- (3.75,1) -- (4.25,2) -- (4.75,2) -- (5.25,3) -- (5.75,3) -- (6.25,4) -- (7,4);

            % dots and stars
            \node at (1,3.5) {$\bullet$};
            \node at (2,2.5) {$\bullet$};
            \node at (3,3.5) {$\bullet$};
            \node at (4,1.5) {$\bullet$};
            \node at (5,2.5) {$\bullet$};
            \node at (6,3.5) {$\bullet$};

            % pseudoline labels
            \node at (-0.3,4) {$4$};
            \node at (-0.3,3) {$3$};
            \node at (-0.3,2) {$2$};
            \node at (-0.3,1) {$1$};

            \node at (7.3,4) {$1$};
            \node at (7.3,3) {$2$};
            \node at (7.3,2) {$3$};
            \node at (7.3,1) {$4$};

            % chamber labels
            \node at (0.5,3.5) {$123$};
            \node at (2,3.5) {$124$};
            \node at (4.5,3.5) {$134$};
            \node at (6.5,3.5) {$234$};

            \node at (1,2.5) {$12$};
            \node at (3.5,2.5) {$14$};
            \node at (6,2.5) {$34$};

            \node at (2,1.5) {$1$};
            \node at (5.5,1.5) {$4$};
    \end{tikzpicture}
\end{center}

We now define a new labelling of the chambers of the ansatz arrangement for ${\mathbf{i}'_0}^{\mathrm{op}}$ in terms of pairs $(k,a)$. This is motivated by two facts:
\begin{enumerate}
    \item Any integer $1 \leq m \leq N:=\binom{n}{2}$ may be written as
        $$m = s_k +a
        $$
    for some unique pair $(k,a)$, with $k=1,\ldots, n-1$, $a=1,\ldots, n-k$.
    \item For such a pair $(k,a)$, the label of the chamber to the \emph{left} of the $\left(N-s_k-a+1\right)$-th crossing is given by
        $$\begin{cases}
        \{1, \ldots, k \} & \text{if } k+a=n, \\
        \{1, \ldots, k \}\cup\{k+a+1, \ldots, n \} & \text{if } k+a\neq n.
        \end{cases}
        $$
\end{enumerate}
It follows that we assign the pair $(k,a)$ to the chamber on the left of the $\left(N-s_k-a+1\right)$-th crossing. The rightmost chambers are labelled consistently, taking $k=0$. We leave the chambers above and below the pseudoline arrangement unlabelled.

Continuing with our dimension $4$ example, the chamber pairs $(k,a)$ are shown in Figure \ref{fig ansatz arrangement i'_0^op dim 4 (k,a) labels}. In general we consider the $(N-s_k-a+1)$-th crossing. The pairs $(k,a)$ for each surrounding chamber are given, where they are defined, by the diagram in Figure \ref{fig chamber pair labels i'_0^op ansatz arrangement}.
\begin{figure}[hb]
\centering
\begin{minipage}[b]{0.47\textwidth}
    \centering
\begin{tikzpicture}[scale=0.85]
    % pseudolines
    \draw (0,4) -- (0.75,4) -- (1.25,3) -- (1.75,3) -- (2.25,2) -- (3.75,2) -- (4.25,1) -- (7,1);
    \draw (0,3) -- (0.75,3) -- (1.25,4) -- (2.75,4) -- (3.25,3) -- (4.75,3) -- (5.25,2) -- (7,2);
    \draw (0,2) -- (1.75,2) -- (2.25,3) -- (2.75,3) -- (3.25,4) -- (5.75,4) -- (6.25,3) -- (7,3);
    \draw (0,1) -- (3.75,1) -- (4.25,2) -- (4.75,2) -- (5.25,3) -- (5.75,3) -- (6.25,4) -- (7,4);

    % dots and stars
    \node at (1,3.5) {$\bullet$};
    \node at (2,2.5) {$\bullet$};
    \node at (3,3.5) {$\bullet$};
    \node at (4,1.5) {$\bullet$};
    \node at (5,2.5) {$\bullet$};
    \node at (6,3.5) {$\bullet$};

    % pseudoline labels
    \node at (-0.3,4) {$4$};
    \node at (-0.3,3) {$3$};
    \node at (-0.3,2) {$2$};
    \node at (-0.3,1) {$1$};

    \node at (7.3,4) {$1$};
    \node at (7.3,3) {$2$};
    \node at (7.3,2) {$3$};
    \node at (7.3,1) {$4$};

    % chamber labels
    \node at (0.3,3.5) {$(3,1)$};
    \node at (2,3.5) {$(2,1)$};
    \node at (4.5,3.5) {$(1,1)$};
    \node at (6.7,3.5) {$(0,1)$};

    \node at (1,2.5) {$(2,2)$};
    \node at (3.5,2.5) {$(1,2)$};
    \node at (6,2.5) {$(0,2)$};

    \node at (2,1.5) {$(1,3)$};
    \node at (5.5,1.5) {$(0,3)$};
\end{tikzpicture}
\caption{The ansatz arrangement for ${\mathbf{i}'_0}^{\mathrm{op}}$ in dimension 4 with $(k,a)$ labelling} \label{fig ansatz arrangement i'_0^op dim 4 (k,a) labels}
\end{minipage}
    \hfill
\begin{minipage}[b]{0.47\textwidth}
    \centering
\begin{tikzpicture}
    %dots and stars
    \node at (1,0.5) {$\bullet$};

    \draw[<-, >=stealth, densely dashed, rounded corners=15] (1.08,0.52) -- (1.7, 0.9) -- (2.2,1.85);
    \draw (-0.3,1) -- (0.75,1) -- (1.25,0) -- (2.3,0);
    \draw (-0.3,0) -- (0.75,0) -- (1.25,1) -- (2.3,1);
    %\labels
    \node at (1,-0.3) {{$(k-1, a+1)$}};
    \node at (0.2,0.5) {{$(k, a)$}};
    \node at (2.1,0.5) {{$(k-1, a)$}};
    \node at (1,1.3) {{$(k, a-1)$}};
    \node at (1,2) {{$\left(N-s_k-a+1\right)$-th crossing}};
\end{tikzpicture}
\caption{Labelling of chamber pairs $(k,a)$ for ${\mathbf{i}'_0}^{\mathrm{op}}$ ansatz arrangement} \label{fig chamber pair labels i'_0^op ansatz arrangement}
\end{minipage}
\end{figure}

With this notation we return to the coordinate change given by the Chamber Ansatz. It requires us to compute the minors corresponding to the chamber labels, which we will do in terms of the pairs $(k,a)$. To help us with this we recall Figure \ref{The graph for u^T when n=4}, namely that in dimension $4$ the graph for
    $$u^T = \mathbf{t}_{1}^{\vee}(z_6^{-1})\mathbf{x}_{1}^{\vee}(z_6) \mathbf{t}_{2}^{\vee}(z_5^{-1})\mathbf{x}_{2}^{\vee}(z_5) \mathbf{t}_{1}^{\vee}(z_4^{-1})\mathbf{x}_{1}^{\vee}(z_4) \mathbf{t}_{3}^{\vee}(z_3^{-1})\mathbf{x}_{3}^{\vee}(z_3) \mathbf{t}_{2}^{\vee}(z_2^{-1})\mathbf{x}_{2}^{\vee}(z_2) \mathbf{t}_{1}^{\vee}(z_1^{-1}) \mathbf{x}_{1}^{\vee}(z_1)
    $$
is given by
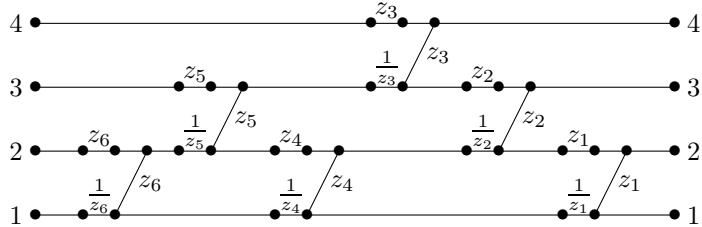
\begin{center}
\begin{tikzpicture}[scale=0.85]
    \draw (0,4) -- (10,4);
    \draw (0,3) -- (10,3);
    \draw (0,2) -- (10,2);
    \draw (0,1) -- (10,1);

    \draw (1.25,1) -- (1.75,2);
    \draw (2.75,2) -- (3.25,3);
    \draw (4.25,1) -- (4.75,2);
    \draw (5.75,3) -- (6.25,4);
    \draw (7.25,2) -- (7.75,3);
    \draw (8.75,1) -- (9.25,2);
    % dots and stars
    \node at (0,1) {$\bullet$};
    \node at (0,2) {$\bullet$};
    \node at (0,3) {$\bullet$};
    \node at (0,4) {$\bullet$};
    \node at (10,1) {$\bullet$};
    \node at (10,2) {$\bullet$};
    \node at (10,3) {$\bullet$};
    \node at (10,4) {$\bullet$};

    \node at (0.75,1) {$\bullet$};
    \node at (1.25,1) {$\bullet$};
    \node at (0.75,2) {$\bullet$};
    \node at (1.25,2) {$\bullet$};
    \node at (1.75,2) {$\bullet$};
    \node at (2.25,2) {$\bullet$};
    \node at (2.75,2) {$\bullet$};
    \node at (2.25,3) {$\bullet$};
    \node at (2.75,3) {$\bullet$};
    \node at (3.25,3) {$\bullet$};

    \node at (3.75,1) {$\bullet$};
    \node at (4.25,1) {$\bullet$};
    \node at (3.75,2) {$\bullet$};
    \node at (4.25,2) {$\bullet$};
    \node at (4.75,2) {$\bullet$};
    \node at (5.25,3) {$\bullet$};
    \node at (5.75,3) {$\bullet$};
    \node at (5.25,4) {$\bullet$};
    \node at (5.75,4) {$\bullet$};
    \node at (6.25,4) {$\bullet$};

    \node at (6.75,2) {$\bullet$};
    \node at (7.25,2) {$\bullet$};
    \node at (6.75,3) {$\bullet$};
    \node at (7.25,3) {$\bullet$};
    \node at (7.75,3) {$\bullet$};
    \node at (8.25,1) {$\bullet$};
    \node at (8.75,1) {$\bullet$};
    \node at (8.25,2) {$\bullet$};
    \node at (8.75,2) {$\bullet$};
    \node at (9.25,2) {$\bullet$};

    % pseudoline labels
    \node at (-0.3,4) {$4$};
    \node at (-0.3,3) {$3$};
    \node at (-0.3,2) {$2$};
    \node at (-0.3,1) {$1$};

    \node at (10.3,4) {$4$};
    \node at (10.3,3) {$3$};
    \node at (10.3,2) {$2$};
    \node at (10.3,1) {$1$};

    % weight labels
    \node at (1,1.3) {$\frac{1}{z_6}$};
    \node at (1,2.2) {$z_6$};
    \node at (1.8,1.5) {$z_6$};
    \node at (2.5,2.3) {$\frac{1}{z_5}$};
    \node at (2.5,3.2) {$z_5$};
    \node at (3.3,2.5) {$z_5$};
    \node at (4,1.3) {$\frac{1}{z_4}$};
    \node at (4,2.2) {$z_4$};
    \node at (4.8,1.5) {$z_4$};
    \node at (5.5,3.3) {$\frac{1}{z_3}$};
    \node at (5.5,4.2) {$z_3$};
    \node at (6.3,3.5) {$z_3$};
    \node at (7,2.3) {$\frac{1}{z_2}$};
    \node at (7,3.2) {$z_2$};
    \node at (7.8,2.5) {$z_2$};
    \node at (8.5,1.3) {$\frac{1}{z_1}$};
    \node at (8.5,2.2) {$z_1$};
    \node at (9.3,1.5) {$z_1$};
\end{tikzpicture}
\end{center}

To compute the minors corresponding to the chamber labels, we first use (\ref{eqn wt matrix z coords}) from Section \ref{subsec The form of the weight matrix} to see that for $k=1,\ldots, n-1$ we have
    $$\Delta^{\{1, \ldots, k\}}_{\{1, \ldots, k\}}(u^T) = \left(\prod_{\substack{1\leq m \leq N \\ i_m \in \{1, \ldots, k\}}} \frac{1}{z_m} \right) \left( \prod_{\substack{1\leq m \leq N \\ i_m \in \{1, \ldots, k-1\}}} z_m \right)
        = \prod_{\substack{1\leq m \leq N \\ i_m = k}} \frac{1}{z_m}.
    $$
Note that if $a=1$ then the chamber in the ansatz arrangement above the relevant crossings has label $\{1, \ldots, n\}$, with corresponding minor
    $$\Delta^{\{1, \ldots, n\}}_{\{1, \ldots, n\}}(u^T) = \left(\prod_{\substack{1\leq m \leq N \\ i_m \in \{1, \ldots, n-1\}}} \frac{1}{z_m} \right) \left( \prod_{\substack{1\leq m \leq N \\ i_m \in \{1, \ldots, n-1\}}} z_m \right)
        = 1.
    $$

For the remaining minors, using the graph for $u^T$ and Theorem \ref{thm minors from paths in graph} we see that
    $$\Delta^{\{1, \ldots, n-a\}}_{\{1, \ldots, k\}\cup\{k+a+1, \ldots, n \}}(u^T) = \Delta^{\{1, \ldots, k\}}_{\{1, \ldots, k\}}(u^T) \Delta^{\{k+1, \ldots, n-a\}}_{\{k+a+1, \ldots, n \}}(u^T).
    $$

The minors $\Delta^{\{k+1, \ldots, n-a\}}_{\{k+a+1, \ldots, n \}}(u^T)$ correspond to paths in the graph for $u^T$ which form `staircases' so their weights have contributions from both horizontal and diagonal edges.

The proof of Lemma \ref{lem our Chamber Ansatz minors are monomial} implies that on each path there are no horizontal edges with non-trivial weight after the last diagonal edge has been traversed. Additionally there is only one horizontal edge with non-trivial weight between each diagonal `step', namely the edge directly preceding the diagonal in Figure \ref{Labelled line segments in graphs in the proof of Lem coord change for u_1 p_i in terms of z_i}.
\begin{figure}[ht!]
\centering
\begin{tikzpicture}[scale=0.85]
    \draw (0,4) -- (3,4);
    \draw (0,3) -- (3,3);
    \draw (0,2) -- (3,2);
    \draw (0,1) -- (3,1);

    \draw (1.5,2) -- (2,3);

    % dots and stars
    \node at (1,2) {$\bullet$};
    \node at (1.5,2) {$\bullet$};
    \node at (1,3) {$\bullet$};
    \node at (1.5,3) {$\bullet$};
    \node at (2,3) {$\bullet$};

    % weight labels
    \node at (1.25,2.35) {$\frac{1}{z_m}$};
    \node at (1.25,3.25) {$z_m$};
    \node at (2.1,2.5) {$z_m$};
\end{tikzpicture}
\caption{Labelled line segments in graphs in the proof of Lemma \ref{lem coord change for u_1 p_i in terms of z_i}} \label{Labelled line segments in graphs in the proof of Lem coord change for u_1 p_i in terms of z_i}
\end{figure}
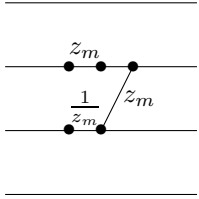

In particular, on each path the contributions from the diagonal edges and the horizontal edges directly preceding them will always cancel. So, roughly speaking, for each path we only need to consider the contributions from the horizontal edges which occur well before the first diagonal `step'; that is, if the first diagonal edge in the path occurs with weight $z_p$ then the only factors in the evaluation of the minor will be $\frac{1}{z_m}$, ${z_m}$ with $m>p$ (note that the ordering of the $i_m$ is reversed in the graph of $u^T$ due to taking the transpose):
    $$\begin{aligned}
    \Delta^{\{k+1, \ldots, n-a\}}_{\{k+a+1, \ldots, n \}}(u^T)
        &= \left( \prod_{\substack{m > i_m+\sum_{j=1}^{a-1} (n-j) \\ i_m\in \{k+1, \ldots, n-a\}}} \frac{1}{z_m} \right) \left( \prod_{\substack{m > i_m+\sum_{j=1}^{a-1} (n-j) \\  i_m\in \{k, \ldots, n-a-1\}}} z_m \right) \\
        &= \left( \prod_{\substack{m > n-a+s_a \\ i_m=n-a}} \frac{1}{z_m} \right) \left( \prod_{\substack{m > k+s_a \\ i_m=k}} z_m \right).
    \end{aligned}
    $$
Noticing that the last occurrence of $i_m=n-a$ in the reduced expression $\mathbf{i}_0$, is exactly when $m=n-a+s_{a}$, we obtain
    $$\Delta^{\{k+1, \ldots, n-a\}}_{\{k+a+1, \ldots, n \}}(u^T)
        = \prod_{\substack{m > k+s_{a} \\ i_m=k}} z_m.
    $$
Additionally we see that this minor is equal to $1$ if $k=0$, since $i_m\geq 1$ by definition.

Combining the above we obtain
    $$\begin{aligned}
    \Delta^{\{1, \ldots, n-a\}}_{\{1, \ldots, k\}\cup\{k+a+1, \ldots, n \}}(u^T)
        &= \Delta^{\{1, \ldots, k\}}_{\{1, \ldots, k\}}(u^T) \Delta^{\{k+1, \ldots, n-a\}}_{\{k+a+1, \ldots, n \}}(u^T) \\
        &= \left( \prod_{\substack{1\leq m \leq N \\ i_m = k}} \frac{1}{z_m} \right) \left( \prod_{\substack{m > k+s_{a} \\ i_m=k}} z_m \right) =  \prod_{\substack{m \leq k+s_{a} \\ i_m=k}} \frac{1}{z_m}.
    \end{aligned}
    $$

We now use the Chamber Ansatz to compute the $p_{s_k+a}$ coordinates, considering separately the case when $a=1$.

For $k=1,\ldots, n-1$, $a=2,\ldots, n-k$ we have
    \begin{equation}
    p_{s_k+a}
        = \frac{\left(\prod\limits_{\substack{m\leq k+s_{a-1} \\ i_{m}=k}} \frac{1}{z_{m}}\right) \left(\prod\limits_{\substack{m\leq k-1+s_{a+1} \\ i_{m}=k-1}} \frac{1}{z_{m}}\right)}{\left(\prod\limits_{\substack{m\leq k+s_{a} \\ i_m=k}} \frac{1}{z_m}\right) \left(\prod\limits_{\substack{ m\leq k-1+s_{a} \\ i_{m}=k-1}} \frac{1}{z_{m}}\right)}
        = \frac{\prod\limits_{\substack{k+s_{a-1} < m\leq k+s_{a} \\ i_m=k}} z_m }{\prod\limits_{\substack{k-1+s_{a} < m\leq k-1+s_{a+1} \\ i_{m}=k-1}} z_m}, \label{eqn t' coord general case}
    \end{equation}
We note that each of the products on the last line has exactly one term, namely when $m$ is equal to the upper bound. Thus we obtain
    $$p_{s_k+a} = \begin{cases}
        z_{k+s_{a}} & \text{if } k=1, \\
        \frac{z_{k+s_{a}}}{z_{k-1+s_{a+1}}} & \text{otherwise}.
        \end{cases}
    $$

We apply a similar argument in the case when $a=1$. For $k=1,\ldots, n-1$ the minor corresponding to the chamber above the pseudoline arrangement is equal to $1$ so we have
    $$\begin{aligned}
    p_{s_k+a}
        = \frac{\prod\limits_{\substack{m\leq k-1+s_{a+1} \\ i_{m}=k-1}} \frac{1}{z_{m}}}{\left(\prod\limits_{\substack{m\leq k+s_{a} \\ i_m=k}} \frac{1}{z_m}\right) \left(\prod\limits_{\substack{ m\leq k-1+s_{a} \\ i_{m}=k-1}} \frac{1}{z_{m}}\right)}
        &= \frac{\prod\limits_{\substack{m\leq k \\ i_m=k}} z_m }{\prod\limits_{\substack{k-1< m\leq k-1+(n-1) \\ i_{m}=k-1}} z_m} \\
        &= \begin{cases}
            z_{k} & \text{if } k=1 \\
            \frac{z_{k}}{z_{k-1+(n-1)}} & \text{otherwise}
            \end{cases}\\
        &= \begin{cases}
            z_{k+s_{a}} & \text{if } k=1 \\
            \frac{z_{k+s_{a}}}{z_{k-1+s_{a+1}}} & \text{otherwise}
            \end{cases}
    \end{aligned}
    $$
since $s_{a+1}=s_2=n-1$ and $s_{a}=s_1=0$.

Note that the need to consider the $a=1$ case separately is highlighted in (\ref{eqn t' coord general case}); if $a=1$ the numerator would be
    $$\prod\limits_{\substack{k+s_{0} < m\leq k+s_{1} \\ i_m=k}} z_m = \prod\limits_{\substack{k < m\leq k \\ i_m=k}} z_m =1
    $$
whereas in fact we should have $z_k=z_{k+s_{a}}$ in the numerator.

Combining the above we obtain the desired coordinate change:
    $$p_{s_k+a} =\begin{cases}
        z_{k+s_{a}} & \text{if } k=1, \\
        \frac{z_{k+s_{a}}}{z_{k-1+s_{a+1}}} & \text{otherwise}.
    \end{cases}
    $$
\end{proof}

\begin{lem} \label{lem coord change for b m_i in terms of p_i}
For $k=1,\ldots, n-1$, $a=1,\ldots, n-k$ we have
    $$\frac{1}{m_{s_k+a}}= \begin{cases}
        \prod\limits_{\substack{r=1,\ldots, k}} \frac{1}{p_{s_{r+1}-a+1}} &\text{if } k=1 \\
        \frac{\prod\limits_{\substack{r=1,\ldots, k-1}} p_{s_{r+1}-a} }{\prod\limits_{\substack{r=1,\ldots, k}} p_{s_{r+1}-a+1}} &\text{otherwise}
    \end{cases}
    \qquad \text{where} \ \
    s_k := \sum_{j=1}^{k-1}(n-j).
    $$
\end{lem}

\begin{proof}

In the proof of Lemma \ref{lem form of u_1 and b factorisations} we used the Chamber Ansatz with $\mathbf{i}_0$ to show that  $b B^{\vee}_+ = \mathbf{y}_{i_1}\left(\frac{1}{m_1}\right) \cdots \mathbf{y}_{i_N}\left(\frac{1}{m_N}\right)B^{\vee}_+$ where
    $$
    \frac{1}{m_k} = \frac{\prod_{j\neq i_k} \Delta^{\omega_j}_{w_{(k)}\omega_j}(u_1)^{-a_{j,i_k}}}{\Delta^{\omega_{i_k}}_{w_{(k)}\omega_{i_k}}(u_1) \Delta^{\omega_{i_k}}_{w_{(k-1)}\omega_{i_k}}(u_1)}, \quad k=1, \ldots, N.
    $$
Similar to the proof of the first coordinate change, since this expression is unpleasant we again work diagrammatically. In particular, we will use the ansatz arrangement for $\mathbf{i}'_0$ and the graph for $u_1$.

As before, to give a visual aid we recall Figure \ref{fig ansatz arrangement i_0 dim 4}, the dimension $4$ example of the ansatz arrangement for $\mathbf{i}_0=(1,2,3,1,2,1)$:
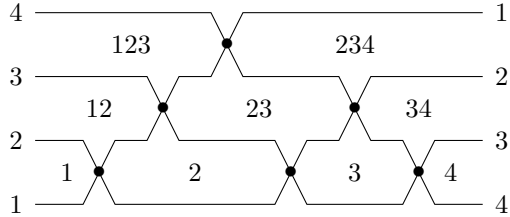
\begin{center}
    \begin{tikzpicture}[scale=0.85]
        % pseudolines
        \draw (0,4) -- (2.75,4) -- (3.25,3) -- (4.75,3) -- (5.25,2) -- (5.75,2) -- (6.25,1) -- (7,1);
        \draw (0,3) -- (1.75,3) -- (2.25,2) -- (3.75,2) -- (4.25,1) -- (5.75,1) -- (6.25,2) -- (7,2);
        \draw (0,2) -- (0.75,2) -- (1.25,1) -- (3.75,1) -- (4.25,2) -- (4.75,2) -- (5.25,3) -- (7,3);
        \draw (0,1) -- (0.75,1) -- (1.25,2) -- (1.75,2) -- (2.25,3) -- (2.75,3) -- (3.25,4) -- (7,4);

        % dots and stars
        \node at (1,1.5) {$\bullet$};
        \node at (2,2.5) {$\bullet$};
        \node at (3,3.5) {$\bullet$};
        \node at (4,1.5) {$\bullet$};
        \node at (5,2.5) {$\bullet$};
        \node at (6,1.5) {$\bullet$};

        % pseudoline labels
        \node at (-0.3,4) {$4$};
        \node at (-0.3,3) {$3$};
        \node at (-0.3,2) {$2$};
        \node at (-0.3,1) {$1$};

        \node at (7.3,4) {$1$};
        \node at (7.3,3) {$2$};
        \node at (7.3,2) {$3$};
        \node at (7.3,1) {$4$};

        % chamber labels
        \node at (1.5,3.5) {$123$};
        \node at (5,3.5) {$234$};

        \node at (1,2.5) {$12$};
        \node at (3.5,2.5) {$23$};
        \node at (6,2.5) {$34$};

        \node at (0.5,1.5) {$1$};
        \node at (2.5,1.5) {$2$};
        \node at (5,1.5) {$3$};
        \node at (6.5,1.5) {$4$};
    \end{tikzpicture}
\end{center}

Similar to the $u^T$ case, we give new labelling of the chambers of the ansatz arrangement in terms of pairs $(k,a)$. This time, for $k=1,\ldots, n-1$, $a=1,\ldots, n-k$, the label of the chamber to the \emph{right} of the $(s_k+a)$-th crossing is given by
    $$\{ k+1, \ldots, k+a \}.
    $$
It is then natural to assign the pair $(k,a)$ to the chamber on the right of the $(s_k+a)$-th crossing. The leftmost chambers will be labelled consistently, taking $k=0$. Again we leave the chambers above and below the pseudoline arrangement unlabelled.

Continuing our dimension $4$ example, the chamber pairs $(k,a)$ are shown in Figure \ref{fig ansatz arrangement i_0 dim 4 (k,a) labels}. In general, the pairs $(k,a)$ for each chamber surrounding the $(s_k+a)$-th crossing, where there are defined, are given by the diagram in Figure \ref{fig chamber pair labels i_0 ansatz arrangement}.
\begin{figure}[ht!]
\centering
\begin{minipage}[b]{0.47\textwidth}
    \centering
\begin{tikzpicture}[scale=0.85]
    % pseudolines
    \draw (0,4) -- (2.75,4) -- (3.25,3) -- (4.75,3) -- (5.25,2) -- (5.75,2) -- (6.25,1) -- (7,1);
    \draw (0,3) -- (1.75,3) -- (2.25,2) -- (3.75,2) -- (4.25,1) -- (5.75,1) -- (6.25,2) -- (7,2);
    \draw (0,2) -- (0.75,2) -- (1.25,1) -- (3.75,1) -- (4.25,2) -- (4.75,2) -- (5.25,3) -- (7,3);
    \draw (0,1) -- (0.75,1) -- (1.25,2) -- (1.75,2) -- (2.25,3) -- (2.75,3) -- (3.25,4) -- (7,4);

    % dots and stars
    \node at (1,1.5) {$\bullet$};
    \node at (2,2.5) {$\bullet$};
    \node at (3,3.5) {$\bullet$};
    \node at (4,1.5) {$\bullet$};
    \node at (5,2.5) {$\bullet$};
    \node at (6,1.5) {$\bullet$};

    % pseudoline labels
    \node at (-0.3,4) {$4$};
    \node at (-0.3,3) {$3$};
    \node at (-0.3,2) {$2$};
    \node at (-0.3,1) {$1$};

    \node at (7.3,4) {$1$};
    \node at (7.3,3) {$2$};
    \node at (7.3,2) {$3$};
    \node at (7.3,1) {$4$};

    % chamber labels
    \node at (1.5,3.5) {$(0,3)$};
    \node at (5,3.5) {$(1,3)$};

    \node at (1,2.5) {$(0,2)$};
    \node at (3.5,2.5) {$(1,2)$};
    \node at (6,2.5) {$(2,2)$};

    \node at (0.3,1.5) {$(0,1)$};
    \node at (2.5,1.5) {$(1,1)$};
    \node at (5,1.5) {$(2,1)$};
    \node at (6.7,1.5) {$(3,1)$};
\end{tikzpicture}
\caption{The ansatz arrangement for $\mathbf{i}_0$ in dimension 4 with $(k,a)$ labelling} \label{fig ansatz arrangement i_0 dim 4 (k,a) labels}
\end{minipage}
    \hfill
\begin{minipage}[b]{0.47\textwidth}
    \centering
\begin{tikzpicture}
    %dots and stars
    \node at (1,0.5) {$\bullet$};

    \draw[<-, >=stealth, densely dashed, rounded corners=15] (1.08,0.52) -- (2, 0.9) -- (2.2,1.85);
    \draw (-0.3,1) -- (0.75,1) -- (1.25,0) -- (2.3,0);
    \draw (-0.3,0) -- (0.75,0) -- (1.25,1) -- (2.3,1);
    %\labels
    \node at (1,-0.3) {{$(k, a-1)$}};
    \node at (-0.2,0.5) {{$(k-1, a)$}};
    \node at (1.9,0.5) {{$(k, a)$}};
    \node at (1,1.3) {{$(k-1, a+1)$}};
    \node at (1,2) {{$(s_k+a)$-th crossing}};
\end{tikzpicture}
\caption{Labelling of chamber pairs $(k,a)$ for $\mathbf{i}_0$ ansatz arrangement} \label{fig chamber pair labels i_0 ansatz arrangement}
\end{minipage}
\end{figure}

Before applying the Chamber Ansatz we again compute minors in terms of the pairs $(k,a)$. We have just seen that in dimension $4$ the matrix $u_1$ is given by
    $$u_1=\mathbf{x}_{\mathbf{i}'_0}^{\vee}\left(z_1, z_4, z_6, \frac{z_2}{z_4}, \frac{z_5}{z_6}, \frac{z_3}{z_5}\right).
    $$
This is quite messy however, so for ease of computation we will continue to work for with the $p_j$ coordinates; $u_1=\mathbf{x}_{\mathbf{i}'_0}^{\vee}(p_1, \ldots, p_6)$. The corresponding graph for $u_1$ is given in Figure \ref{The graph for u_1 when n=4, p coords}.
\begin{figure}[ht!]
\centering
    \begin{tikzpicture}[scale=0.85]
        % pseudolines
        \draw (0,4) -- (7,4);
        \draw (0,3) -- (7,3);
        \draw (0,2) -- (7,2);
        \draw (0,1) -- (7,1);

        \draw (0.75,3) -- (1.25,4);
        \draw (1.75,2) -- (2.25,3);
        \draw (2.75,1) -- (3.25,2);
        \draw (3.75,3) -- (4.25,4);
        \draw (4.75,2) -- (5.25,3);
        \draw (5.75,3) -- (6.25,4);
        % dots and stars
        \node at (0,1) {$\bullet$};
        \node at (0,2) {$\bullet$};
        \node at (0,3) {$\bullet$};
        \node at (0,4) {$\bullet$};
        \node at (7,1) {$\bullet$};
        \node at (7,2) {$\bullet$};
        \node at (7,3) {$\bullet$};
        \node at (7,4) {$\bullet$};

        \node at (0.75,3) {$\bullet$};
        \node at (1.25,4) {$\bullet$};
        \node at (1.75,2) {$\bullet$};
        \node at (2.25,3) {$\bullet$};
        \node at (2.75,1) {$\bullet$};
        \node at (3.25,2) {$\bullet$};

        \node at (3.75,3) {$\bullet$};
        \node at (4.25,4) {$\bullet$};
        \node at (4.75,2) {$\bullet$};
        \node at (5.25,3) {$\bullet$};
        \node at (5.75,3) {$\bullet$};
        \node at (6.25,4) {$\bullet$};

        % pseudoline labels
        \node at (-0.3,4) {$4$};
        \node at (-0.3,3) {$3$};
        \node at (-0.3,2) {$2$};
        \node at (-0.3,1) {$1$};

        \node at (7.3,4) {$4$};
        \node at (7.3,3) {$3$};
        \node at (7.3,2) {$2$};
        \node at (7.3,1) {$1$};

        % weight labels
        \node at (1.3,3.5) {$p_1$};
        \node at (2.3,2.5) {$p_2$};
        \node at (3.3,1.5) {$p_3$};
        \node at (4.3,3.5) {$p_4$};
        \node at (5.3,2.5) {$p_5$};
        \node at (6.3,3.5) {$p_6$};
    \end{tikzpicture}
\caption{The graph for $u_1$ when $n=4$, $\boldsymbol{p}$ coordinates} \label{The graph for u_1 when n=4, p coords}
\end{figure}

From the proof of Lemma \ref{lem our Chamber Ansatz minors are monomial} we recall that there is only one family of vertex disjoint paths in the graph for $u_1$ from the set of sources $\{1, \ldots, a\}$ to the set of sinks $\{k+1, \ldots, k+a\}$. Within this family, the weight of the path from $b \in \{1,\ldots,a\}$ to $k+b$ is given by
    $$\prod_{\substack{m= (n-b)+\sum_{j=2}^r (n-j) \\ r=1,\ldots, k}} p_m = \prod_{\substack{r=1,\ldots, k}} p_{s_{r+1}-b+1}.
    $$

Taking the product over this family, that is, the product of the weights of the paths from $\{1, \ldots, a \}$ to $\{k+1, \ldots, k+a\}$, we obtain the desired minor
    $$\Delta^{\{1, \ldots, a \}}_{\{k+1, \ldots, k+a\}}(u_1)
        = \prod_{\substack{r=1,\ldots, k \\ b=1, \ldots, a}} p_{s_{r+1}-b+1}.
    $$

Note that since $u_1 \in U^{\vee}_+$, the minors corresponding to the leftmost chambers and the chamber above the pseudoline arrangement are all equal to $1$.

We now use the Chamber Ansatz to compute the $\frac{1}{m_j}$ coordinates.
For $k=1,\ldots, n-1$, $a=1,\ldots, n-k$ we have
    \begin{align}
    \frac{1}{m_{s_k+a}}
        &= \frac{\left(\prod\limits_{\substack{r=1,\ldots, k-1 \\ b=1, \ldots, a+1}} p_{s_{r+1}-b+1}\right) \left(\prod\limits_{\substack{r=1,\ldots, k \\ b=1, \ldots, a-1}} p_{s_{r+1}-b+1}\right)}{\left(\prod\limits_{\substack{r=1,\ldots, k-1 \\ b=1, \ldots, a}} p_{s_{r+1}-b+1}\right) \left(\prod\limits_{\substack{r=1,\ldots, k \\ b=1, \ldots, a}} p_{s_{r+1}-b+1}\right)} \notag \\
        &= \begin{cases}
        \prod\limits_{\substack{r=1,\ldots, k}} \frac{1}{p_{s_{r+1}-a+1}} &\text{if } k=1, \\
        \frac{\prod\limits_{\substack{r=1,\ldots, k-1}} p_{s_{r+1}-a} }{\prod\limits_{\substack{r=1,\ldots, k}} p_{s_{r+1}-a+1}} &\text{otherwise.}
    \end{cases} \label{eqn m coords in terms of p}
    \end{align}

\end{proof}

We now ready to prove Theorem \ref{thm coord change}, the statement of which we recall here:
\begin{thm*}
For $k=1,\ldots, n-1$, $a=1,\ldots, n-k$ we have
    $$m_{s_k+a} = \begin{cases}
            z_{1+s_{n-a}} &\text{if } k=1 \\
            \frac{z_{k+s_{n-k-a+1}}}{z_{k-1+s_{n-k-a+1}}} &\text{otherwise}
            \end{cases}
    \qquad \text{where} \ \
    s_k := \sum_{j=1}^{k-1}(n-j).
    $$
\end{thm*}

\begin{proof}

To obtain the $\frac{1}{m_{j}}$ coordinates in terms of the $z_j$, we need to compose the two coordinate transformations from Lemmas \ref{lem coord change for u_1 p_i in terms of z_i} and \ref{lem coord change for b m_i in terms of p_i}. To do this we write the factors $p_{s_{r+1}-a}$, $p_{s_{r+1}-a+1}$ of the products in (\ref{eqn m coords in terms of p}) in the form $p_{s_{k'}+{a'}}$ for a suitable pairs $(k',a')$;
\begin{enumerate}
    \item Firstly, given a pair $(k,a)$ with $k>1$, for each $r=1,\ldots, k-1$, we wish to find the pair $(k',a')$, $k'\in \{1, \ldots, n-1\}$, $a' \in \{1,\ldots, n-k'\}$ such that
        $$s_{k'}+a'=s_{r+1}-a.
        $$
    Indeed since
        $$s_{r+1}-a = s_{r}+n-r-a
        $$
    we take $k'=r$, $a'=n-r-a$, noting that $k'$ and $a'$ satisfy the necessary conditions
        $$ \begin{aligned}
        k'&=r \in \{1, \ldots, k-1\}\subseteq \{1, \ldots, n-1\}, \\
        a'&=n-r-a = n-k'-a \in \{k-k', \ldots, n-k'-1\} \subseteq \{1, \ldots, n-k'\}.
        \end{aligned}
        $$

    \item Similarly given a pair $(k,a)$ for any $k$, i.e. $1 \leq k \leq n-1$, for each $r=1,\ldots, k$, we wish to find these pairs $(k',a')$, $k'\in \{1, \ldots, n-1\}$, $a' \in \{1,\ldots, n-k'\}$ such that
        $$s_{k'}+a'=s_{r+1}-a+1 = s_{r}+n-r-a+1.
        $$
    We take $k'=r$, $a'=n-r-a+1$, again noting that $k'$ and $a'$ satisfy the necessary conditions
        $$ \begin{aligned}
        k'&=r \in \{1, \ldots, k\}\subseteq \{1, \ldots, n-1\}, \\
        a'&=n-r-a+1 = n-k'-a+1 \in \{k-k'+1, \ldots, n-k'\} \subseteq \{1, \ldots, n-k'\}.
        \end{aligned}
        $$
\end{enumerate}

In order to evaluate the products in (\ref{eqn m coords in terms of p}) we need to consider what happens to the pair $(k',a')$ when we increase $r$ by $1$.
\begin{enumerate}
    \item In the first case: for $r\in\{1, \ldots, k-2\}$
        $$s_{r+2}-a = s_{r+1}+n-(r+1)-a = s_{k'}+a'+ n-k'-1 = s_{k'+1}+a'-1.
        $$
    Note that $k'+1$ and $a'-1$ satisfy the necessary conditions
        $$ \begin{aligned}
        k'+1 &=r+1 \in \{2, \ldots, k-1\}\subseteq \{1, \ldots, n-1\}, \\
        a'-1 & =n-r-a-1 \in \{k-(k'+1), \ldots, n-(k'+1)-1\} \subseteq \{1, \ldots, n-(k'+1)\}.
        \end{aligned}
        $$

    \item Similarly in the second case: for $r \in\{1, \ldots, k-1\}$
        $$s_{r+2}-a+1 = s_{r+1}+n-(r+1)-a+1 = s_{k'}+a'+ n-k'-1 = s_{k'+1}+a'-1.
        $$
    Again we note that $k'+1$ and $a'-1$ satisfy the necessary conditions
        $$ \begin{aligned}
        k'+1 &=r+1 \in \{2, \ldots, k\}\subseteq \{1, \ldots, n-1\}, \\
        a'-1 &=n-r-a \in \{k-(k'+1)+1, \ldots, n-(k'+1)\} \subseteq \{1, \ldots, n-(k'+1)\}.
        \end{aligned}
        $$
\end{enumerate}

Now recalling
    $$p_{s_k+a} =\begin{cases}
        z_{k+s_{a}} & \text{if } k=1 \\
        \frac{z_{k+s_{a}}}{z_{k-1+s_{a+1}}} & \text{otherwise}
    \end{cases}
    $$
we have, for $r>1$ (so that $k'>1$)
    $$\begin{aligned}
    p_{s_{r+1}-a}p_{s_{r+2}-a}
        &= p_{s_{k'}+a'}p_{s_{k'+1}+a'-1} &\quad \text{where } k'=r,\ a'=n-r-a\\
        &= \frac{z_{k'+s_{a'}}}{z_{k'-1+s_{a'+1}}}\frac{z_{k'+1+s_{a'-1}}}{z_{k'+s_{a'}}} \\
        &= \frac{z_{k'+1+s_{a'-1}}}{z_{k'-1+s_{a'+1}}}.
    \end{aligned}
    $$
Consequently the following product becomes telescopic, so we have
    $$\prod\limits_{\substack{r=1,\ldots, k-1}} p_{s_{r+1}-a} = z_{k-1+s_{n-(k-1)-a}}.
    $$

Similarly in the second case
    $$\begin{aligned}
    p_{s_{r+1}-a+1}p_{s_{r+2}-a+1}
        &= p_{s_{k'}+a'}p_{s_{k'+1}+a'-1} &\quad \text{where } k'=r-1,\ a'=n-r-a+1\\
        &= \frac{z_{k'+1+s_{a'-1}}}{z_{k'-1+s_{a'+1}}}.
    \end{aligned}
    $$
Again we have a telescopic product, giving
    $$\prod\limits_{\substack{r=1,\ldots, k}} p_{s_{r+1}-a+1} = z_{k+s_{n-k-a+1}}.
    $$

Thus we have
    $$\frac{1}{m_{s_k+a}} = \begin{cases}
        \frac{1}{z_{k+s_{n-k-a+1}}} &\text{if } k=1 \\
        \frac{z_{k-1+s_{n-k-a+1}}}{z_{k+s_{n-k-a+1}}} &\text{otherwise}
    \end{cases}
    $$
which gives the desired coordinate change: for $k=1,\ldots, n-1$, $a=1,\ldots, n-k$
    $$m_{s_k+a} = \begin{cases}
            z_{1+s_{n-a}} &\text{if } k=1 \\
            \frac{z_{k+s_{n-k-a+1}}}{z_{k-1+s_{n-k-a+1}}} &\text{otherwise.}
            \end{cases}
    $$
\end{proof}

\section{Givental-type quivers and critical points} \label{sec Givental-type quivers and critical points}

In this section we recall an earlier Landau-Ginzburg model for the full flag variety, defined on a torus by Givental \cite{Givental1997}. We relate our tori from Sections \ref{sec Mirror symmetry for G/B applied to representation theory} and \ref{sec The ideal coordinates}, as well as the superpotential defined in Section \ref{subsec Landau-Ginzburg models}, to Givental's torus and his formulation of the superpotential. We then use this to start describing the critical points of the superpotential. The key result in this section is Theorem \ref{thm crit points, sum at vertex is nu_i}, describing the $\boldsymbol{m}$ coordinates of such a critical point. This formula was conjectured by Konstanze Rietsch and checked in a particular case by Zainab Al-Sultani (\cite{ZainabMastersThesis}).

\subsection{The Givental superpotential}

In this section we recall Givental's construction from \cite{Givental1997}. We begin by considering a quiver, which consists of $n(n+1)/2$ vertices in lower triangular form together with arrows going up and left. We label the vertices with $v_{ij}$ for $1\leq j \leq i \leq n$ in the same way as we would for matrix entries and denote the set of such vertices by $\mathcal{V}$. The vertices $v_{ii}$ are star vertices and all others are dot vertices. We denote the sets of star and dot vertices respectively by
    $$\mathcal{V}^* = \left\{ v_{ii} \ | \ 1 \leq i \leq n \right\}, \quad \mathcal{V}^{\bullet} = \left\{ v_{ij} \ | \ 1 \leq j < i \leq n \right\}.
    $$

The set of arrows of the quiver is denoted $\mathcal{A}=\mathcal{A}_{\mathrm{v}} \cup \mathcal{A}_{\mathrm{h}}$. The vertical arrows $a_{ij}\in \mathcal{A}_{\mathrm{v}}$ are labelled such that $h(a_{ij})=v_{ij}$ where $h(a) \in \mathcal{V}$ denotes the head of the arrow $a$. The horizontal arrows $b_{ij} \in \mathcal{A}_{\mathrm{h}}$ are labelled such that $t(b_{ij})=v_{ij}$ where $t(a)\in \mathcal{V}$ denotes the tail of the arrow $a$. For example when $n=4$ the quiver is given in Figure \ref{fig Quiver for n=4}.
\begin{figure}[ht]
\centering
\begin{minipage}[b]{0.47\textwidth}
    \centering
\begin{tikzpicture}
    %dots and stars
    \node (41) at (0,0) {$\bullet$};
    \node (31) at (0,1.5) {$\bullet$};
    \node (21) at (0,3) {$\bullet$};
    \node (11) at (0,4.5) {$\boldsymbol{*}$};
        \node at (0.2,4.8) {$v_{11}$};
    \node (42) at (1.5,0) {$\bullet$};
    \node (32) at (1.5,1.5) {$\bullet$};
    \node (22) at (1.5,3) {$\boldsymbol{*}$};
        \node at (1.7,3.3) {$v_{22}$};
    \node (43) at (3,0) {$\bullet$};
    \node (33) at (3,1.5) {$\boldsymbol{*}$};
        \node at (3.2,1.8) {$v_{33}$};
    \node (44) at (4.5,0) {$\boldsymbol{*}$};
        \node at (4.7,0.3) {$v_{44}$};

    %arrows
    \draw[->] (41) -- (31);
    \draw[->] (31) -- (21);
    \draw[->] (21) -- (11);
    \draw[->] (42) -- (32);
    \draw[->] (32) -- (22);
    \draw[->] (43) -- (33);

    \draw[->] (22) -- (21);
    \draw[->] (33) -- (32);
    \draw[->] (32) -- (31);
    \draw[->] (44) -- (43);
    \draw[->] (43) -- (42);
    \draw[->] (42) -- (41);

    %dot vertex labels
    \node at (-0.4,3) {$v_{21}$};
    \node at (-0.4,1.5) {$v_{31}$};
    \node at (-0.4,0) {$v_{41}$};

    \node at (1.1,1.8) {$v_{32}$};
    \node at (1.5,-0.3) {$v_{42}$};
    \node at (3,-0.3) {$v_{43}$};

    %arrow labels
    \node at (-0.3,3.75) {$a_{11}$};
    \node at (-0.3,2.25) {$a_{21}$};
    \node at (-0.3,0.75) {$a_{31}$};

    \node at (0.75,2.7) {$b_{22}$};
    \node at (0.75,1.2) {$b_{32}$};
    \node at (0.75,-0.3) {$b_{42}$};

    \node at (1.2,2.25) {$a_{22}$};
    \node at (1.2,0.75) {$a_{32}$};
    \node at (2.7,0.75) {$a_{33}$};

    \node at (2.25,1.2) {$b_{33}$};
    \node at (2.25,-0.3) {$b_{43}$};
    \node at (3.75,-0.3) {$b_{44}$};
\end{tikzpicture}
\caption{Quiver when $n=4$} \label{fig Quiver for n=4}
\end{minipage}
    \hfill
\begin{minipage}[b]{0.47\textwidth}
    \centering
\begin{tikzpicture}
    %dots and stars
    \node (41) at (0,0) {};
    \node (31) at (0,1.5) {};
    \node (42) at (1.5,0) {};
    \node (32) at (1.5,1.5) {};

    %arrows
    \draw[->] (41) -- (31);
    \draw[->] (42) -- (32);

    \draw[->] (32) -- (31);
    \draw[->] (42) -- (41);

    %arrow labels
    \node at (-0.4,0.75) {$a_2$};
    \node at (1.9,0.75) {$a_3$};

    \node at (0.75,1.75) {$a_4$};
    \node at (0.75,-0.3) {$a_1$};
\end{tikzpicture}
\caption{Subquiver for box relations} \label{fig quiver box relations}
\end{minipage}
\end{figure}

We consider three tori which are defined in terms of the quiver, of which we introduce the first two now. The first torus is $(\mathbb{K}^*)^{\mathcal{V}}$ with coordinates $x_v$ for $v \in \mathcal{V}$, we call this the vertex torus. The second torus, $\bar{\mathcal{M}} \subset (\mathbb{K}^*)^{N-n+1}$, corresponds to the arrows of the quiver and so will be called the arrow torus. It is given by
    $$ \bar{\mathcal{M}}:= \left\{ (r_a)_{a \in \mathcal{A}} \in (\mathbb{K}^*)^{\mathcal{A}} \ | \ r_{a_1}r_{a_2}=r_{a_3}r_{a_4} \text{ when } a_1,a_2,a_3,a_4 \text{ form a square as in Figure \ref{fig quiver box relations}} \right\}.
    $$
These two tori are related by the following surjection, given coordinate-wise:
    \begin{equation} \label{eqn quiver tori relation map}
    (\mathbb{K}^*)^{\mathcal{V}} \to \bar{\mathcal{M}}, \quad r_a = \frac{x_{h(a)}}{x_{t(a)}}.
    \end{equation}
Note that we get a point in the preimage of $\boldsymbol{r}_{\mathcal{A}}:=(r_a)_{a \in \mathcal{A}}\in \bar{\mathcal{M}}$ by first setting $x_{v_{nn}}=1$. Then for $v\neq v_{nn}$ we take $x_v=\prod_{a \in P_v} r_a$ where $P_v$ is any path from $v_{nn}$ to $v$. This map is well-defined since the preimage of $(1)_{a \in \mathcal{A}}$ is the set $\{ (c)_{v\in \mathcal{V}} \ | \ c \in \mathbb{K}^* \}$.

An analogy of this surjection is the map
    $$T^{\vee} \to (\mathbb{K}^*)^{n-1}
    $$
given by the simple roots $\alpha^{\vee}_1, \ldots, \alpha^{\vee}_{n-1}$ of $G^{\vee}$.

It will be more convenient for our purposes to use arrow coordinates rather than vertex coordinates. If we were to work with $SL_n$ then the arrow torus, $\bar{\mathcal{M}}$, would suffice, however we wish to work with $GL_n$ and so need to keep track of which fibre of the map (\ref{eqn quiver tori relation map}) we are in. Consequently we now introduce our third torus, $\mathcal{M}$, which we call the quiver torus:
    $$\mathcal{M} := \left\{ ( \boldsymbol{x}_{\mathcal{V}^*} , \boldsymbol{r}_{\mathcal{A}_{\mathrm{v}}}) \in (\mathbb{K}^*)^{\mathcal{V}^*} \times (\mathbb{K}^*)^{\mathcal{A}_{\mathrm{v}}} \right\}
    \hookrightarrow (\mathbb{K}^*)^{\mathcal{V}} \times \bar{\mathcal{M}}.
    $$
The quiver torus is isomorphic to the vertex torus, $(\mathbb{K}^*)^{\mathcal{V}}$, and we will work with these tori interchangeably. This isomorphism is a consequence of the following observations:

We note that if we choose the star vertex coordinate $x_{v_{nn}}$, then we are taking a particular lift of $\boldsymbol{r}_{\mathcal{A}} \in \bar{\mathcal{M}}$ such that in this fibre the map (\ref{eqn quiver tori relation map}) restricts to an isomorphism. Moreover, if we describe the coordinates of the star vertices and vertical arrows, then the coordinates of the horizontal arrows, and thus also of the dot vertices, are given uniquely using (\ref{eqn quiver tori relation map}) and the box relations; namely we use the fact that
    \begin{equation} \label{eqn relation in quiver from stars}
    r_{b_{i+1,i+1}} = \frac{x_{v_{ii}}}{x_{v_{i+1, i+1}}}\frac{1}{r_{a_{ii}}}, \quad i=1, \ldots, n-1
    \end{equation}
together with the relations $r_{a_1} r_{a_2} = r_{a_3} r_{a_4}$ when the arrows $a_1$, $a_2$, $a_3$, $a_4$ form a square as in Figure \ref{fig quiver box relations}.

We now recall the definition of Givental's superpotential. On the vertex torus this can be defined as
    $$\mathcal{F} :(\mathbb{K}^*)^{\mathcal{V}} \to \mathbb{K}, \quad \boldsymbol{x}_{\mathcal{V}} \mapsto \sum_{a \in \mathcal{A}} \frac{x_{h(a)}}{x_{t(a)}}.
    $$
This factors naturally through the arrow torus via (\ref{eqn quiver tori relation map}) and the following map:
    $$\bar{\mathcal{F}} : \bar{\mathcal{M}} \to \mathbb{K}, \quad \boldsymbol{r}_{\mathcal{A}} \mapsto \sum_{a \in \mathcal{A}} r_a.
    $$

We can now define the highest weight and weight maps on the vertex torus. The highest weight map is given by
    $$\kappa :(\mathbb{K}^*)^{\mathcal{V}} \to T^{\vee}, \quad \boldsymbol{x}_{\mathcal{V}} \mapsto (x_{v_{ii}})_{i=1, \ldots, n}.
    $$
The weight map \cite{JoeKim2003} on the vertex torus is defined in two steps; firstly, for $i=1, \ldots,n$ we let $\mathcal{D}_i := \{ v_{i,1}, v_{i+1,2}, \ldots, v_{n, n-i+1} \}$ be the $i$-th diagonal and let
    \begin{equation} \label{eqn xi for wt map defn}
    \Xi_i := \prod_{v\in \mathcal{D}_i} x_v \quad \text{with} \quad \Xi_{n+1}:=1.
    \end{equation}
Then the weight map is given by
    \begin{equation}\label{eqn defn gamma and t}
    \gamma :(\mathbb{K}^*)^{\mathcal{V}}\to T^{\vee}, \quad \boldsymbol{x}_{\mathcal{V}} \mapsto (t_i)_{i=1, \ldots, n} \quad \text{where} \quad t_i = \frac{\Xi_i}{\Xi_{i+1}}.
    \end{equation}

Unlike the superpotential $\mathcal{F}$, the maps $\kappa$ and $\gamma$ do not directly factor through $\bar{\mathcal{M}}$ but thanks to the map (\ref{eqn quiver tori relation map}) we have commutative diagrams:
\begin{equation*}
\begin{tikzcd}
(\mathbb{K}^*)^{\mathcal{V}} \arrow[d] \arrow[r, "\kappa"] & T^{\vee} \arrow[d, "{(\alpha^{\vee}_1, \ldots, \alpha^{\vee}_{n-1})}"] \\
\bar{\mathcal{M}} \arrow[r, "q"] & (\mathbb{K}^*)^{n-1}
\end{tikzcd}
\hspace{1.5cm}
\begin{tikzcd}
(\mathbb{K}^*)^{\mathcal{V}} \arrow[d] \arrow[r, "\gamma"] & T^{\vee} \arrow[d, "{(\alpha^{\vee}_1, \ldots, \alpha^{\vee}_{n-1})}"] \\
\bar{\mathcal{M}} \arrow[r] & (\mathbb{K}^*)^{n-1}
\end{tikzcd}
\end{equation*}
The maps along the bottom making the diagrams commute exist and are unique. For example, the map on $\bar{\mathcal{M}}$ corresponding to $\kappa$ is the map called $q$ defined by Givental as follows:
    $$q : \bar{\mathcal{M}} \to (\mathbb{K}^*)^{n-1}, \quad \boldsymbol{r}_{\mathcal{A}} \mapsto \left( r_{a_{11}}r_{b_{22}}, \, \ldots, \, r_{a_{n-1, n-1}}r_{b_{nn}} \right).
    $$

\begin{rem}
In Givental's version of (Fano) mirror symmetry, the arrow torus $\bar{\mathcal{M}}$ (taken over $\mathbb{C}$), is viewed as a family of varieties via the map $q$. Each fibre $\bar{\mathcal{M}}_{\boldsymbol{q}}$, being a torus, comes equipped with a natural holomorphic volume form $\omega_{\boldsymbol{q}}$. He proves a version of mirror symmetry relating the A-model connection, built out of Gromov-Witten invariants of the flag variety, to period integrals $S(\boldsymbol{q})$ on the family $(\bar{\mathcal{M}}_{\boldsymbol{q}}, \bar{\mathcal{F}}_{\boldsymbol{q}}, \omega_{\boldsymbol{q}} )$, defined using the superpotential:
    $$S(\boldsymbol{q}):= \int_{\Gamma} e^{\bar{\mathcal{F}}_{\boldsymbol{q}}} \omega_{\boldsymbol{q}}.
    $$
\end{rem}

\subsection[The quiver torus as another toric chart on Z]{The quiver torus as another toric chart on $Z$} \label{subsec The quiver torus as another toric chart on Z}

In this section we recall a map $\mathcal{M} \to Z$ from \cite{Rietsch2008} which allows us to relate the Givental superpotential on the arrow torus, as well as the highest weight and weight maps on the vertex torus, to their analogues on $Z$. In particular our toric charts from the previous sections factor through this map.

A nice way to describe how the string and ideal toric charts factor through $\mathcal{M}$ is by decorating the arrows and vertices of the quiver. Indeed, using our previous work relating the string and ideal coordinate systems, we may easily describe a monomial map from the torus $T^{\vee}\times (\mathbb{K}^*)^N$ of string coordinates to the quiver torus $\mathcal{M}$. Composing this with the map $\mathcal{M} \to Z$ will recover our string toric chart.
We then use the coordinate change given in Theorem \ref{thm coord change} to decorate the quiver with the ideal coordinates.

In order to make sense of the way the string toric chart factors through $\mathcal{M}$, we choose an ordering of the $N$ vertical arrow coordinates. Starting at the lower left corner of the quiver and moving up each column in succession we obtain
    $$\boldsymbol{r}_{\mathcal{A}_{\mathrm{v}}} := \left(r_{a_{n-1, 1}}, r_{a_{n-2, 1}}, \ldots, r_{a_{1, 1}}, r_{a_{n-1, 2}}, r_{a_{n-2, 2}}, r_{a_{2, 2}}, \ldots, r_{a_{n-1, n-2}}, r_{a_{n-2, n-2}}, r_{a_{n-1, n-1}}\right).
    $$
Similarly, we give an ordering of the star vertex coordinates:
    $$\boldsymbol{x}_{\mathcal{V}^*}:=\left(x_{v_{11}}, x_{v_{22}}, \ldots, x_{v_{nn}}\right).
    $$
Then for the reduced expression
    $$\mathbf{i}'_0 = (i'_1, \ldots, i'_N) := (n-1, n-2, \ldots, 1, n-1, n-2, \ldots, 2, \ldots, n-1, n-2, n-1)
    $$
we have a map $\theta_{\mathbf{i}'_0}:\mathcal{M} \to Z$ given by
    $$\left(\boldsymbol{x}_{\mathcal{V}^*} , \boldsymbol{r}_{\mathcal{A}_{\mathrm{v}}} \right)
    \mapsto
    \Phi\left(\kappa_{\vert_{\mathcal{V}^*}}\left(\boldsymbol{x}_{\mathcal{V}^*}\right) , \mathbf{x}^{\vee}_{\mathbf{i}'_0}\left(\boldsymbol{r}_{\mathcal{A}_{\mathrm{v}}}\right)\right)
    = \mathbf{x}^{\vee}_{\mathbf{i}'_0}\left(\boldsymbol{r}_{\mathcal{A}_{\mathrm{v}}}\right) \kappa_{\vert_{\mathcal{V}^*}}\left(\boldsymbol{x}_{\mathcal{V}^*}\right) \bar{w}_0 u_2
    $$
where $\kappa_{\vert_{\mathcal{V}^*}}$ is the restriction of $\tilde{\kappa}$ to the star vertex coordinates, $u_2 \in U^{\vee}$ is the unique element such that $\theta_{\mathbf{i}'_0}\left(\boldsymbol{r}_{\mathcal{A}_{\mathrm{v}}}\right) \in Z$ and $\mathbf{x}^{\vee}_{\mathbf{i}}$, $\Phi$ are the maps from Section \ref{subsec The string coordinates}.

Next we recall the following definitions, also from Section \ref{subsec The string coordinates}:
    $$b:=u_1d\bar{w}_0 u_2 \in Z, \quad u_1:=\iota(\eta^{w_0,e}(u)), \quad u:=\mathbf{x}^{\vee}_{-\mathbf{i}_0}(\boldsymbol{z}),
    $$
together with the fact that by Lemmas \ref{lem form of u_1 and b factorisations} and \ref{lem coord change for u_1 p_i in terms of z_i} we may factorise $u_1$ as
    $$
    u_1 = \mathbf{x}_{i'_1}^{\vee}(p_1) \cdots \mathbf{x}_{i'_N}^{\vee}(p_N) \\
    $$
where, for $k=1,\ldots, n-1$, $a=1,\ldots, n-k$, we have
    $$p_{s_k+a} =\begin{cases}
        z_{1+s_{a}} & \text{if } k=1 \\
        \frac{z_{k+s_{a}}}{z_{k-1+s_{a+1}}} & \text{otherwise}
        \end{cases}
    \qquad \text{where} \ \
    s_k := \sum_{j=1}^{k-1}(n-j).
    $$
Thus taking $\boldsymbol{r}_{\mathcal{A}_{\mathrm{v}}}=(p_1, \ldots, p_N)$ we see that $\mathbf{x}^{\vee}_{\mathbf{i}'_0}\left(\boldsymbol{r}_{\mathcal{A}_{\mathrm{v}}}\right)=u_1$ and letting the star vertex coordinates be given by
    $$x_{v_{ii}}=d_i, \quad i=1,\ldots, n
    $$
we obtain $\kappa_{\vert_{\mathcal{V}^*}}\left(\boldsymbol{x}_{\mathcal{V}^*}\right) = d$. This quiver decoration in dimension $4$ is given in Figure \ref{fig Quiver decoration, n=4, z coordinates}.

Due to the relations in the quiver, namely (\ref{eqn relation in quiver from stars}) and the box relations, this decoration extends to all the vertices and arrows of the quiver. In particular, we can extend $\theta_{\mathbf{i}'_0}$ to a map
    $$\bar{\theta}_{\mathbf{i}'_0}: (\mathbb{K}^*)^{\mathcal{V}} \times \bar{\mathcal{M}} \to Z
    $$
given by first taking the projection onto $\mathcal{M}$ and then applying $\theta_{\mathbf{i}'_0}$. Thus decorating the quiver in this way and then applying the map $\bar{\theta}_{\mathbf{i}'_0}$ (or $\theta_{\mathbf{i}'_0}$) gives the string toric chart factored through $(\mathbb{K}^*)^{\mathcal{V}} \times \bar{\mathcal{M}}$ (respectively $\mathcal{M}$), that is
    $$
    T^{\vee} \times (\mathbb{K}^*)^N \to (\mathbb{K}^*)^{\mathcal{V}} \times \bar{\mathcal{M}} \to Z.
    $$

We summarise the results from this section:
\begin{lem}[{\cite[Theorem 9.2 and Lemma 9.3]{Rietsch2008}}]
With the above notation, we have the following:
    $$\mathcal{W} \circ \bar{\theta}_{\mathbf{i}'_0} = \mathcal{F} \circ \mathrm{pr}, \quad \mathrm{hw} \circ \bar{\theta}_{\mathbf{i}'_0} = \kappa \circ \mathrm{pr}, \quad \mathrm{wt} \circ \bar{\theta}_{\mathbf{i}'_0} = \gamma \circ \mathrm{pr}.
    $$
where $\mathrm{pr}$ is the projection of $(\mathbb{K}^*)^{\mathcal{V}} \times \bar{\mathcal{M}}$ onto the first factor.
\end{lem}

\begin{figure}%[hb]
\centering
\begin{minipage}[b]{0.47\textwidth}
    \centering
\begin{tikzpicture}[scale=0.85]
    %dots and stars
    \node (41) at (0,0) {$\bullet$};
    \node (31) at (0,2) {$\bullet$};
    \node (21) at (0,4) {$\bullet$};
    \node (11) at (0,6) {$\boldsymbol{*}$};
        \node at (0.3,6.3) {$d_{1}$};
    \node (42) at (2,0) {$\bullet$};
    \node (32) at (2,2) {$\bullet$};
    \node (22) at (2,4) {$\boldsymbol{*}$};
        \node at (2.3,4.3) {$d_{2}$};
    \node (43) at (4,0) {$\bullet$};
    \node (33) at (4,2) {$\boldsymbol{*}$};
        \node at (4.3,2.3) {$d_{3}$};
    \node (44) at (6,0) {$\boldsymbol{*}$};
        \node at (6.3,0.3) {$d_{4}$};

    %arrows
    \draw[->] (41) -- (31);
    \draw[->] (31) -- (21);
    \draw[->] (21) -- (11);
    \draw[->] (42) -- (32);
    \draw[->] (32) -- (22);
    \draw[->] (43) -- (33);

    \draw[->] (22) -- (21);
    \draw[->] (33) -- (32);
    \draw[->] (32) -- (31);
    \draw[->] (44) -- (43);
    \draw[->] (43) -- (42);
    \draw[->] (42) -- (41);

    %arrow labels
    \node at (-0.3,4.8) {\scriptsize{$z_6$}};
    \node at (-0.3,2.8) {\scriptsize{$z_4$}};
    \node at (-0.3,0.8) {\scriptsize{$z_1$}};

    \node at (1.7,2.8) {\scriptsize{$\frac{z_5}{z_6}$}};
    \node at (1.7,0.8) {\scriptsize{$\frac{z_2}{z_4}$}};
    \node at (3.7,0.8) {\scriptsize{$\frac{z_3}{z_5}$}};

    %arrow labels
    \node at (1,3.6) {\scriptsize{$\frac{d_1}{d_2} \frac{1}{z_6}$}};
    \node at (1,1.6) {\scriptsize{$\frac{d_1}{d_2} \frac{z_5}{z_4z_6^2}$}};
    \node at (1,-0.4) {\scriptsize{$\frac{d_1}{d_2} \frac{z_2z_5}{z_1z_4^2z_6^2}$}};

    \node at (3,1.6) {\scriptsize{$\frac{d_2}{d_3} \frac{z_6}{z_5}$}};
    \node at (3,-0.4) {\scriptsize{$\frac{d_2}{d_3} \frac{z_3z_4z_6}{z_2z_5^2}$}};
    \node at (5,-0.4) {\scriptsize{$\frac{d_3}{d_4} \frac{z_5}{z_3}$}};
\end{tikzpicture}
\caption{Quiver decoration when $n=4$, \\ string coordinates} \label{fig Quiver decoration, n=4, z coordinates}
\end{minipage}\hfill
\begin{minipage}[b]{0.47\textwidth}
    \centering
\begin{tikzpicture}[scale=0.85]
    %dots and stars
    \node (41) at (0,0) {$\bullet$};
    \node (31) at (0,2) {$\bullet$};
    \node (21) at (0,4) {$\bullet$};
    \node (11) at (0,6) {$\boldsymbol{*}$};
        \node at (0.3,6.3) {$d_1$};
    \node (42) at (2,0) {$\bullet$};
    \node (32) at (2,2) {$\bullet$};
    \node (22) at (2,4) {$\boldsymbol{*}$};
        \node at (2.3,4.3) {$d_2$};
    \node (43) at (4,0) {$\bullet$};
    \node (33) at (4,2) {$\boldsymbol{*}$};
        \node at (4.3,2.3) {$d_3$};
    \node (44) at (6,0) {$\boldsymbol{*}$};
        \node at (6.3,0.3) {$d_4$};

    %arrows
    \draw[->] (41) -- (31);
    \draw[->] (31) -- (21);
    \draw[->] (21) -- (11);
    \draw[->] (42) -- (32);
    \draw[->] (32) -- (22);
    \draw[->] (43) -- (33);

    \draw[->] (22) -- (21);
    \draw[->] (33) -- (32);
    \draw[->] (32) -- (31);
    \draw[->] (44) -- (43);
    \draw[->] (43) -- (42);
    \draw[->] (42) -- (41);

    %arrow labels
    \node at (-0.3,0.8) {\scriptsize{$m_3$}};
    \node at (-0.3,2.8) {\scriptsize{$m_2$}};
    \node at (-0.3,4.8) {\scriptsize{$m_1$}};

    \node at (1.5,0.8) {\scriptsize{$\frac{m_3m_5}{m_2}$}};
    \node at (1.5,2.8) {\scriptsize{$\frac{m_2m_4}{m_1}$}};
    \node at (3.3,0.8) {\scriptsize{$\frac{m_3m_5m_6}{m_2m_4}$}};

    \node at (1,-0.4) {\scriptsize{$\frac{d_1}{d_2} \frac{m_4m_5}{{m_1}^2m_2}$}};
    \node at (3,-0.4) {\scriptsize{$\frac{d_2}{d_3} \frac{m_1m_6}{m_2{m_4}^2}$}};
    \node at (5,-0.4) {\scriptsize{$\frac{d_3}{d_4} \frac{m_2m_4}{m_3m_5m_6}$}};
    \node at (1,1.6) {\scriptsize{$\frac{d_1}{d_2} \frac{m_4}{{m_1}^2}$}};
    \node at (3,1.6) {\scriptsize{$\frac{d_2}{d_3} \frac{m_1}{m_2m_4}$}};
    \node at (1,3.6) {\scriptsize{$\frac{d_1}{d_2} \frac{1}{m_1}$}};
\end{tikzpicture}
\caption{Quiver decoration when $n=4$, \\ ideal coordinates} \label{fig Quiver decoration, n=4, m coordinates}
\end{minipage}
\end{figure}

To decorate the quiver with the ideal coordinates, we begin by recalling Theorem \ref{thm coord change}, namely that for $k=1,\ldots, n-1$, $a=1,\ldots, n-k$ we have
    $$m_{s_k+a} = \begin{cases}
            z_{1+s_{n-a}} &\text{if } k=1, \\
            \frac{z_{k+s_{n-k-a+1}}}{z_{k-1+s_{n-k-a+1}}} &\text{otherwise}.
            \end{cases}
    $$

In the above quiver decoration we have $\boldsymbol{r}_{\mathcal{A}}=(p_1, \ldots, p_N)$, that is
    $$r_{a_{ij}}=p_{s_j+n-i}=\begin{cases}
        z_{1+s_{n-i}} & \text{if } j=1, \\
        \frac{z_{j+s_{n-i}}}{z_{j-1+s_{n-i+1}}} & \text{otherwise}.
    \end{cases}
    $$
We notice that if $j=1$ then
    $$r_{a_{i1}}=p_{s_1+n-i}=z_{1+s_{n-i}}=m_{s_1+i}.
    $$
We also want to write $r_{a_{ij}}$ in terms of the $m_{s_k+a}$ coordinates when $j>2$. To do this we first note that for $k=2$ we have
    $$m_{s_k+a}m_{s_{k-1}+a+1} = \frac{z_{k+s_{n-k-a+1}}}{z_{k-1+s_{n-k-a+1}}}z_{k-1+s_{n-(k-1)-(a+1)+1}}
    = z_{k+s_{n-k-a+1}}.
    $$
Similarly for $j>2$ we have
    $$m_{s_k+a}m_{s_{k-1}+a+1}
        = \frac{z_{k+s_{n-k-a+1}}}{z_{k-1+s_{n-k-a+1}}}\frac{z_{k-1+s_{n-k-a+1}}}{z_{k-2+s_{n-k-a+1}}} \\
        = \frac{z_{k+s_{n-k-a+1}}}{z_{k-2+s_{n-k-a+1}}}.
    $$
Thus the following product is telescopic:
    $$\prod_{r=0}^{k-1} m_{s_{k-r}+a+r} = z_{k+s_{n-k-a+1}}.
    $$
Now taking quotients of these products allows us to write the vertical arrow coordinates, $r_{a_{ij}}=p_{s_j+n-i}$, in terms of the $m_j$; for $j>2$ we have
    $$r_{a_{ij}}=p_{s_j+n-i}
        =\frac{z_{j+s_{n-i}}}{z_{j-1+s_{n-i+1}}}
        =\frac{z_{j+s_{n-j-(i-j+1)+1}}}{z_{j-1+s_{n-(j-1)-(i-j+1)+1}}}
        =\frac{\prod_{r=0}^{j-1} m_{s_{j-r}+i-j+1+r}}{\prod_{r=0}^{j-2} m_{s_{j-1-r}+i-j+1+r}}.
    $$
This doesn't seem particularly helpful at first glance, however it leads to an iterative description of the new quiver decoration which will be very useful in the proof of the main theorem in Section \ref{sec Givental-type quivers and critical points}, Theorem \ref{thm crit points, sum at vertex is nu_i}. Denoting the numerator of $r_{a_{ij}}$ by $n(r_{a_{ij}})$, for $j=1, \ldots, n-1$ we have
    \begin{equation} \label{eqn rai,j+1 in terms of m's}
    r_{a_{i,j+1}}
        =\frac{\prod_{r=0}^{j} m_{s_{j+1-r}+i-j+r}}{\prod_{r=0}^{j-1} m_{s_{j-r}+i-j+r}}
        =m_{s_{j+1}+i-j}\frac{\prod_{r=0}^{j-1} m_{s_{j-r}+i+1-j+r}}{\prod_{r=0}^{j-1} m_{s_{j-r}+i-j+r}}
        =m_{s_{j+1}+i-j}\frac{n(r_{a_{ij}})}{n(r_{a_{i-1,j}})}. \\
    \end{equation}
For example, in dimension $4$ the quiver decoration is given in Figure \ref{fig Quiver decoration, n=4, m coordinates}.

\subsection{Critical points of the superpotential}

We begin by recalling the highest weight map on the vertex torus:
    $$\kappa :(\mathbb{K}^*)^{\mathcal{V}} \to T^{\vee}, \quad \boldsymbol{x}_{\mathcal{V}} \mapsto \left(x_{v_{ii}}\right)_{i=1, \ldots, n}.
    $$
Now in the fibre over $d\in T^{\vee}$ we have
    $$x_{v_{ii}}=d_i \quad \text{for} \ i=1,\ldots, n.
    $$
The remaining $x_v$ for $v \in \mathcal{V}^{\bullet}$ form a coordinate system on this fibre. In particular we can use these coordinates to compute critical points of the superpotential, as follows:
    $$x_v \frac{\partial \mathcal{F}}{\partial x_v} = \sum_{a:h(a)=v} \frac{x_{h(a)}}{x_{t(a)}} - \sum_{a:t(a)=v} \frac{x_{h(a)}}{x_{t(a)}}.
    $$
Thus the critical point conditions are
    $$\sum_{a:h(a)=v} \frac{x_{h(a)}}{x_{t(a)}} = \sum_{a:t(a)=v} \frac{x_{h(a)}}{x_{t(a)}} \quad \text{for} \ v \in \mathcal{V}^{\bullet}.
    $$
Since we favour working with arrow coordinates, we rewrite these equations:
    \begin{equation}\label{eqn crit pt conditions} \sum_{a:h(a)=v} r_a = \sum_{a:t(a)=v} r_a \quad \text{for} \ v \in \mathcal{V}^{\bullet}.
    \end{equation}
This means we can now use the arrow coordinates in the quiver, and thus the ideal coordinates, to give simple descriptions of both the superpotential and the defining equations of its critical points. In fact, using the quiver decoration in terms of the ideal coordinates, we can take this a step further:

\begin{thm} \label{thm crit points, sum at vertex is nu_i}
If the critical point conditions hold at every dot vertex $v \in \mathcal{V}^{\bullet}$, then the sum of the outgoing arrow coordinates at each dot vertex $v_{ik}$ is given in terms of the ideal coordinates by
    $$ \varpi(v_{ik}) := \sum_{a:t(a)=v_{ik}} r_a = m_{s_k +i-k}.
    $$
\end{thm}

\begin{proof}
By construction of the quiver labelling there is only one outgoing arrow at each $v_{i1}$ for $i=2, \ldots, n$, namely $a_{i-1,1}$, and indeed $r_{a_{i-1,1}}= m_{i-1} = m_{s_1+i-1}$. Since we have the desired property for the vertices $v_{i1}$, $i=2, \ldots, n-1$, we proceed by an inductive argument increasing both vertex subscripts simultaneously.

We consider the subquiver given in Figure \ref{fig Critical point conditions in the quiver} and suppose the sum of outgoing arrows at $v_{ik}$ is $m_{s_k +i-k}$. Then by the critical point condition at this vertex we have $m_{s_k +i-k} = r_{a_{ik}}+r_{b_{i,k+1}}$.
\begin{figure}[hb]
\centering
\begin{tikzpicture}[scale=0.9]
    %dots and stars
    \node (31) at (0,0) {$\bullet$};
    \node (21) at (0,1.5) {$\bullet$};
    \node (32) at (1.5,0) {$\bullet$};
    \node (22) at (1.5,1.5) {$\bullet$};
    \node (11) at (0,3) {$\bullet$};
    \node (20) at (-1.5,1.5) {$\bullet$};

    %arrows
    \draw[->] (21) -- (11);
    \draw[->] (31) -- (21);
    \draw[->] (32) -- (22);

    \draw[->] (21) -- (20);
    \draw[->] (22) -- (21);
    \draw[->] (32) -- (31);

    %arrow labels
    \node at (-0.48,2.25) {\scriptsize{$a_{i-1,k}$}};
    \node at (-0.3,0.75) {\scriptsize{$a_{ik}$}};
    \node at (2,0.75) {\scriptsize{$a_{i,k+1}$}};

    \node at (0.75,-0.3) {\scriptsize{$b_{i+1,k+1}$}};
    \node at (0.75,1.7) {\scriptsize{$b_{i,k+1}$}};
    \node at (-0.75,1.7) {\scriptsize{$b_{ik}$}};

    \node at (-0.55,0) {\scriptsize{$v_{i+1,k}$}};
    \node at (0.35,1.25) {\scriptsize{$v_{ik}$}};
    \node at (2.19,0) {\scriptsize{$v_{i+1,k+1}$}};
    \node at (2.07,1.5) {\scriptsize{$v_{i,k+1}$}};
\end{tikzpicture}
\caption{Critical point conditions in the quiver} \label{fig Critical point conditions in the quiver}
\end{figure}
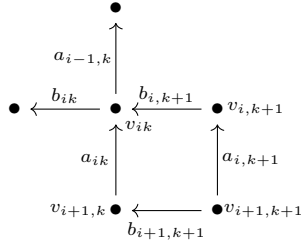

By the definition of the arrow coordinates we have
    $$r_{a_{i-1,k}} = m_{s_k +i-k} \frac{n(r_{a_{i-1,k-1}})}{n(r_{a_{i-2,k-1}})},  \quad
    r_{a_{i,k+1}} = m_{s_{k+1} +i-k} \frac{n(r_{a_{ik}})}{n(r_{a_{i-1,k}})}, \quad
    r_{a_{ik}} = m_{s_k +i+1-k} \frac{n(r_{a_{i,k-1}})}{n(r_{a_{i-1,k-1}})}.
    $$
In particular we see that
    \begin{equation} \label{eqn quotient of arrow coords, crit point in quiver}
    \frac{r_{a_{i,k+1}}}{r_{a_{ik}}} = m_{s_{k+1} +i-k} \frac{n(r_{a_{i-1,k-1}})}{n(r_{a_{i-1,k}})}.
    \end{equation}
The sum of outgoing arrows at the vertex $v_{i+1,k+1}$ is given by
    $$\begin{aligned}
    \varpi(v_{i+1,k+1}) :=\sum_{a:t(a)=v_{i+1,k+1}} r_a &=  r_{a_{i,k+1}} + r_{b_{i+1,k+1}} &\\
        &= r_{a_{i,k+1}} + \frac{r_{a_{i,k+1}}r_{b_{i,k+1}}}{r_{a_{ik}}} & \begin{aligned}[t]\text{by the box relation} \\ r_{a_{ik}}r_{b_{i+1,k+1}} = r_{a_{i,k+1}}r_{b_{i,k+1}} \end{aligned} \\
        &= r_{a_{i,k+1}} + \frac{r_{a_{i,k+1}}}{r_{a_{ik}}}(m_{s_k +i-k} - r_{a_{ik}}) & \text{by the inductive hypothesis}\\
        &= \frac{r_{a_{i,k+1}}}{r_{a_{ik}}}m_{s_k +i-k} & \\
        &= \frac{m_{s_{k+1} +i-k} n(r_{a_{i-1,k-1}})}{n(r_{a_{i-1,k}})} m_{s_k +i-k} & \text{using} \ (\ref{eqn quotient of arrow coords, crit point in quiver}) \\
        &= m_{s_{k+1} +i-k} & \text{by definition of} \ n(r_{a_{i-1,k}}) \\
        &= m_{s_{k+1}+(i+1)-(k+1)}.
    \end{aligned}$$
\end{proof}

\begin{rem}
At first glance, in the above theorem we seem to have lost the information about the highest weight element $d \in T^{\vee}$. However at a critical point this information can be partially recovered from the $\boldsymbol{m}$-coordinates; considering the dot vertices on the bottom wall of the quiver, for each $j=1, \ldots, n-1$ we have
    $$\begin{aligned}
    r_{b_{n, j+1}} &= \varpi(v_{nj}) & \text{by the critical point conditions} \\
        &= m_{s_j+n-j} & \text{by Theorem \ref{thm crit points, sum at vertex is nu_i}}.
    \end{aligned}
    $$
By the quiver decoration, each arrow coordinate $r_{b_{n, j+1}}$ is given by
    $$r_{b_{n, j+1}} = \frac{d_j}{d_{j+1}} \frac{\prod_{i \in I_{j_1}} m_i}{\prod_{i \in I_{j_2}} m_i}
    $$
for some multisets $I_{j_1}, I_{j_2}$ of the integers $1, \ldots, N$. Thus we have
    $$\alpha^{\vee}_j(d)=\frac{d_j}{d_{j+1}} = m_{s_j+n-j} \frac{\prod_{i \in I_{j_2}} m_i}{\prod_{i \in I_{j_1}} m_i}, \quad j=1, \ldots, n-1.
    $$
\end{rem}

We complete this section by tying together the quiver, critical points and our interest in the form of the weight matrix from Sections \ref{subsec The form of the weight matrix} and \ref{sec The ideal coordinates}. Namely it is natural to ask what happens to the weight matrix at critical points.

\begin{prop} \label{prop weight at crit point non trop}
At a critical point in the fibre over $d\in T^{\vee}$, the weight matrix is an $n\times n$ matrix $\mathrm{diag}(c, \ldots, c)$ where
    $$c^n = \prod_{i=1}^n d_i.
    $$
\end{prop}

In order to prove this we require the following lemma:
\begin{lem}[{\cite[Lemma 5.9]{Judd2018}}] \label{lem Jamie's lemma 5.9}
Suppose we have a quiver like the one given in Figure \ref{fig Diagonal subquiver}. We attach a variable $r_a$ to each arrow such that the box relations $r_{a_1}r_{a_2}=r_{a_3}r_{a_4}$ hold whenever $a_1, a_2, a_3, a_4$ form a square (see Figure \ref{fig Arrows forming a box}) and the critical point conditions hold at each black vertex. For each box $B_j$, $1\leq j \leq t$, let $O_j = r_{a_1}r_{a_3}$ and $I_j=r_{a_2}r_{a_4}$. Additionally let $K_t=\prod r_a$ where the product is over a (any) path from $v_t$ to $v_0$. Then we have
    $$\prod_{j=1}^t O_j \frac{r_{a_{\mathrm{out}}}}{r_{a_{\mathrm{in}}}} = K_t \quad \text{and} \quad \prod_{j=1}^t I_j \frac{r_{a_{\mathrm{in}}}}{r_{a_{\mathrm{out}}}} = K_t .
    $$
Note this agrees with $\prod_{j=1}^t O_j \prod_{j=1}^t I_j = K_t^2$.
\end{lem}

\begin{figure}[ht]
\centering
\begin{minipage}[b]{0.47\textwidth}
    \centering
    \begin{tikzpicture}[scale=0.9]
        %dots and stars
        \node (11) at (0,6) {$\boldsymbol{\circ}$};
            \node at (0.5,6) {$v_{-1}$};
        \node (22) at (1,5) {$\bullet$};
            \node at (0.6,4.9) {$v_0$};
        \node (33) at (2,4) {$\bullet$};
        \node (44) at (3,3) {$\bullet$};
            \node at (3.5,2.6) {$\ddots$};
        \node (55) at (4,2) {$\bullet$};
        \node (66) at (5,1) {$\bullet$};
            \node at (4.9,0.7) {$v_t$};
        \node (77) at (6,0) {$\boldsymbol{\circ}$};
            \node at (6.5,0) {$v_{t+1}$};

        \node (32) at (1,4) {$\boldsymbol{\circ}$};
        \node (43) at (2,3) {$\boldsymbol{\circ}$};
        \node (65) at (4,1) {$\boldsymbol{\circ}$};

        \node (23) at (2,5) {$\boldsymbol{\circ}$};
        \node (34) at (3,4) {$\boldsymbol{\circ}$};
        \node (56) at (5,2) {$\boldsymbol{\circ}$};

        %arrows
        \draw[->] (22) -- (11);
        \draw[->] (77) -- (66);

        \draw[->] (33) -- (32);
        \draw[->] (32) -- (22);
        \draw[->] (33) -- (23);
        \draw[->] (23) -- (22);

        \draw[->] (44) -- (43);
        \draw[->] (43) -- (33);
        \draw[->] (44) -- (34);
        \draw[->] (34) -- (33);

        \draw[->] (66) -- (65);
        \draw[->] (65) -- (55);
        \draw[->] (66) -- (56);
        \draw[->] (56) -- (55);

        %box labels
        \node at (1.5,4.5) {$B_1$};
        \node at (2.5,3.5) {$B_2$};
        \node at (4.5,1.5) {$B_t$};

        %arrow labels
        \node at (1,5.55) {$a_{\mathrm{out}}$};
        \node at (5.9,0.55) {$a_{\mathrm{in}}$};
    \end{tikzpicture}
\caption{Diagonal subquiver} \label{fig Diagonal subquiver}
\end{minipage} \hfill
\begin{minipage}[b]{0.47\textwidth}
    \centering
    \begin{tikzpicture}
        %dots and stars
        \node (41) at (0,0) {};
        \node (31) at (0,1) {};
        \node (42) at (1,0) {};
        \node (32) at (1,1) {};

        %arrows
        \draw[->] (41) -- (31);
        \draw[->] (42) -- (32);

        \draw[->] (32) -- (31);
        \draw[->] (42) -- (41);

        %arrow labels
        \node at (-0.4,0.5) {$a_2$};
        \node at (1.4,0.5) {$a_3$};

        \node at (0.5,1.2) {$a_4$};
        \node at (0.5,-0.3) {$a_1$};
    \end{tikzpicture}
\caption{Arrows forming a box} \label{fig Arrows forming a box}
\end{minipage}
\end{figure}

\begin{proof}[Proof of Proposition {\ref{prop weight at crit point non trop}}]
We consider the set of all arrows with either head or tail on the $i$th diagonal $\mathcal{D}_i$. These arrows form a subquiver like in Figure \ref{fig Diagonal subquiver}, with $v_t=v_{n, n-i+1}$. Moreover we have
    $$K_t=\frac{x_{v_0}}{x_{v_t}}, \quad r_{a_{\mathrm{in}}} = \frac{x_{v_t}}{x_{v_{t+1}}} , \quad r_{a_{\mathrm{out}}} = \frac{x_{v_{-1}}}{x_{v_0}}, \quad \text{so} \quad \frac{1}{K_t} \frac{r_{a_{\mathrm{out}}}}{r_{a_{\mathrm{in}}}} = \frac{x_{v_{-1}}x_{v_{t+1}}}{x_{v_0}^2}.
    $$
Thus, using Lemma \ref{lem Jamie's lemma 5.9}, we see that
    $$\frac{t_{i-1}}{t_i}=\frac{\Xi_{i+1}\Xi_{i-1}}{\Xi_i^2} = \frac{x_{v_{-1}}x_{v_{t+1}}}{x_{v_0}^2} \prod O_j = \frac{1}{K_t} \prod_{j=1}^t O_j \frac{r_{a_{\mathrm{out}}}}{r_{a_{\mathrm{in}}}} = 1.
    $$

So at a critical point, the weight matrix $\mathrm{diag}(t_1, t_2, \ldots, t_n)$ is given by $\mathrm{diag}(c,c, \ldots, c)$ for some $c$. By taking the determinant we obtain
    $$\prod_{i=1}^n t_i = \prod_{i=1}^n \frac{\Xi_i}{\Xi_{i+1}} = \frac{\Xi_1}{\Xi_{n+1}} = \Xi_1
    $$
recalling $\Xi_{n+1} =1$ by definition. This gives the desired value of $c$ as follows:
    $$c^n = \prod_{i=1}^n t_i = \Xi_1 = \prod_{i=1}^n d_i.
    $$
\end{proof}

\section{The tropical viewpoint} \label{sec The tropical viewpoint}

In this section we recall how, by tropicalisation, we can use the superpotential to obtain polytopes associated to a given highest weight. These polytopes depend on the choice of positive toric chart. The goal in this section is to describe the polytope we get from the ideal coordinates. Additionally we show that for each choice of highest weight, the associated critical point of the superpotential gives rise to a point inside this polytope, which is Judd's tropical critical point \cite{Judd2018} and has a beautiful description in terms of so called ideal fillings.

\subsection{The basics of tropicalisation} \label{subsec The basics of tropicalisation}

In this section we explain the concept of tropicalisation, following an original construction due to Lusztig \cite{Lusztig1994}. In order to do this we work over the field of Generalised Puiseux series, which we will denote by $\mathbf{K}$.

A generalised Puiseux series in a variable $t$ is a series with an exponent set $(\mu_k) = (\mu_0, \mu_1, \mu_2, \ldots ) \in \mathbb{R}$ which is strictly monotone and either finite, or countable and tending to infinity. That is,
    $$(\mu_k) \in \mathrm{MonSeq} = \left\{ A \subset \mathbb{R} \ | \ \mathrm{Cardinality}(A\cap \mathbb{R}_{\leq x}) < \infty \ \text{for arbitrarily large} \ x \in \mathbb{R} \right\} .
    $$
Thus we have
    $$\mathbf{K} = \left\{ c(t) = \sum_{(\mu_k) \in \mathrm{MonSeq}} c_{\mu_k}t^{\mu_k} \ | \ c_{\mu_k} \in \mathbb{C} \right\}.
    $$
The positive part of the field $\mathbf{K}$ is given by
    $$\mathbf{K}_{>0}:= \left\{ c(t)\in \mathbf{K} \ | \ c_{\mu_0} \in \mathbb{R}_{>0}\right\}
    $$
where we may assume that the lowest order term has a non-zero coefficient.

Given a torus $\mathcal{T}$, we note that we may identify $\mathcal{T}(\mathbf{K})=\mathrm{Hom}(M, \mathbf{K}^*)$, where we write $M$ for the character group of $\mathcal{T}$, $X^*(\mathcal{T})$, viewed as an abstract group and written additively. The positive part of $\mathcal{T}(\mathbf{K})$ is defined by those homomorphisms which take values in $\mathbf{K}_{>0}$, namely $\mathcal{T}(\mathbf{K}_{>0})=\mathrm{Hom}(M, \mathbf{K}_{>0})$. For $v \in M$ and $h \in \mathcal{T}(\mathbf{K})$ we will write $\chi^v(h)$ for the associated evaluation $h(v)$ in $\mathbf{K}^*$. We call $\chi^v$ the character associated to $v$.

We call a $\mathbf{K}$-linear combination of the characters $\chi^v$, a Laurent polynomial on $\mathcal{T}$. In addition, a Laurent polynomial is said to be positive if the coefficients of the characters lie in $\mathbf{K}_{>0}$. Now let $\mathcal{T}^{(1)}$, $\mathcal{T}^{(2)}$ be two tori over $\mathbf{K}$. We say that a rational map
    $$\psi: \mathcal{T}^{(1)} \dashrightarrow \mathcal{T}^{(2)}
    $$
is a positive rational map if, for any character $\chi$ of $\mathcal{T}^{(2)}$, the composition $\chi \circ \psi : \mathcal{T}^{(1)} \to \mathbf{K}$ is given by a quotient of positive Laurent polynomials on $\mathcal{T}^{(1)}$.

We now define the tropicalisation of these positive rational maps. Roughly speaking, it captures what happens to the leading term exponents. In order to define tropicalisation we use the natural valuation on $\mathbf{K}$ given by
    $$\mathrm{Val}_{\mathbf{K}} : \mathbf{K} \to \mathbb{R} \cup \{\infty\}, \quad
    \mathrm{Val}_{\mathbf{K}}\left(c(t)\right) = \begin{cases}
    \mu_0 & \text{if} \ c(t)=\sum_{(\mu_k) \in \mathrm{MonSeq}} c_{\mu_k}t^{\mu_k} \neq 0, \\
    \infty & \text{if} \ c(t)=0.
    \end{cases}
    $$

We define an equivalence relation $\sim$ on $\mathcal{T}(\mathbf{K}_{>0})$ using this valuation: we say $h \sim h'$ if and only if $\mathrm{Val}_{\mathbf{K}}(\chi(h))=\mathrm{Val}_{\mathbf{K}}(\chi(h'))$ for all characters $\chi$ of $\mathcal{T}$. Then the tropicalisation of the torus $\mathcal{T}$ is defined to be
    $$\mathrm{Trop}(\mathcal{T}):= T(\mathbf{K}_{>0})/\sim.
    $$
This set inherits the structure of an abelian group from the group structure of $\mathcal{\mathcal{T}}(\mathbf{K}_{>0})$, we denote this as addition.

In practical terms, when $\mathcal{T}=(\mathbf{K}^*)^r$, so that $\mathcal{T}(\mathbf{K}_{>0})=(\mathbf{K}_{>0})^r$, the valuation $\mathrm{Val}_{\mathbf{K}}$ on each coordinate gives an identification
    $$\mathrm{Trop}(\mathcal{T}) \to \mathbb{R}^r, \quad
    \left[ \left( c_1(t), \ldots, c_r(t) \right) \right] \mapsto \left( \mathrm{Val}_{\mathbf{K}}\left( c_1(t) \right), \ldots, \mathrm{Val}_{\mathbf{K}}\left( c_r(t) \right) \right).
    $$
To state this in a coordinate-free way, if $\mathcal{T}$ is a torus with cocharacter lattice $N:=X_*(\mathcal{T})$, then $\mathrm{Trop}(\mathcal{T})$ is identified with $N_{\mathbb{R}}=N\otimes \mathbb{R}$, see for example \cite{JuddRietsch2019}. We note that $N_{\mathbb{R}}$ is also identified with the Lie algebra of the torus taken over $\mathbb{R}$, for example $\mathrm{Trop}(T^{\vee})=\mathfrak{h}^*_{\mathbb{R}}$ (c.f. Section \ref{subsec Notation and definitions}).

We make the convention that if the coordinates of our torus $\mathcal{T}\cong(\mathbf{K}^*)^r$ are labelled by Roman letters, then the corresponding coordinates on $\mathrm{Trop}(\mathcal{T}) \cong \mathbb{R}^r$ are labelled by the associated Greek letters. In addition, by $(\mathbf{K}^*)^r_{\boldsymbol{b}}$ we mean $(\mathbf{K}^*)^r$ with coordinates $(b_1, \ldots, b_r)$, and similarly for $\mathbb{R}^N_{\boldsymbol{\zeta}}$, etc.

Suppose that $\mathcal{T}^{(1)}$, $\mathcal{T}^{(2)}$ are two tori over $\mathbf{K}$ and $\psi: \mathcal{T}^{(1)} \dashrightarrow \mathcal{T}^{(2)}$ is a positive rational map. The map
    $$\psi(\mathbf{K}_{>0}) : \mathcal{T}^{(1)}(\mathbf{K}_{>0}) \to \mathcal{T}^{(2)}(\mathbf{K}_{>0})
    $$
is well-defined and compatible with the equivalence relation $\sim$ (using the positivity of the leading terms). The tropicalisation $\mathrm{Trop}(\psi)$ is then defined to be the resulting map
    $$\mathrm{Trop}(\psi) : \mathrm{Trop}(\mathcal{T}^{(1)}) \to \mathrm{Trop}(\mathcal{T}^{(2)})
    $$
between equivalence classes. It is piecewise-linear with respect to the linear structures on the $\mathrm{Trop}(\mathcal{T}^{(i)})$.

In the case of a variety $X$ with a `positive atlas' consisting of torus charts related by positive birational maps (see \cite{FockGoncharov2006}, \cite{BerensteinKazhdan2007}), there is a well-defined positive part $X(\mathbf{K}_{>0})$ and tropical version $\mathrm{Trop}(X)$, which comes with a tropical atlas whose tropical charts $\mathrm{Trop}(X)\to \mathbb{R}^r$ are related by piecewise-linear maps. $\mathrm{Trop}(X)$ in this more general setting is a space with a piecewise-linear structure.

\begin{ex}
Let $\mathcal{T}=(\mathbf{K}^*)^2_{\boldsymbol{b}}$ and consider the following map\footnote{
This is the superpotential for $\mathbb{CP}^2$ (see \cite{EguchiHoriXiong1997}).
}:
    $$\psi:\mathcal{T} \to \mathbf{K}, \quad \psi(b_1,b_2) = b_1 + b_2 + \frac{t^3}{b_1b_2}.
    $$
We may consider $\psi$ as a positive birational map $\mathcal{T} \dashrightarrow \mathbf{K}^*$, and the corresponding map $\mathrm{Trop}(\psi) : \mathrm{Trop}(\mathcal{T}) \to \mathbb{R}$ is given in terms of the natural coordinates $\beta_1, \beta_2$ on $\mathrm{Trop}(\mathcal{T}) \cong \mathbb{R}^2$ by
    $$\mathrm{Trop}(\psi)(\beta_1, \beta_2) = \min\{ \beta_1, \beta_2, 3-\beta_1-\beta_2 \}.
    $$
In practice, we may think of tropicalisation as replacing addition by $\min$ and replacing multiplication by addition.
\end{ex}

\subsection{Constructing polytopes} \label{subsec Constructing polytopes}

In this section we return to the Landau-Ginzburg model for $G/B$, defined in Section \ref{subsec Landau-Ginzburg models} as the pair $(Z, \mathcal{W})$. Working now over the field of generalised Puiseux series, $\mathbf{K}$, there is a well-defined notion of the totally positive part of $Z(\mathbf{K})$, denoted by $Z(\mathbf{K}_{>0})$. It is defined, for a given torus chart on $Z(\mathbf{K})$, by the subset where the characters take values in $\mathbf{K}_{>0}$. Moreover, each of the string, ideal and quiver torus charts mentioned in previous sections, gives an isomorphism
    \begin{equation} \label{eqn toric chart isom}
    T^{\vee}(\mathbf{K}_{>0}) \times (\mathbf{K}_{>0})^N \xrightarrow{\sim} Z(\mathbf{K}_{>0})
    \end{equation}
where we consider $T^{\vee}(\mathbf{K}_{>0})$ to be the highest weight torus.

We will now restrict our attention to a fibre of the highest weight map (see Section \ref{subsec Landau-Ginzburg models}). To do so, we observe that since a dominant integral weight $\lambda \in X^*(T)^+$ is a cocharacter of $T^{\vee}$, we can define $t^{\lambda} \in T^{\vee}(\mathbf{K}_{>0})$ via the condition $\chi(t^{\lambda}) = t^{\langle \lambda, \chi\rangle}$ for $\chi \in X^*(T^{\vee})$.
Extending %(\ref{eqn bilinear form pair lattices})
$\mathbb{R}$-bilinearly to the perfect pairing
    $$ \langle \ , \ \rangle :  X^*(T)_{\mathbb{R}} \times  X^*(T^{\vee})_{\mathbb{R}} \to \mathbb{R},
    $$
we have that $t^{\lambda}$ is well defined for all $\lambda \in X^*(T)_{\mathbb{R}}$, by the same formula. We therefore do not require that $\lambda$ be integral, though we continue to be interested in those $\lambda$ which are dominant, that is $\lambda\in X^*(T)^+_{\mathbb{R}}$.
This allows us, for a dominant weight $\lambda\in X^*(T)^+_{\mathbb{R}}$, to define
    $$Z_{t^\lambda}(\mathbf{K}):= \{b \in Z(\mathbf{K}) \ | \ \mathrm{hw}(b)=t^\lambda \}.
    $$
We denote the restriction of the superpotential to this fibre by
    $$\mathcal{W}_{t^{\lambda}}: Z_{t^\lambda}(\mathbf{K}) \to \mathbf{K}.
    $$

For a fixed element $t^{\lambda} \in T^{\vee}(\mathbf{K}_{>0})$ of the highest weight torus, the isomorphisms (\ref{eqn toric chart isom}) for the string and ideal toric charts restrict to
    $$\begin{aligned}
    \phi_{t^{\lambda},\boldsymbol{z}} &: (\mathbf{K}_{>0})^N \to Z_{t^{\lambda}}(\mathbf{K}_{>0}), \\
    \phi_{t^{\lambda},\boldsymbol{m}} &: (\mathbf{K}_{>0})^N \to Z_{t^{\lambda}}(\mathbf{K}_{>0}) \\
    \end{aligned}
    $$
respectively, with coordinates denoted by $(z_1, \ldots, z_N)$ and $(m_1, \ldots, m_N)$.
These toric charts may be considered as defining a positive atlas for $Z_{t^{\lambda}}(\mathbf{K}_{>0})$.
We denote the respective compositions of $\phi_{t^{\lambda},\boldsymbol{z}}$ and $\phi_{t^{\lambda},\boldsymbol{m}}$ with the superpotential $\mathcal{W}_{t^{\lambda}}$, by
    $$\begin{aligned}
    \mathcal{W}_{t^{\lambda},\boldsymbol{z}} &: (\mathbf{K}_{>0})^N \to \mathbf{K}_{>0}, \\
    \mathcal{W}_{t^{\lambda},\boldsymbol{m}} &: (\mathbf{K}_{>0})^N \to \mathbf{K}_{>0} \\
    \end{aligned}
    $$
and observe that both are positive rational maps. We denote their tropicalisations respectively by
    $$\begin{aligned}
    \mathrm{Trop}\left(\mathcal{W}_{t^{\lambda},\boldsymbol{z}}\right) &: \mathbb{R}^N_{\boldsymbol{\zeta}} \to \mathbb{R}, \\
    \mathrm{Trop}\left(\mathcal{W}_{t^{\lambda},\boldsymbol{m}}\right) &: \mathbb{R}^N_{\boldsymbol{\mu}} \to \mathbb{R}. \\
    \end{aligned}
    $$

We may associate convex polytopes to our tropical superpotentials, defined as follows:
    $$\begin{aligned}
    \mathcal{P}_{\lambda,\boldsymbol{\zeta}} &:= \left\{ \boldsymbol{\alpha} \in \mathbb{R}^N_{\boldsymbol{\zeta}} \ | \ \mathrm{Trop}\left(\mathcal{W}_{t^{\lambda},\boldsymbol{z}}\right)(\boldsymbol{\alpha}) \geq 0 \right\}, \\
    \mathcal{P}_{\lambda,\boldsymbol{\mu}} &:= \left\{ \boldsymbol{\alpha} \in \mathbb{R}^N_{\boldsymbol{\mu}} \ | \ \mathrm{Trop}\left(\mathcal{W}_{t^{\lambda},\boldsymbol{m}}\right)(\boldsymbol{\alpha}) \geq 0 \right\}. \\
    \end{aligned}
    $$

To motivate the definition of these polytopes, first recall the string toric chart, $\varphi_{\mathbf{i}}$, for an arbitrary reduced expression $\mathbf{i}$, defined by (\ref{eqn string coord chart defn}) in Section \ref{subsec The string coordinates}. Using this we have generalisations
    $$\begin{aligned}
    \phi^{\mathbf{i}}_{t^{\lambda},\boldsymbol{z}} &: (\mathbf{K}_{>0})^N \to Z_{t^{\lambda}}(\mathbf{K}_{>0}), \\
    \mathcal{W}^{\mathbf{i}}_{t^{\lambda},\boldsymbol{z}} &: (\mathbf{K}_{>0})^N \to \mathbf{K}_{>0}
    \end{aligned}
    $$
of the maps above, such that $\phi_{t^{\lambda},\boldsymbol{z}} = \phi^{\mathbf{i}_0}_{t^{\lambda},\boldsymbol{z}}$ and $\mathcal{W}_{t^{\lambda},\boldsymbol{z}}=\mathcal{W}^{\mathbf{i}_0}_{t^{\lambda},\boldsymbol{z}}$. With this notation we have the following theorem:

\begin{thm}[{\cite[Theorem 4.1]{Judd2018}}] \label{thm string polytope from string coords}
Consider a general reduced expression $\mathbf{i}$ for $\bar{w}_0$, and the superpotential for $GL_n/B$ written in the associated string coordinates, namely $\mathcal{W}^{\mathbf{i}}_{t^{\lambda},\boldsymbol{z}}$. Then the polytope
    $$\mathcal{P}^{\mathbf{i}}_{\lambda,\boldsymbol{\zeta}} = \left\{ \boldsymbol{\alpha} \in \mathbb{R}^N_{\boldsymbol{\zeta}} \ \big| \ \mathrm{Trop}\left(\mathcal{W}^{\mathbf{i}}_{t^{\lambda},\boldsymbol{z}}\right)(\boldsymbol{\alpha}) \geq 0 \right\}
    $$
is the string polytope associated to $\mathbf{i}$, $\mathrm{String}_{\mathbf{i}}(\lambda)$.
\end{thm}

\begin{rem}
The polytopes $\mathcal{P}_{\lambda,\boldsymbol{\zeta}}$ and $\mathcal{P}_{\lambda,\boldsymbol{\mu}}$ are simply linear transformations of each other.

If instead we were to take the toric chart given by the vertex torus, then the resulting polytope would be the respective Gelfand-Tsetlin polytope for $\lambda$. Since the quiver torus is so closely related to the vertex torus, the coordinates of the quiver toric chart provide a bridge between the string and Gelfand-Tsetlin polytopes. In particular, we see that the tropicalisation of the coordinate change from the string to the vertex coordinates defines an affine map between these two polytopes.
\end{rem}

\subsection{Tropical critical points and the weight map} \label{subsec Tropical critical points and the weight map}

We recall the critical point conditions of the superpotential given in (\ref{eqn crit pt conditions}) which, working over $\mathbf{K}_{>0}$, define critical points of $\mathcal{W}$ in the fibres $Z_{t^\lambda}(\mathbf{K}_{>0})$. Judd showed in the $SL_n$ case that for each dominant integral weight $\lambda$, there is in fact only one critical point that lies in $Z_{t^{\lambda}}(\mathbf{K}_{>0})$ (see \cite[Section 5]{Judd2018}). We refer to this unique point as the positive critical point of $\mathcal{W}_{t^{\lambda}}$, denoted $p_{\lambda}$. Judd's statement extends to the $GL_n$ case that we are considering here, with the same proof. This also follows from the more general result of Judd and Rietsch in \cite{JuddRietsch2019}, moreover the assumption on $\lambda$ to be integral can be dropped, therefore we also have a unique positive critical point of $\mathcal{W}_{t^{\lambda}}$ in $Z_{t^{\lambda}}(\mathbf{K}_{>0})$ for any dominant $\lambda \in \mathbb{R}^n$. We use the same notation, $p_{\lambda}$, for this point. In addition we will use the term dominant weight loosely, to mean $\lambda \in \mathbb{R}^n$ such that $\lambda_1 \geq \lambda_2 \geq \cdots \geq \lambda_n$, and say dominant integral weight if, in addition, $\lambda \in \mathbb{Z}^n$.

This critical point $p_{\lambda} \in Z_{t^{\lambda}}(\mathbf{K}_{>0})$ defines a point $p_{\lambda}^{\mathrm{trop}} \in \mathrm{Trop}(Z_{t^{\lambda}})$, called the tropical critical point of $\mathcal{W}_{t^{\lambda}}$. Explicitly, using a positive chart (such as $\phi_{t^{\lambda},\boldsymbol{z}}$ or $\phi_{t^{\lambda},\boldsymbol{m}}$) we apply the valuation $\mathrm{Val}_{\mathbf{K}}$ to every coordinate of $p_{\lambda}$. This gives rise to the corresponding point ($p_{\lambda, \boldsymbol{\zeta}}^{\mathrm{trop}}$ or $p_{\lambda, \boldsymbol{\mu}}^{\mathrm{trop}}$ respectively) in the associated tropical chart $\mathrm{Trop}(Z_{t^{\lambda}}) \to \mathbb{R}^N$.
Moreover, for a choice of positive chart the tropical critical point lies in the interior of the respective superpotential polytope, for example, $p_{\lambda, \boldsymbol{\zeta}}^{\mathrm{trop}} \in \mathcal{P}_{\lambda,\boldsymbol{\zeta}}$ and $p_{\lambda, \boldsymbol{\mu}}^{\mathrm{trop}} \in \mathcal{P}_{\lambda,\boldsymbol{\mu}}$. This is implicit in Judd's work in \cite{Judd2018} but is also true more generally, with an explicit statement given by Judd and Rietsch in \cite[Theorem 1.2]{JuddRietsch2019}.

We also have the tropicalisation of the weight map, $\mathrm{wt}:Z_{t^{\lambda}}\to T^{\vee}$ defined in Section \ref{subsec Landau-Ginzburg models}, which can be interpreted as a kind of projection
    $$\mathrm{Trop}(\mathrm{wt}) : \mathrm{Trop}(Z_{t^{\lambda}}) \to \mathfrak{h}^*_{\mathbb{R}}.
    $$
In particular, in the case of integral $\lambda$, the image under this projection of either superpotential polytope, $\mathcal{P}_{\lambda,\boldsymbol{\zeta}}$ or $\mathcal{P}_{\lambda,\boldsymbol{\mu}}$, is exactly the weight polytope. We therefore generalise the standard definition of the weight polytope to be the projection of the superpotential polytope under $\mathrm{Trop}(\mathrm{wt})$. This extended definition holds for all dominant weights $\lambda$.

In the $SL_n$ case, Judd proved that $\mathrm{Trop}(\mathrm{wt})\left(p_{\lambda}^{\mathrm{trop}}\right)=0$ (see \cite[Theorem 5.1]{Judd2018}). Working more generally in the $GL_n$ case, we obtain that $\mathrm{Trop}(\mathrm{wt})\left(p_{\lambda}^{\mathrm{trop}}\right)$, the image of the tropical critical point under this weight projection, is in fact the centre of mass of the weight polytope:
\begin{cor}[Corollary of Proposition {\ref{prop weight at crit point non trop}}] \label{cor weight matrix at trop crit point}
Given a dominant weight $\lambda$, the weight matrix at the critical point in the fibre over $t^{\lambda} \in T^{\vee}(\mathbf{K}_{>0})$ is an $n\times n$ matrix $\mathrm{diag}\left(t^{\ell}, \ldots, t^{\ell}\right)$ where
    $$\ell = \frac{1}{n}\sum_{i=1}^n \lambda_i.
    $$
\end{cor}
\hfill\qedsymbol

\begin{ex}[Dimension 3] \label{ex nu'_i coords in dim 3} Recalling Example \ref{ex dim 3 z coords} and working now over the field of generalised Puiseux series, $\mathbf{K}$, we have
    $$\begin{aligned}
    b &= \begin{pmatrix} 1 & z_3 & z_2 \\ & 1 & z_1 +\frac{z_2}{z_3} \\ & & 1\end{pmatrix}
        \begin{pmatrix} d_1& &  \\ & d_2 & \\ & & d_3 \end{pmatrix}
        \begin{pmatrix} & & 1 \\ & -1 & \\ 1 & & \end{pmatrix}
        \begin{pmatrix} 1 & \frac{d_2}{d_3} \frac{z_3}{z_2} & \frac{d_1}{d_3}\frac{1}{z_1z_3} \\ & 1 & \frac{d_1}{d_2}\left(\frac{1}{z_3}+\frac{z_2}{z_1z_3^2}\right) \\ & & 1\end{pmatrix} \\
    &=\begin{pmatrix} d_3 z_2 & & \\ d_3 \left(z_1+\frac{z_2}{z_3} \right) & d_2 \frac{z_1z_3}{z_2} & \\ d_3 & d_2 \frac{z_3}{z_2} & d_1 \frac{1}{z_1z_3} \end{pmatrix}.
    \end{aligned}$$
Additionally, in reference to the previous section, we give the quiver for this coordinate system in Figure \ref{fig quiver n=3 z coords}.
\begin{figure}[ht]
    \centering
    \begin{tikzpicture}[scale=0.85]
        %dots and stars
        \node (31) at (0,0) {$\bullet$};
        \node (21) at (0,2) {$\bullet$};
        \node (11) at (0,4) {$\boldsymbol{*}$};
            \node at (0.3,4.3) {$d_1$};
        \node (32) at (2,0) {$\bullet$};
        \node (22) at (2,2) {$\boldsymbol{*}$};
            \node at (2.3,2.3) {$d_2$};
        \node (33) at (4,0) {$\boldsymbol{*}$};
            \node at (4.3,0.3) {$d_3$};

        %arrows
        \draw[->] (31) -- (21);
        \draw[->] (21) -- (11);
        \draw[->] (32) -- (22);

        \draw[->] (22) -- (21);
        \draw[->] (33) -- (32);
        \draw[->] (32) -- (31);

        %arrow labels
        \node at (-0.3,0.8) {\scriptsize{$z_1$}};
        \node at (-0.3,2.8) {\scriptsize{$z_3$}};
        \node at (1.7,0.8) {\scriptsize{$\frac{z_2}{z_3}$}};

        \node at (1,-0.4) {\scriptsize{$\frac{d_1}{d_2}\frac{z_2}{z_1z_3^2}$}};
        \node at (3,-0.4) {\scriptsize{$\frac{d_2}{d_3}\frac{z_3}{z_2}$}};
        \node at (1,1.6) {\scriptsize{$\frac{d_1}{d_2}\frac{1}{z_3}$}};
    \end{tikzpicture}
\caption{Quiver decoration when $n=3$, string coordinates} \label{fig quiver n=3 z coords}
\end{figure}
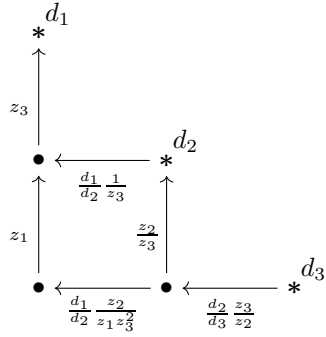
We recall that we can use the quiver to read off the superpotential. It is the sum of the arrow coordinates and is the same map as given in Example \ref{ex dim 3 z coords}, obtained from $b$ via the formula in that section:
    $$\mathcal{W}(d,\boldsymbol{z}) = z_1 +\frac{z_2}{z_3} + z_3 + \frac{d_1}{d_2}\left(\frac{1}{z_3} + \frac{z_2}{z_1z_3^2}\right) + \frac{d_2}{d_3}\frac{z_3}{z_2}.
    $$

Now in order to obtain a polytope from $\mathcal{W}$, we need to tropicalise. To do so we take our highest weight torus element $d$ to be $t^{\lambda}\in T^{\vee}(\mathbf{K}_{>0})$, with $\lambda=(\lambda_1 \geq \lambda_2 \geq \lambda_3)$, that is $t^{\lambda}=\mathrm{diag}(t^{\lambda_1}, t^{\lambda_2}, t^{\lambda_3})$. Then our tropical superpotential is
    \begin{multline*}
    \mathrm{Trop}\left(\mathcal{W}_{t^{\lambda},\boldsymbol{z}}\right) (\zeta_1,\zeta_2,\zeta_3)
    = \min\left\{\zeta_1, \zeta_2-\zeta_3, \zeta_3, \lambda_1 - \lambda_2 - \zeta_3, \lambda_1 - \lambda_2 - \zeta_1 + \zeta_2 - 2\zeta_3, \right. \\ \left. \lambda_2 - \lambda_3 - \zeta_2 + \zeta_3 \right\}.
    \end{multline*}
The corresponding polytope, $\mathcal{P}_{\lambda,\boldsymbol{\zeta}} = \left\{ \boldsymbol{\zeta} \in \mathbb{R}^N \ | \ \mathrm{Trop}\left(\mathcal{W}_{t^{\lambda},\boldsymbol{z}}\right)(\boldsymbol{\zeta}) \geq 0 \right\}$, is then cut out by the following inequalities:
    $$\begin{aligned} 0 \leq &\zeta_1 \leq \lambda_1 - \lambda_2 + \zeta_2 - 2\zeta_3, \\
    \zeta_3 \leq &\zeta_2 \leq \lambda_2 - \lambda_3 + \zeta_3, \\
    0 \leq &\zeta_3 \leq \lambda_1 - \lambda_2.
    \end{aligned}$$
This polytope is given in Figure \ref{fig polytope n=3 z coords} for $\lambda=(2,1,-1)$.
\begin{figure}[ht!]
\centering
\begin{minipage}[b]{0.47\textwidth}
    \centering
    \begin{tikzpicture}[scale=1]
        \draw[black!50, ->] (-0.4,0.2) -- (4,-2); % x axis
            \node[black!50] at (4,-1.72) {\scriptsize{$\zeta_1$}};
        \draw[black!50, ->] (-0.5,-0.1) -- (5,1); % y axis
            \node[black!50] at (4.9,1.21) {\scriptsize{$\zeta_2$}};
        \draw[black!50, ->] (0,-0.4) -- (0,2); % z axis
            \node[black!50] at (-0.28,1.9) {\scriptsize{$\zeta_3$}};
            % \node[black!60] at (-0.2, 0.3) {O};

        \draw[thick, dashed, rounded corners=0.5] (0,0) -- (2,0.4) -- (3,1.6) -- (1,1.2) -- cycle;
        \draw[thick, dashed, rounded corners=0.5] (2,0.4) -- (4.4,-0.8) -- (4.6,0.8) -- (3,1.6) --  cycle;
        \draw[thick, rounded corners=0.5] (1,1.2) -- (3,1.6) -- (4.6,0.8) -- cycle;
        \draw[thick, rounded corners=0.5] (0,0) -- (0.8,-0.4) -- (1,1.2) -- cycle;
        \draw[thick, rounded corners=0.5] (0.8,-0.4) -- (4.4,-0.8) -- (4.6,0.8) -- (1,1.2) -- cycle;

        \filldraw[blue] (0,0) circle (1pt);
        \filldraw (2,0.4) circle (1pt);
        \filldraw (3,1.6) circle (1pt);
        \filldraw (1,1.2) circle (1pt);
        \filldraw (0.8,-0.4) circle (1pt);
        \filldraw (2.8,0) circle (1pt);
        \filldraw (3.6,-0.4) circle (1pt);
        \filldraw (4.4,-0.8) circle (1pt);
        \filldraw[red] (4.6,0.8) circle (1pt);

        \node at (0,-1.19) {\tiny$\begin{aligned} &(0,0,0) \\ &(\lambda_1-\lambda_2,0,0) \\ &(0,\lambda_2-\lambda_3,0) \\ &(\lambda_1-\lambda_2,\lambda_2-\lambda_3,0) \\ &(\lambda_2-\lambda_3, \lambda_2-\lambda_3, 0) \\ &(\lambda_1-\lambda_3,\lambda_2-\lambda_3,0) \end{aligned}$};
            \draw[<-, >=stealth, densely dashed, black!50, rounded corners=4] (-0.03,-0.05) -- (-0.14,-0.3) -- (-0.35,-0.4);
            \draw[<-, >=stealth, densely dashed, black!50, rounded corners=5] (0.77,-0.45) -- (0.66,-0.65) -- (0.41,-0.72);
            \draw[<-, >=stealth, densely dashed, black!50, rounded corners=25] (1.98,0.35) -- (1.6,-0.7) -- (0.4,-1.03);
            \draw[<-, >=stealth, densely dashed, black!50, rounded corners=21] (2.77,-0.05) -- (2.2,-1) -- (1.15,-1.35);
            \draw[<-, >=stealth, densely dashed, black!50, rounded corners=26] (3.56,-0.45) -- (2.8,-1.25) -- (1.16,-1.66);
            \draw[<-, >=stealth, densely dashed, black!50, rounded corners=28] (4.34,-0.84) -- (3,-1.65) -- (1.16,-1.97);
        \node at (1.73,2) {\tiny$\begin{aligned} &(\lambda_2-\lambda_3,\lambda_1-\lambda_3,\lambda_1-\lambda_2) \\ &(0,\lambda_1-\lambda_3,\lambda_1-\lambda_2) \\ &(0,\lambda_1-\lambda_2,\lambda_1-\lambda_2) \end{aligned}$};
            \draw[<-, >=stealth, densely dashed, black!50, rounded corners=18] (4.57,0.86) -- (4,1.8) -- (3.25,2.3);
            \draw[<-, >=stealth, densely dashed, black!50, rounded corners=5] (2.97,1.65) -- (2.76,1.9) -- (2.5,2);
            \draw[<-, >=stealth, densely dashed, black!50, rounded corners=7] (0.95,1.22) -- (0.4,1.27) -- (-0.05,1.5) -- (0.18,1.66);
    \end{tikzpicture}
\caption{Superpotential polytope $\mathcal{P}_{\lambda,\boldsymbol{\zeta}}$ for $n=3$ and $\lambda=(2,1,-1)$ (string coordinates)} \label{fig polytope n=3 z coords}
\end{minipage} \hfill
\begin{minipage}[b]{0.47\textwidth}
    \centering
    \begin{tikzpicture}[scale=0.55]
        %horizontal
        \draw[black!50, densely dotted] (-3,0) -- (3,0);
        \draw[black!50, densely dotted] (-3,0.87) -- (3,0.87);
        \draw[black!50, densely dotted] (-3,-0.87) -- (3,-0.87);
        \draw[black!50, densely dotted] (-3,1.74) -- (3,1.74);
        \draw[black!50, densely dotted] (-3,-1.74) -- (3,-1.74);
        \draw[black!50, densely dotted] (-3,2.61) -- (3,2.61);
        \draw[black!50, densely dotted] (-3,-2.61) -- (3,-2.61);
        % up slant
        \draw[black!50, densely dotted] (-3,1.74) -- (-2.5,2.61);
        \draw[black!50, densely dotted] (-3,0) -- (-1.5,2.61);
        \draw[black!50, densely dotted] (-3,-1.74) -- (-0.5,2.61);
        \draw[black!50, densely dotted] (-2.5,-2.61) -- (0.5,2.61);
        \draw[black!50, densely dotted] (-1.5,-2.61) -- (1.5,2.61);
        \draw[black!50, densely dotted] (-0.5,-2.61) -- (2.5,2.61);
        \draw[black!50, densely dotted] (0.5,-2.61) -- (3,1.74);
        \draw[black!50, densely dotted] (1.5,-2.61) -- (3,0);
        \draw[black!50, densely dotted] (2.5,-2.61) -- (3,-1.74);
        % down slant
        \draw[black!50, densely dotted] (-3,-1.74) -- (-2.5,-2.61);
        \draw[black!50, densely dotted] (-3,0) -- (-1.5,-2.61);
        \draw[black!50, densely dotted] (-3,1.74) -- (-0.5,-2.61);
        \draw[black!50, densely dotted] (-2.5,2.61) -- (0.5,-2.61);
        \draw[black!50, densely dotted] (-1.5,2.61) -- (1.5,-2.61);
        \draw[black!50, densely dotted] (-0.5,2.61) -- (2.5,-2.61);
        \draw[black!50, densely dotted] (0.5,2.61) -- (3,-1.74);
        \draw[black!50, densely dotted] (1.5,2.61) -- (3,0);
        \draw[black!50, densely dotted] (2.5,2.61) -- (3,1.74);

            \node[black!60] at (0.15,0.09) {\tiny{$0$}};

        \draw (-2.5,-0.87) -- (-2.5,0.87);
        \draw (-2.5,-0.87) -- (2,1.74);
        \draw (-2.5,-0.87) -- (0.5,-2.61);
        \draw (-2.5,0.87) -- (0.5,2.61);
        \draw (-2.5,0.87) -- (2,-1.74);
        \draw (0.5,2.61) -- (0.5,-2.61);
        \draw (0.5,2.61) -- (2,1.74);
        \draw (2,1.74) -- (2,-1.74);
        \draw (0.5,-2.61) -- (2,-1.74);

        \filldraw[blue] (-2.5,-0.87) circle (1.5pt);
        \filldraw (-2.5,0.87) circle (1.5pt);
        \filldraw (-1,0) circle (1.5pt);
        \filldraw (0.5,0.87) circle (1.5pt);
        \filldraw (0.5,-0.87) circle (1.5pt);
        \filldraw (0.5,2.61) circle (1.5pt);
        \filldraw (0.5,-2.61) circle (1.5pt);
        \filldraw (2,-1.74) circle (1.5pt);
        \filldraw[red] (2,1.74) circle (1.5pt);

        \node at (5.25,0.5) {\tiny$(\lambda_2,\lambda_2,\lambda_1-\lambda_2+\lambda_3)$};
            \draw[<-, >=stealth, densely dashed, black!50, rounded corners=13] (0.6,0.8) -- (1.5,0.45) -- (3,0.5);
        \node at (5.25,0) {\tiny$(\lambda_1-\lambda_2+\lambda_3,\lambda_2,\lambda_2)$};
            \draw[<-, >=stealth, densely dashed, black!50, rounded corners=20] (-0.89,0.005) -- (0.5,-0.25) -- (3,0);
        \node at (5.25,-0.5) {\tiny$(\lambda_2,\lambda_1-\lambda_2+\lambda_3,\lambda_2)$};
            \draw[<-, >=stealth, densely dashed, black!50, rounded corners=13] (0.6,-0.8) -- (1.75,-0.45) -- (3,-0.5);

        \node at (4.2,2.5) {\tiny$(\lambda_2,\lambda_1,\lambda_3)$};
            \draw[<-, >=stealth, densely dashed, black!50, rounded corners=12] (0.7,2.55) -- (1.8,2.35) -- (3,2.5);
        \node at (4.2,2) {\tiny$(\lambda_1,\lambda_2,\lambda_3)$};
            \draw[<-, >=stealth, densely dashed, black!50, rounded corners=4] (2.1,1.8) -- (2.7,2) -- (2.9,2);
        \node at (4.2,1.5) {\tiny$(\lambda_3,\lambda_1,\lambda_2)$};
            \draw[<-, >=stealth, densely dashed, black!50, rounded corners=27] (-2.37,0.86) -- (0.5,1.6) -- (3,1.5);
        \node at (4.2,-1.5) {\tiny$(\lambda_3,\lambda_2,\lambda_1)$};
            \draw[<-, >=stealth, densely dashed, black!50, rounded corners=27] (-2.37,-0.86) -- (0.5,-1.6) -- (3,-1.5);
        \node at (4.2,-2) {\tiny$(\lambda_1,\lambda_3,\lambda_2)$};
            \draw[<-, >=stealth, densely dashed, black!50, rounded corners=4] (2.1,-1.8) -- (2.7,-2) -- (2.9,-2);
        \node at (4.2,-2.5) {\tiny$(\lambda_2,\lambda_3,\lambda_1)$};
            \draw[<-, >=stealth, densely dashed, black!50, rounded corners=12] (0.7,-2.55) -- (1.8,-2.35) -- (3,-2.5);
    \end{tikzpicture}
\caption{Projection of superpotential polytope $\mathcal{P}_{\lambda,\boldsymbol{\zeta}}$ onto weight lattice, $\lambda=(2,1,-1)$} \label{fig Projection of superpotential polytope onto weight lattice lambda=(2,1,-1)}
\end{minipage}
\end{figure}

Finally, recalling the weight matrix
    $$\begin{pmatrix} d_3z_2 & & \\ & d_2\frac{z_1z_3}{z_2} & \\ & & d_1\frac{1}{z_1z_3} \end{pmatrix}
    $$
we see that a point $\left(\zeta_1, \zeta_2, \zeta_3\right)$ in the polytope has weight
    $$\left(\lambda_3 + \zeta_2, \lambda_2 + \zeta_1 - \zeta_2 + \zeta_3, \lambda_1 - \zeta_1 - \zeta_3 \right).
    $$
In particular the weight projection, given for $\lambda=(2,1,-1)$ in Figure \ref{fig Projection of superpotential polytope onto weight lattice lambda=(2,1,-1)}, acts on the vertices and distinguished points as described in Table \ref{tab Vertices and distinguished points, and their corresponding weights, string coordinates}.
\begin{table}[ht!]
\centering
\begin{tabular}{c | c}
    Regular vertices & Weight \\
    \hline\hline
    $(0, 0, 0)$
        & $(\lambda_3, \lambda_2, \lambda_1)$ \\ \hline
    $(0, \lambda_2-\lambda_3, 0)$
        & $(\lambda_2, \lambda_3, \lambda_1)$ \\ \hline
    $(\lambda_1-\lambda_2, 0, 0)$
        & $(\lambda_3, \lambda_1, \lambda_2)$ \\ \hline
    $(\lambda_1-\lambda_3, \lambda_2-\lambda_3, 0)$
        & $(\lambda_2, \lambda_1, \lambda_3)$ \\ \hline
    $(0, \lambda_1-\lambda_3, \lambda_1-\lambda_2)$
        & $(\lambda_1, \lambda_3, \lambda_2)$ \\ \hline \vspace{0.3cm}
    $(\lambda_2-\lambda_3, \lambda_1-\lambda_3, \lambda_1-\lambda_2)$
        & $(\lambda_1, \lambda_2, \lambda_3)$ \\
    Irregular vertex and distinguished points & Weight \\
    \hline\hline
    $(0, \lambda_1-\lambda_2, \lambda_1-\lambda_2)$
        & $(\lambda_1-\lambda_2+\lambda_3, \lambda_2, \lambda_2)$ \\ \hline
    $(\lambda_1-\lambda_2, \lambda_2-\lambda_3, 0)$
        & $(\lambda_2, \lambda_1-\lambda_2+\lambda_3, \lambda_2)$ \\ \hline
    $(\lambda_2-\lambda_3, \lambda_2-\lambda_3, 0)$
        & $(\lambda_2, \lambda_2, \lambda_1-\lambda_2+\lambda_3)$ \\ \hline
\end{tabular}
\caption{Vertices and distinguished points, and their corresponding weights, string coordinates} \label{tab Vertices and distinguished points, and their corresponding weights, string coordinates}
\end{table}

\end{ex}

\begin{ex}
For comparison, we now run through the previous example using the ideal coordinates instead. We begin by recalling the matrix $b$:
    $$b=\mathbf{y}_{\mathbf{i}_0}^{\vee}\left(\frac{1}{m_1}, \frac{1}{m_2}, \frac{1}{m_3} \right)
        \begin{pmatrix} d_3 m_2m_3 & & \\ & d_2 \frac{m_1}{m_3} & \\ & & d_1 \frac{1}{m_1m_2} \end{pmatrix}.
    $$
In this case, the quiver is given in Figure \ref{fig quiver n=3 m coords}.
\begin{figure}[ht]
\centering
\begin{tikzpicture}[scale=0.85]
    %dots and stars
    \node (31) at (0,0) {$\bullet$};
    \node (21) at (0,2) {$\bullet$};
    \node (11) at (0,4) {$\boldsymbol{*}$};
        \node at (0.3,4.3) {$d_1$};
    \node (32) at (2,0) {$\bullet$};
    \node (22) at (2,2) {$\boldsymbol{*}$};
        \node at (2.3,2.3) {$d_2$};
    \node (33) at (4,0) {$\boldsymbol{*}$};
        \node at (4.3,0.3) {$d_3$};

    %arrows
    \draw[->] (31) -- (21);
    \draw[->] (21) -- (11);
    \draw[->] (32) -- (22);

    \draw[->] (22) -- (21);
    \draw[->] (33) -- (32);
    \draw[->] (32) -- (31);

    %arrow labels
    \node at (-0.3,0.8) {\scriptsize{$m_2$}};
    \node at (-0.3,2.8) {\scriptsize{$m_1$}};
    \node at (1.5,0.8) {\scriptsize{$\frac{m_2m_3}{m_1}$}};

    \node at (1,-0.4) {\scriptsize{$\frac{d_1}{d_2}\frac{m_3}{m_1^2}$}};
    \node at (3,-0.4) {\scriptsize{$\frac{d_2}{d_3}\frac{m_1}{m_2m_3}$}};
    \node at (1,1.6) {\scriptsize{$\frac{d_1}{d_2}\frac{1}{m_1}$}};
\end{tikzpicture}
\caption{Quiver decoration when $n=3$, ideal coordinates} \label{fig quiver n=3 m coords}
\end{figure}

The superpotential is given by
    $$\mathcal{W}(d, \boldsymbol{m}) = m_1 + m_2 + \frac{m_2m_3}{m_1} + \frac{d_2}{d_3} \frac{m_1}{m_2m_3} + \frac{d_1}{d_2} \left(\frac{m_3}{{m_1}^2}+\frac{1}{m_1}\right).
    $$
We again take our torus element $d$ to be $t^{\lambda}\in T^{\vee}(\mathbf{K}_{>0})$, with $\lambda=(\lambda_1 \geq \lambda_2 \geq \lambda_3)$. Our tropical superpotential is given by
    \begin{multline*}\mathrm{Trop}\left(\mathcal{W}_{t^{\lambda},\boldsymbol{m}}\right) (\mu_1,\mu_2,\mu_3)
    = \min\left\{\mu_1, \mu_2, -\mu_1+\mu_2+\mu_3, \lambda_2-\lambda_3 +\mu_1-\mu_2-\mu_3, \right. \\ \left. \lambda_1-\lambda_2 -2\mu_1+\mu_3, \lambda_1-\lambda_2 -\mu_1 \right\}.
    \end{multline*}
The corresponding polytope $\mathcal{P}_{\lambda,\boldsymbol{\mu}}$ is cut out by
    $$\begin{aligned} 0 &\leq \mu_1 \leq \lambda_1-\lambda_2, \\
    0 &\leq \mu_2, \\
    0 &\leq  -\mu_1 +\mu_2 +\mu_3 \leq \lambda_2-\lambda_3, \\
    2\mu_1-\mu_3 &\leq \lambda_1-\lambda_2.
    \end{aligned}$$
This polytope is given in Figure \ref{fig polytope n=3 m coords} for $\lambda=(2,1,-1)$.
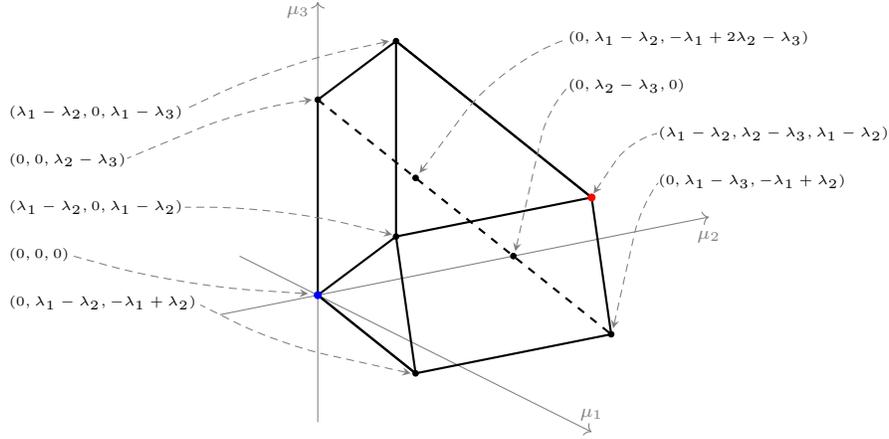
\begin{figure}[ht]
\centering
\begin{tikzpicture}[scale=1.3]
    \draw[black!50, ->] (-0.8,0.4) -- (2.8,-1.4); % x axis
        \node[black!50] at (2.8,-1.22) {\scriptsize{$\mu_1$}};
    \draw[black!50, ->] (-1,-0.2) -- (4,0.8); % y axis
        \node[black!50] at (4,0.6) {\scriptsize{$\mu_2$}};
    \draw[black!50, ->] (0,-1.3) -- (0,3); % z axis
        \node[black!50] at (-0.2,2.9) {\scriptsize{$\mu_3$}};
        % \node[black!60] at (-0.2, 0.3) {O};

    \draw[thick, dashed, rounded corners=0.5] (0,2) -- (3,-0.4);

    \draw[thick, rounded corners=0.5] (0,0) -- (0,2) -- (0.8,2.6) -- (2.8,1) -- (3,-0.4) --  (1,-0.8) -- cycle;
    \draw[thick, rounded corners=0.5] (0.8,0.6) -- (0,0) -- (1,-0.8) -- cycle;
    \draw[thick, rounded corners=0.5] (0.8,0.6) -- (2.8,1) -- (0.8,2.6) -- cycle;

    \filldraw[blue] (0,0) circle (1pt);
    \filldraw (0.8,0.6) circle (0.8pt);
    \filldraw (0,2) circle (0.8pt);
    \filldraw (0.8,2.6) circle (0.8pt);
    \filldraw (1,-0.8) circle (0.8pt);
    \filldraw (1,1.2) circle (0.8pt);
    \filldraw (2,0.4) circle (0.8pt);
    \filldraw (3,-0.4) circle (0.8pt);
    \filldraw[red] (2.8,1) circle (1pt);
    % \filldraw[yellow] (0.9,0.4) circle (1pt);

    \node at (-2.2,0.9) {\tiny$\begin{aligned} &(\lambda_1-\lambda_2, 0, \lambda_1-\lambda_3) \\ \\  &(0, 0, \lambda_2-\lambda_3) \\ \\ &(\lambda_1-\lambda_2, 0, \lambda_1-\lambda_2) \\ \\ &(0,0,0) \\ \\ &(0, \lambda_1-\lambda_2, -\lambda_1+\lambda_2) \end{aligned}$};
        \draw[<-, >=stealth, densely dashed, black!50, rounded corners=25] (0.73,2.6) -- (-0.41,2.35) -- (-1.36,1.9);
        \draw[<-, >=stealth, densely dashed, black!50, rounded corners=19] (-0.06,2) -- (-0.75,1.9) -- (-1.95,1.4);
        \draw[<-, >=stealth, densely dashed, black!50, rounded corners=26] (0.76,0.62) -- (-0.4,0.9) -- (-1.36,0.9);
        \draw[<-, >=stealth, densely dashed, black!50, rounded corners=27] (-0.1,0.015) -- (-1.4,0.1) -- (-2.5,0.4);
        \draw[<-, >=stealth, densely dashed, black!50, rounded corners=18] (0.95,-0.8) -- (-0.4,-0.5) -- (-1.2,-0.1);
    \node at (4.2,1.9) {\tiny$\begin{aligned} &(0,\lambda_1-\lambda_2, -\lambda_1+2\lambda_2-\lambda_3) \\ \\  &(0, \lambda_2-\lambda_3,0) \\ \\ &\hspace{1.2cm}(\lambda_1-\lambda_2, \lambda_2-\lambda_3, \lambda_1-\lambda_2) \\ \\ &\hspace{1.2cm}(0, \lambda_1-\lambda_3, -\lambda_1+\lambda_2) \end{aligned}$};
        \draw[<-, >=stealth, densely dashed, black!50, rounded corners=24] (1.04,1.25) -- (1.9,2.4) -- (2.55,2.62);
        \draw[<-, >=stealth, densely dashed, black!50, rounded corners=10] (2.02,0.45) -- (2.35,1.9) -- (2.55,2.12);
        \draw[<-, >=stealth, densely dashed, black!50, rounded corners=10] (2.82,1.04) -- (3.2,1.6) -- (3.45,1.65);
        \draw[<-, >=stealth, densely dashed, black!50, rounded corners=10] (3.02,-0.34) -- (3.3,0.9) -- (3.48,1.15);
\end{tikzpicture}
\caption{Superpotential polytope $\mathcal{P}_{\lambda,\boldsymbol{\mu}}$ for $n=3$ and $\lambda=(2,1,-1)$ (ideal coordinates)} \label{fig polytope n=3 m coords}
\end{figure}

Similar to the previous example, using the weight matrix
    $$\begin{pmatrix} d_3 m_2m_3 & & \\ & d_2 \frac{m_1}{m_3} & \\ & & d_1 \frac{1}{m_1m_2} \end{pmatrix}
    $$
we see that a point $(\mu_1, \mu_2, \mu_3)$ in the polytope has weight
    $$\left(\lambda_3 + \mu_2 + \mu_3, \lambda_2 + \mu_1 - \mu_3, \lambda_1 - \mu_1 - \mu_2 \right).
    $$
The weight projection acts on the vertices and distinguished points as described in Table \ref{tab Vertices and distinguished points, and their corresponding weights, ideal coordinates}.
\begin{table}[ht!]
\centering
\begin{tabular}{c | c}
    Regular vertices & Weight \\
    \hline\hline
    $(0,0,0)$ & $(\lambda_3, \lambda_2, \lambda_1)$ \\ \hline
    $(0, 0, \lambda_2-\lambda_3)$ & $(\lambda_2, \lambda_3, \lambda_1)$ \\ \hline
    $(0, \lambda_1-\lambda_2, -\lambda_1+\lambda_2)$ & $(\lambda_3, \lambda_1, \lambda_2)$ \\ \hline
    $(0, \lambda_1-\lambda_3, -\lambda_1+\lambda_2)$ & $(\lambda_2, \lambda_1, \lambda_3)$ \\ \hline
    $(\lambda_1-\lambda_2, 0, \lambda_1-\lambda_3)$ & $(\lambda_1, \lambda_3, \lambda_2)$ \\ \hline\vspace{0.3cm}
    $(\lambda_1-\lambda_2, \lambda_2-\lambda_3, \lambda_1-\lambda_2)$ & $(\lambda_1, \lambda_2, \lambda_3)$ \\
    Irregular vertex and distinguished points & Weight \\
    \hline\hline
    $(\lambda_1-\lambda_2, 0, \lambda_1-\lambda_2)$ & $(\lambda_1-\lambda_2+\lambda_3, \lambda_2, \lambda_2)$ \\ \hline
    $(0, \lambda_1-\lambda_2, -\lambda_1+2\lambda_2-\lambda_3)$ & $(\lambda_2, \lambda_1-\lambda_2+\lambda_3, \lambda_2)$ \\ \hline
    $(0, \lambda_2-\lambda_3, 0)$ & $(\lambda_2, \lambda_2, \lambda_1-\lambda_2+\lambda_3)$ \\
\end{tabular}
\caption{Vertices and distinguished points, and their corresponding weights, ideal coordinates} \label{tab Vertices and distinguished points, and their corresponding weights, ideal coordinates}
\end{table}
\end{ex}

\subsection{Ideal fillings}

In the previous section we saw that we could tropicalise the critical point to obtain a unique point in the superpotential polytope. In \cite[Proposition 5.6]{Judd2018}, Judd shows that we obtain the same point by first tropicalising the critical point conditions and then looking for solutions of this new system. In the same paper he relates this point to a new combinatorial object he introduces: ideal fillings.

In this section we generalise this relation from the $SL_n$ case to the $GL_n$ case. In order to do so we first extend Judd's definition of ideal fillings to be suitable for working with $GL_n$ and then describe the tropical critical point conditions.

The benefit of considering ideal fillings will be a better description of the tropical critical point, and thus also the preimage of the weight polytope centre of mass under the weight projection.

\begin{defn}
Take a grid of $n(n-1)/2$ boxes in upper triangular form and assign a non-negative real number to each box. This is called a filling and written as $\{n_{ij}\}_{1\leq i<j \leq n}$.

A filling is said to be ideal if $n_{ij}=\max\{n_{i+1, j}, n_{i, j-1} \} $ for $j-i\geq 2$ and is integral if all the $n_{ij}$ are integral.
\end{defn}

For an example with $n=4$, see Figure \ref{Ideal filling for n 4}. We note that an ideal filling is completely determined by the entries in the first diagonal since $n_{ij}= \max_{i\leq k\leq j-1}\{n_{k , k+1}\}$.

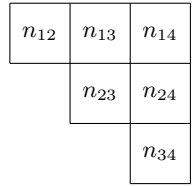
\begin{figure}[ht]
\centering
\begin{tikzpicture}
    %box lines
    %horizontal
    \draw (0,0) -- (2.4,0);
    \draw (0,-0.8) -- (2.4,-0.8);
    \draw (0.8,-1.6) -- (2.4,-1.6);
    \draw (1.6,-2.4) -- (2.4,-2.4);

    % \draw (3.2,0) -- (4,0);
    % \draw (3.2,-0.8) -- (4,-0.8);
    % \draw (3.2,-1.6) -- (4,-1.6);
    % \draw (3.2,-2.4) -- (4,-2.4);

    % \draw (3.2,-3.2) -- (4,-3.2);
    % \draw (3.2,-4) -- (4,-4);

    %vertical
    \draw (0,0) -- (0,-0.8);
    \draw (0.8,0) -- (0.8,-1.6);
    \draw (1.6,0) -- (1.6,-2.4);
    \draw (2.4,0) -- (2.4,-2.4);

    % \draw (3.2,0) -- (3.2,-2.4);
    % \draw (4,0) -- (4,-2.4);

    % \draw (3.2,-3.2) -- (3.2,-4);
    % \draw (4,-3.2) -- (4,-4);

    %nij labels
    \node at (0.4,-0.4) {\small{$n_{12}$}};

    \node at (1.2,-0.4) {\small{$n_{13}$}};
    \node at (1.2,-1.2) {\small{$n_{23}$}};

    \node at (2,-0.4) {\small{$n_{14}$}};
    \node at (2,-1.2) {\small{$n_{24}$}};
    \node at (2,-2) {\small{$n_{34}$}};

    % \node at (3.6,-0.4) {\small{$n_{1n}$}};
    % \node at (3.6,-1.2) {\small{$n_{2n}$}};
    % \node at (3.6,-2) {\small{$n_{3n}$}};
    % \node at (3.6,-3.6) {\scriptsize{$n_{n-1, n}$}};
\end{tikzpicture}
\caption{Ideal filling for $n=4$} \label{Ideal filling for n 4}
\end{figure}

For a dominant integral weight $\lambda$ of $SL_n$, Judd in \cite{Judd2018} defined an ideal filling for $\lambda$ to be an ideal filling $\{n_{ij}\}_{1\leq i<j \leq n}$ such that $\sum n_{ij}\alpha_{ij} = \lambda$. Unfortunately, although this definition is suitable when working with $SL_n$, it is not sufficient for $GL_n$. We take the following generalisation:

\begin{defn}\label{def ideal filling for lambda GLn version}
We say that $\{n_{ij}\}_{1\leq i<j \leq n}$ is an ideal filling for a dominant weight $\lambda$ of $GL_n$, if it is an ideal filling and $\sum n_{ij}\alpha_{ij} + \ell \sum \epsilon_i  = \lambda$, where $\ell := \frac1n \sum \lambda_i$.
\end{defn}

It is worth noting that this is the same $\ell$ comes up in the weight matrix at a critical point (see Corollary \ref{cor weight matrix at trop crit point}).

Returning our attention to critical points, we recall the conditions given in (\ref{eqn crit pt conditions}), which define critical points in the fibre over some $d\in T^{\vee}$:
    $$ \sum_{a:h(a)=v} r_a = \sum_{a:t(a)=v} r_a \quad \text{for} \ v \in \mathcal{V}^{\bullet}.
    $$
Working over the field of generalised Puiseux series we consider critical points of the superpotential in the fibre over some $t^{\lambda}\in T^{\vee}(\mathbf{K}_{>0})$ for a dominant weight $\lambda$. % \in {X^*(T)}^+$.
Tropicalising the above expression, and writing $\rho_a:=\mathrm{Val}_{\mathbf{K}}(r_a)$ following our notational convention, we obtain the tropical critical point conditions:
    \begin{equation} \label{eqn trop crit point conds}
    \min_{a:h(a)=v}\{\rho_a\} = \min_{a:t(a)=v}\{\rho_a\} \quad \text{for} \ v \in \mathcal{V}^{\bullet}.
    \end{equation}
This system has a unique solution (\cite[Proposition 5.6]{Judd2018}), given by the valuation of the critical point. We will often refer to (\ref{eqn trop crit point conds}) as the tropical critical point conditions for $\lambda$, or with highest weight $\lambda$, to highlight the representation theoretic connection.

With these definitions in mind we now give the main result of this section. It is a generalisation of a proposition given by Judd in \cite{Judd2018}, extending it from the $SL_n$ case to the $GL_n$ case.

\begin{prop}[Generalisation of {\cite[Proposition 6.2]{Judd2018}}] \label{prop GLn version of Jamie's 6.2}
Let $\lambda$ be a dominant weight of $GL_n$ and $\ell := \frac1n \sum \lambda_i$. Then we have a bijective correspondence:
    $$\begin{Bmatrix} \text{solutions to the tropical critical}\\ \text{conditions (\ref{eqn trop crit point conds}) with highest weight } \lambda \end{Bmatrix} \leftrightarrow \begin{Bmatrix} \text{ideal fillings } \{n_{ij}\}_{1\leq i<j \leq n} \text{ for }\lambda, \\ \text{i.e. such that } \sum n_{ij}\alpha_{ij} + \ell \sum \epsilon_i  = \lambda \end{Bmatrix}
    $$
\end{prop}

\begin{proof} Judd, in his proof of \cite[Proposition 6.2]{Judd2018}, defines a pair of maps between the two sets which are inverse to each other. We will follow the majority of his proof. Consequently, it suffices to simply give an outline which highlights the necessary generalisations. The exception to this is a new proof that the filling we construct (from the solutions to the tropical critical point conditions) is indeed an ideal filling.

\medskip

\noindent
\textbf{Map from ideal fillings for $\lambda$ to solutions to the tropical critical conditions.}

Let $\{n_{ij}\}_{1\leq i<j\leq n}$ be an ideal filling for $\lambda$. For each pair $(i,j)$ such that $1\leq i\leq j\leq n$, we define two sums of entries of the ideal filling; roughly speaking, those $n_{il}$ strictly to the right of $n_{ij}$ and those $n_{lj}$ strictly above $n_{ij}$ respectively:
    \begin{equation} \label{eqn Hh Hv definitions}
    H^h_{ij}:= \sum_{l>j}n_{il}, \quad H^v_{ij}:= \sum_{l<i}n_{lj}.
    \end{equation}
Making the first adaption of Judd's proof; for $\ell = \frac{1}{n} \sum \lambda_i$, we define a map from ideal fillings for $\lambda$ to tropical vertex coordinates of the quiver as follows:
    $$\delta_{v_{ji}}:=H^h_{ij}-H^v_{ij}+\ell.
    $$
We need to show that this defines a solution to the tropical critical conditions for $\lambda$.

Of note, the addition of $\ell$ in the above definition doesn't affect the tropical arrow coordinates. Indeed, computing the corresponding vertical arrow coordinates for $1\leq i \leq j < n$, and the horizontal arrow coordinates for $1\leq i < j \leq n$, we respectively obtain
    $$\delta_{v_{ji}}-\delta_{v_{j+1,i}}=H^v_{i+1,j+1} - H^v_{ij}
    \quad \text{and} \quad
    \delta_{v_{ji}}-\delta_{v_{j,i+1}}=H^h_{i,j-1} - H^h_{i+1,j}.
    $$
Both of these are $\geq 0$, so it follows that the point lies in $\left\{\mathrm{Trop}(\mathcal{W}_{t^{\lambda}})\geq 0\right\}$. Additionally, we see it will lie in the fibre over $\lambda$ as follows: for $\epsilon^{\vee}_k \in X^*(T^{\vee})$ we have
    $$\begin{aligned}
    \lambda_k = \left\langle \lambda, \epsilon^{\vee}_k \right\rangle
        &= \bigg\langle \sum_{1\leq i < j \leq n} n_{ij}(\epsilon_i-\epsilon_j) + \ell \sum_{1\leq i \leq n}\epsilon_i , \ \epsilon^{\vee}_k \ \bigg\rangle & \text{since } \{n_{ij}\} \text{ is an ideal filling for } \lambda \\
    &= H^h_{kk}-H^v_{kk}+\ell = \delta_{v_{kk}}.
    \end{aligned}
    $$

It remains to show that the point we have defined satisfies the tropical critical point conditions. Following Judd, we require a lemma:
\begin{lem}[{\cite[Lemma 6.7]{Judd2018}}] \label{lem Hh is Hh or Hv is Hv}
For $1 \leq i < j \leq n$, write $\bar{H}^h_{ij}:= H^h_{ij}+n_{ij}$, $\bar{H}^v_{ij}:= H^v_{ij} + n_{ij}$.
Then if $j-i\geq 1$, either
    $$\bar{H}^v_{i,j+1}=\bar{H}^v_{ij} \quad \text{or} \quad \bar{H}^h_{i,j+1}=\bar{H}^h_{i+1,j+1}
    $$
or both are true. Hence we have $\min\{ \bar{H}^v_{i,j+1} - \bar{H}^v_{ij}, \bar{H}^h_{i,j+1}-\bar{H}^h_{i+1,j+1} \} = 0$.
\end{lem}
We may use this lemma directly and so omit the proof. Let $v_{ji} \in \mathcal{V}^{\bullet}$ with $1<i<j<n$, that is, $v_{ji}$ doesn't lie on either wall of the quiver. Then the minimum over incoming arrows at $v_{ji}$ is
    $$\min\{ H^v_{i+1,j+1} - H^v_{ij}, H^h_{i,j-1}-H^h_{i+1,j} \} = n_{ij} + \min\{ \bar{H}^v_{i,j+1} - \bar{H}^v_{ij}, \bar{H}^h_{i,j+1}-\bar{H}^h_{i+1,j+1} \} = n_{ij}.
    $$
Similarly the minimum over outgoing arrows at $v_{ji}$ is
    $$\min\{ H^v_{i+1,j} - H^v_{i,j-1}, H^h_{i-1,j-1}-H^h_{ij} \} = n_{ij} + \min\{ \bar{H}^v_{i-1,j} - \bar{H}^v_{i-1,j-1}, \bar{H}^h_{i-1,j}-\bar{H}^h_{ij} \} = n_{ij}.
    $$
Thus the tropical critical point conditions are satisfied in this case. Finally, if $v_{ij}$ lies on the left wall there is only one outgoing arrow, $H^v_{2,j}-H^v_{1,j-1}=n_{1j}$, and if it lies on the bottom wall there is only one incoming arrow, $H^h_{i,n-1}-H^h_{i+1,n}=n_{in}$. Thus our point is indeed a tropical critical point for $\lambda$, as required.

\medskip

\noindent
\textbf{Map from solutions to the tropical critical conditions to ideal fillings.}

Suppose $(\rho_a)_{\ \in \mathcal{A}}$ is a solution to the tropical critical conditions for $\lambda$. Then for $v \in \mathcal{V}^{\bullet}$ we define the map
    $$\pi : \mathcal{V}^{\bullet} \to \mathbb{R}, \quad \pi(v):=\min_{a:h(a)=v}\{\rho_a\}
    .
    $$
We will first deviate from Judd's work to give an alternative proof that $\{ n_{ij}=\pi(v_{ij}) \}_{1\leq i < j \leq n}$ defines an ideal filling. Then we follow his proof to see that this is an ideal filling for $\lambda$.

\begin{lem} \label{lem trop crit pt is an ideal filling}
At a tropical critical point, the filling $\{n_{ij} = \pi(v_{ji})\}$ is an ideal filling. That is, if we have the following sub-diagram
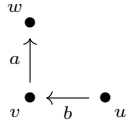
\begin{center}
\begin{tikzpicture}
    %dots and stars
    \node (31) at (0,0) {$\bullet$};
        \node at (-0.2,-0.2) {\scriptsize{$v$}};
    \node (21) at (0,1) {$\bullet$};
        \node at (-0.2,1.2) {\scriptsize{$w$}};
    \node (32) at (1,0) {$\bullet$};
        \node at (1.2,-0.2) {\scriptsize{$u$}};

    %arrows
    \draw[->] (31) -- (21);
    \draw[->] (32) -- (31);

    %arrow labels
    \node at (-0.2,0.5) {\scriptsize{$a$}};
    \node at (0.5,-0.2) {\scriptsize{$b$}};
\end{tikzpicture}
\end{center}
then we must have $\pi(v)=\max\{ \pi(u), \pi(w)\}$.
\end{lem}

\begin{proof}
We show by induction that $\pi(t(b))\leq \pi(h(b))$ for each horizontal arrow $b$ and we have $\pi(h(a))\leq \pi(t(a))$ for each vertical arrow $a$.

First we consider the arrows in the bottom and left hand walls, described in Figure \ref{fig Arrows in the bottom and left hand walls of the quiver}.
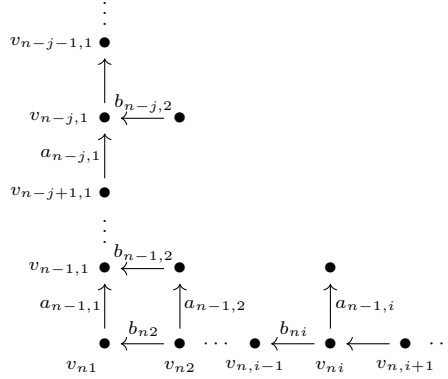
\begin{figure}[ht!]
\centering
\begin{tikzpicture}
    %dots and stars
    \node (n1) at (0,0) {$\bullet$};
        \node at (-0.3,-0.3) {\scriptsize{$v_{n1}$}};
    \node (n2) at (1,0) {$\bullet$};
        \node at (1,-0.3) {\scriptsize{$v_{n2}$}};
            \node at (1.5,0) {\scriptsize{$\cdots$}};
    \node (ni-1) at (2,0) {$\bullet$};
        \node at (2,-0.3) {\scriptsize{$v_{n,i-1}$}};
    \node (ni) at (3,0) {$\bullet$};
        \node at (3,-0.3) {\scriptsize{$v_{ni}$}};
    \node (ni+1) at (4,0) {$\bullet$};
        \node at (4,-0.3) {\scriptsize{$v_{n,i+1}$}};
        \node at (4.5,0) {\scriptsize{$\cdots$}};

    \node (n-11) at (0,1) {$\bullet$};
        \node at (-0.6,1) {\scriptsize{$v_{n-1,1}$}};
        \node at (0,1.6) {\scriptsize{$\vdots$}};
    \node (n-j+11) at (0,2) {$\bullet$};
        \node at (-0.7,2) {\scriptsize{$v_{n-j+1,1}$}};
    \node (n-j1) at (0,3) {$\bullet$};
        \node at (-0.6,3) {\scriptsize{$v_{n-j,1}$}};
    \node (n-j-11) at (0,4) {$\bullet$};
        \node at (-0.7,4) {\scriptsize{$v_{n-j-1,1}$}};
        \node at (0,4.5) {\scriptsize{$\vdots$}};

    \node (n-12) at (1,1) {$\bullet$};
    \node (n-1i) at (3,1) {$\bullet$};

    \node (n-j2) at (1,3) {$\bullet$};

    %arrows
    \draw[->] (ni+1) -- (ni);
    \draw[->] (ni) -- (ni-1);
    \draw[->] (n2) -- (n1);

    \draw[->] (n-12) -- (n-11);
    \draw[->] (n-j2) -- (n-j1);

    \draw[->] (n1) -- (n-11);
    \draw[->] (n-j+11) -- (n-j1);
    \draw[->] (n-j1) -- (n-j-11);

    \draw[->] (n2) -- (n-12);
    \draw[->] (ni) -- (n-1i);

    %arrow labels
    \node at (-0.45,0.5) {\scriptsize{$a_{n-1,1}$}};
    \node at (1.47,0.5) {\scriptsize{$a_{n-1,2}$}};
    \node at (3.47,0.5) {\scriptsize{$a_{n-1,i}$}};

    \node at (-0.45,2.5) {\scriptsize{$a_{n-j,1}$}};

    \node at (0.52,0.2) {\scriptsize{$b_{n2}$}};
    \node at (0.52,1.2) {\scriptsize{$b_{n-1,2}$}};
    \node at (0.52,3.2) {\scriptsize{$b_{n-j,2}$}};

    \node at (2.52,0.2) {\scriptsize{$b_{ni}$}};
\end{tikzpicture}
\caption{Arrows in the bottom and left hand walls of the quiver} \label{fig Arrows in the bottom and left hand walls of the quiver}
\end{figure}

We recall that the tropical critical point conditions hold, namely $\min_{a:h(a)=v}\{\rho_a\} = \min_{a:t(a)=v}\{\rho_a\}$ for all dot vertices $v \in \mathcal{V}^{\bullet}$. Then by considering the outgoing arrows at $v_{ni}$ for $i=2, \ldots, n-1$, we see that
    $$\pi(v_{ni}) = \min\{ \rho_{a_{n-1,i}}, \rho_{b_{ni}} \} \leq \rho_{b_{ni}} = \pi(v_{n,i-1}).
    $$
Similarly considering the incoming arrows at $v_{n-j,1}$ for $j=1, \ldots, n-2$, we have
    $$\pi(v_{n-j,1}) = \min\{ \rho_{a_{n-j,1}}, \rho_{b_{n-j,2}} \} \leq \rho_{a_{n-j,1}} = \pi(v_{n-j+1,1}).
    $$

For the inductive step we will show that if we have a sub-diagram like the one in Figure \ref{fig Subquiver for the proof of Lem trop crit pt is an ideal filling},
such that $\pi(u)\leq\pi(v)$ and $\pi(w) \leq \pi(v)$, then $\pi(x)\leq\pi(w)$ and $\pi(x) \leq \pi(u)$.
\begin{figure}[hb!]
\centering
\begin{tikzpicture}
    %dots and stars
    \node (31) at (0,0) {$\bullet$};
        \node at (-0.2,-0.2) {\scriptsize{$v$}};
    \node (21) at (0,1) {$\bullet$};
        \node at (-0.2,1.2) {\scriptsize{$w$}};
    \node (32) at (1,0) {$\bullet$};
        \node at (1.2,-0.2) {\scriptsize{$u$}};
    \node (22) at (1,1) {$\bullet$};
        \node at (1.2,1.2) {\scriptsize{$x$}};

    %arrows
    \draw[->] (31) -- (21);
    \draw[->] (32) -- (22);

    \draw[->] (22) -- (21);
    \draw[->] (32) -- (31);

    %arrow labels
    \node at (-0.2,0.5) {\scriptsize{$a$}};
    \node at (1.2,0.5) {\scriptsize{$d$}};

    \node at (0.5,-0.2) {\scriptsize{$b$}};
    \node at (0.5,1.2) {\scriptsize{$c$}};
\end{tikzpicture}
\caption{Subquiver for the proof of Lemma \ref{lem trop crit pt is an ideal filling}} \label{fig Subquiver for the proof of Lem trop crit pt is an ideal filling}
\end{figure}
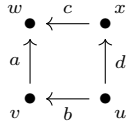

If $\rho_a \leq \rho_c$ then by the tropical box relation $\rho_a+\rho_b=\rho_c+\rho_d$ we have $\rho_d \leq \rho_b$. This means that $\pi(w)=\rho_a$ and $\pi(u)=\rho_d$, so
    \begin{equation} \label{eqn a leq c trop crit pt is ideal filling} \pi(x) \leq \rho_d = \pi(u) \leq \pi(v) \leq \rho_a = \pi(w).
    \end{equation}
Similarly if $\rho_a\geq \rho_c$ then we have $\rho_d \geq \rho_b$. This means that $\pi(w)=\rho_c$ and $\pi(u)=\rho_b$, so
    \begin{equation} \label{eqn a geq c trop crit pt is ideal filling} \pi(x) \leq \rho_c = \pi(w) \leq \pi(v) \leq \rho_b = \pi(u).
    \end{equation}
In both cases we have $\pi(x)\leq\pi(w)$ and $\pi(x) \leq \pi(u)$ as desired.

Finally we note that $\rho_a\leq \rho_c$ implies $\pi(v)=\pi(w) \geq \pi(u)$ since by (\ref{eqn a leq c trop crit pt is ideal filling}) and the inductive assumption we have
    $$ \pi(u)\leq \pi(v) \leq \pi(w) \leq \pi(v).
    $$
Similarly if $\rho_a \geq \rho_c$ then $\pi(v)=\pi(u) \geq \pi(w)$ as a consequence of (\ref{eqn a geq c trop crit pt is ideal filling}).
This completes the proof of Lemma \ref{lem trop crit pt is an ideal filling}.
\end{proof}

Now, again following Judd, we will show that $\{ n_{ij}=\pi(v_{ij}) \}_{1\leq i < j \leq n}$ is an ideal filling for $\lambda$. To do this, we need the vertex coordinates of the quiver at the tropical critical point, which we denote by $(\delta_v)_{v \in \mathcal{V}}$. In particular we notice that at the bottom left vertex we have
    \begin{align*}
    \delta_{v_{n1}}
    &= \mathrm{Val}_{\mathbf{K}}\left(\frac{\Xi_n}{\Xi_{n+1}}\right) &\text{by definition of } \Xi_i \text{ given in (\ref{eqn xi for wt map defn}), and noting } \Xi_{n+1}=1 \\
    &= \mathrm{Val}_{\mathbf{K}}(t_n) & \text{recalling the $t_i$ defined in (\ref{eqn defn gamma and t})} \\
    &= \ell & \text{by Corollary \ref{cor weight matrix at trop crit point}.}
    \end{align*}
We require a slight generalisation here:
\begin{lem}[Generalisation of {\cite[Lemma 6.9]{Judd2018}}] \label{lem Generalisation of Judd's Lemma 6.9}
For $v \in \mathcal{V}$ we write $\mathrm{bel}(v)$ and $\mathrm{lef}(v)$ for the sets of vertices directly below and directly to the left of $v$ respectively. Then at a tropical critical point we have
    $$\delta_v=\sum_{w \in \mathrm{bel}(v)} \pi(w) - \sum_{w \in \mathrm{lef}(v)} \pi(w) +\ell.
    $$
\end{lem}

\begin{proof}
The proof is by induction on the horizontal and vertical arrows. The initial case is the bottom left vertex, which we have already seen to take the value $\ell$ at critical points, as required. Since the horizontal and vertical inductive steps are similar it suffices to only consider the horizontal case; if we take the subquiver in Figure \ref{fig one arrow quiver}
\begin{figure}[hb!]
\centering
\begin{tikzpicture}
    %dots and stars
    \node (31) at (0,0) {$\bullet$};
        \node at (-0.2,-0.2) {\scriptsize{$v$}};
    \node (32) at (1,0) {$\bullet$};
        \node at (1.2,-0.2) {\scriptsize{$w$}};

    %arrows
    \draw[->] (32) -- (31);

    %arrow labels
    \node at (0.5,-0.2) {\scriptsize{$c$}};
\end{tikzpicture}
\caption{Subquiver for induction, horizontal case} \label{fig one arrow quiver}
\end{figure}
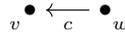
such that the relation in the statement of the lemma holds for $v$, then we need to show it also holds for $w$. To do this, we consider the subquiver in Figure \ref{fig Subquiver for the proof of Lem Generalisation of Judd's Lemma 6.9}.
\begin{figure}[hb!]
\centering
\begin{tikzpicture}
    %dots and stars
    \node (n-j1) at (0,3) {$\bullet$};
        \node at (-0.3,3) {\scriptsize{$v_0$}};
    \node (n-j+11) at (0,2) {$\bullet$};
        \node at (-0.3,2) {\scriptsize{$v_1$}};
        \node at (0,1.6) {\scriptsize{$\vdots$}};
    \node (n-11) at (0,1) {$\bullet$};
        \node at (-0.5,1) {\scriptsize{$v_{m-1}$}};
    \node (n1) at (0,0) {$\bullet$};
        \node at (-0.35,0) {\scriptsize{$v_m$}};

    \node (n-j2) at (1,3) {$\bullet$};
        \node at (1.3,3) {\scriptsize{$v'_0$}};
    \node (n-j+12) at (1,2) {$\bullet$};
        \node at (1.3,2) {\scriptsize{$v'_1$}};
    \node at (1,1.6) {\scriptsize{$\vdots$}};
    \node (n-12) at (1,1) {$\bullet$};
        \node at (1.5,1) {\scriptsize{$v'_{m-1}$}};
    \node (n2) at (1,0) {$\bullet$};
        \node at (1.35,0) {\scriptsize{$v'_m$}};

    %horizontal arrows
    \draw[->] (n-j2) -- (n-j1);
    \draw[->] (n-j+12) -- (n-j+11);
    \draw[->] (n-12) -- (n-11);
    \draw[->] (n2) -- (n1);

    %vertical arrows
    \draw[->] (n-j+11) -- (n-j1);
    \draw[->] (n1) -- (n-11);

    \draw[->] (n-j+12) -- (n-j2);
    \draw[->] (n2) -- (n-12);

    %arrow labels
    \node at (0.5,3.2) {\scriptsize{$c_0$}};
    \node at (0.5,2.2) {\scriptsize{$c_1$}};
    \node at (0.5,1.2) {\scriptsize{$c_{m-1}$}};
    \node at (0.5,0.2) {\scriptsize{$c_m$}};
\end{tikzpicture}
\caption{Subquiver for the proof of Lemma \ref{lem Generalisation of Judd's Lemma 6.9}} \label{fig Subquiver for the proof of Lem Generalisation of Judd's Lemma 6.9}
\end{figure}
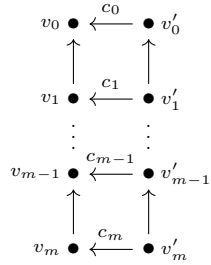

We suppose this is part of the full diagram which depicts a solution to the tropical critical conditions, and that $c_m$ lies on the bottom wall of this full quiver. Then, following Judd, we claim:
    \begin{equation}\label{eqn sum pi quiver vertices claim}
    \pi(v_0)+\pi(v_1)+ \cdots +\pi(v_m) = \rho_{c_0} +\pi(v'_1)+ \cdots +\pi(v'_m).
    \end{equation}
This can be proved by induction and, since it is unaffected by our addition of $\ell$ in the statement of Lemma \ref{lem Generalisation of Judd's Lemma 6.9}, we refer the reader to \cite{Judd2018} for the details.

Now, returning to Figure \ref{fig one arrow quiver}, we use the identity (\ref{eqn sum pi quiver vertices claim}) with $v_0=v$, $c_0=c$ and $v_0'=w$ to see that we have
    $$\pi(v)+ \sum_{u \in \mathrm{bel}(v)} \pi(u) = \rho_c + \sum_{u \in \mathrm{bel}(w)} \pi(u).
    $$
Using this we obtain the desired result, namely that the relation in the lemma holds for $w$:
    \begin{align*}
    \sum_{u \in \mathrm{bel}(w)} \pi(u) - \sum_{u \in \mathrm{lef}(w)} \pi(u) +\ell
    &= -\rho_c + \sum_{u \in \mathrm{bel}(v)} \pi(u) +\pi(v) - \sum_{u \in \mathrm{lef}(w)} \pi(u) +\ell \\
    &= -\rho_c + \sum_{u \in \mathrm{bel}(v)} \pi(u) - \sum_{u \in \mathrm{lef}(v)} \pi(u) +\ell \\
    &= -\rho_c + \delta_v \\
    &= \delta_w.
    \end{align*}
\end{proof}

Using this lemma we see that at a tropical critical point, the ideal filling $\{ n_{ij}=\pi(v_{ij}) \}$ is an ideal filling for $\lambda$:
    $$\begin{aligned}
    \bigg\langle \sum_{1\leq i < j \leq n} n_{ij}\alpha_{ij} +\ell \sum_{1\leq i \leq n}\epsilon_i, \ \epsilon^{\vee}_k \bigg\rangle
    &=  \sum_{k < j \leq n} n_{kj} - \sum_{1\leq i < k} n_{ik} +\ell \\
    &= \sum_{k < l \leq n} \pi(v_{lk}) - \sum_{1\leq l < k} \pi(v_{kl}) + \ell \\
    &= \sum_{w \in \mathrm{bel}(v_{kk})} \pi(w) - \sum_{w \in \mathrm{lef}(v_{kk})} \pi(w) + \ell \\
    &= \delta_{v_{kk}} = \lambda_k.
    \end{aligned}$$

To complete the proof of Proposition \ref{prop GLn version of Jamie's 6.2}, we note that the maps defined above are inverse to each other by construction.
\end{proof}

\begin{rem}
The bijection in the above proposition (\ref{prop GLn version of Jamie's 6.2}) preserves integrality if $\lambda$ and $\ell$ are both integral.
This proposition also implies that there is a unique ideal filling for $\lambda$ due to the uniqueness of the tropical critical point.
\end{rem}

\begin{cor} \label{cor nu'_i coords ideal filling for lambda}
For a dominant weight $\lambda$, let the positive critical point $p_{\lambda} \in Z_{t^{\lambda}}(\mathbf{K}_{>0})$ of $\mathcal{W}_{t^{\lambda}}$ be written in the ideal coordinates $(m_1, \ldots, m_N)$. Then the valuations $\mu_k=\mathrm{Val}_{\mathbf{K}}(m_k)$ defining the tropical critical point, $p_{\lambda, \boldsymbol{\mu}}^{\mathrm{trop}}$, give rise to an ideal filling $\left\{n_{ij}= \mu_{s_i+j-i} \right\}_{1\leq i < j \leq n}$ for $\lambda$ (where we recall the definition of $s_i$ given in Section \ref{sec The ideal coordinates}).
\end{cor}

\begin{proof}
By Theorem \ref{thm crit points, sum at vertex is nu_i}, at a critical point we have
    $$ \varpi(v_{ji}) = \sum_{a:t(a)=v_{ji}} r_a = m_{s_i +j-i}.$$
Thus by Proposition \ref{prop GLn version of Jamie's 6.2} we see that
    $$n_{ij} = \pi(v_{ji}) =\mathrm{Val}_{\mathbf{K}}(\varpi(v_{ji})) = \mathrm{Val}_{\mathbf{K}}(m_{s_i +j-i}) = \mu_{s_i+j-i}
    $$
defines an ideal filling for $\lambda$.
\end{proof}

\begin{ex} \label{ex trop crit pt from ideal filling dim 3}

By Proposition \ref{prop GLn version of Jamie's 6.2} we have a one to one correspondence between solutions to the tropical critical conditions and ideal fillings for $\lambda = (\lambda_1\geq\lambda_2\geq\lambda_3)$. In this example we show that given an ideal filling in dimension 3, imposing the condition that this ideal filling is an ideal filling for $\lambda$ is the same as restricting our attention to those points with weight $(\ell, \ell, \ell)$ where $\ell=\frac13 \sum \lambda_i$ (see Corollary \ref{cor weight matrix at trop crit point}). Moreover, this will aid our geometric intuition.

\begin{figure}[ht!]
\centering
    \begin{minipage}[b]{0.45\linewidth}
    \centering
    \begin{tikzpicture}
        %box lines
        %horizontal
        \draw (0,0) -- (1.6,0);
        \draw (0,-0.8) -- (1.6,-0.8);
        \draw (0.8,-1.6) -- (1.6,-1.6);

        %vertical
        \draw (0,0) -- (0,-0.8);
        \draw (0.8,0) -- (0.8,-1.6);
        \draw (1.6,0) -- (1.6,-1.6);

        %nij labels
        \node at (0.4,-0.4) {\small{$n_{12}$}};

        \node at (1.2,-0.4) {\small{$n_{13}$}};
        \node at (1.2,-1.2) {\small{$n_{23}$}};
    \end{tikzpicture}
    \end{minipage}
\hspace{-1cm}
    \begin{minipage}[b]{0.45\linewidth}
    \centering
    \begin{tikzpicture}
        %box lines
        %horizontal
        \draw (0,0) -- (1.6,0);
        \draw (0,-0.8) -- (1.6,-0.8);
        \draw (0.8,-1.6) -- (1.6,-1.6);

        %vertical
        \draw (0,0) -- (0,-0.8);
        \draw (0.8,0) -- (0.8,-1.6);
        \draw (1.6,0) -- (1.6,-1.6);

        %nij labels
        \node at (0.4,-0.4) {\small{$\beta_1$}};

        \node at (1.2,-0.4) {\small{$\beta_2$}};
        \node at (1.2,-1.2) {\small{$\beta_3$}};
    \end{tikzpicture}
    \end{minipage}
\caption{Fillings in dimension $3$} \label{fig Fillings in dimension 3}
\end{figure}
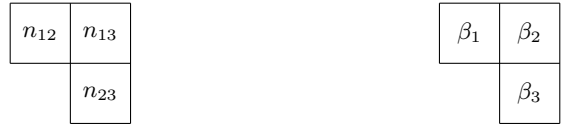

The filling in dimension $3$ is given in Figure \ref{fig Fillings in dimension 3}, we will write $\beta_1=n_{12}$, $\beta_2=n_{13}$, $\beta_3=n_{23}$.
The ideal filling condition $\beta_2 = \max\{\beta_1, \beta_3\}$ defines the following piecewise-linear subspace:
    \begin{equation} \label{eqn dim 3 ideal filling subpace}
    \{\beta_2=\beta_1\geq \beta_3\}\cup \{\beta_2=\beta_3\geq \beta_1\}.
    \end{equation}
The condition that the ideal filling is an ideal filling for $\lambda$ is the following:
    $$
    \lambda = \sum n_{ij}\alpha_{ij} + \ell \sum \epsilon_i \nonumber = \beta_1 \alpha_{12} + \beta_2 \alpha_{13} + \beta_3\alpha_{23} + \ell  \sum \epsilon_i \nonumber = \left( \ell + \beta_1+\beta_2, \ell -\beta_1+\beta_3, \ell -\beta_2-\beta_3 \right).
    $$
This gives a further set of constraints on the $c_i$. These additional constraints may also be obtained by setting the weight of a point $(\beta_1,\beta_2,\beta_3)$ equal to the weight of the positive critical point;
    $$(\lambda_3 + \beta_2 + \beta_3, \lambda_2 + \beta_1 - \beta_3, \lambda_1 - \beta_1 - \beta_2 )=( \ell, \ell, \ell ).$$
Intersecting this condition with (\ref{eqn dim 3 ideal filling subpace}) we find exactly two possibilities for the tropical critical point, depending on which of $\lambda_1-\lambda_2$ or $\lambda_2-\lambda_3$ is greater. By Proposition \ref{prop GLn version of Jamie's 6.2} this point lies within the superpotential polytope. We obtain:
\begin{itemize}
    \item If $\beta_1=\beta_2$ then $\lambda_1- 2\beta_1 = \ell$ and $\lambda_3+\beta_1+\beta_3 = \lambda_2 +\beta_1 - \beta_3$. So
    $$(\beta_1,\beta_2,\beta_3)=\left( \frac16 \left( 2\lambda_1 -\lambda_2-\lambda_3 \right), \frac16 \left( 2\lambda_1 -\lambda_2-\lambda_3 \right), \frac12 \left(\lambda_2-\lambda_3\right) \right).$$

    \item If $\beta_2=\beta_3$ then $\lambda_3 + 2\beta_2 = \ell $ and $ \lambda_2+\beta_1-\beta_2 = \lambda_1 -\beta_1 - \beta_2$. So
    $$(\beta_1,\beta_2,\beta_3)=\left( \frac12 \left(\lambda_1-\lambda_2\right), \frac16 \left( \lambda_1 +\lambda_2-2\lambda_3 \right), \frac16 \left( \lambda_1 +\lambda_2-2\lambda_3 \right) \right).
    $$
\end{itemize}

Computing the tropical critical point coordinates in this way is much quicker and easier than solving the simultaneous equations given by the tropical critical point conditions. Additionally it is now unsurprising that there is only one ideal filling for $\lambda$ in dimension $3$, since we are intersecting a the piecewise-linear 2-dimensional subspace and a line.

Alternatively we could obtain the same result using the proof of Proposition \ref{prop GLn version of Jamie's 6.2}.
The benefit is that, in addition to the tropical critical point, we would also gain vertex coordinates (and thus arrow coordinates) for the quiver. Unfortunately however, since this approach is algorithmic we lose some of the more visual interpretation.
\end{ex}

\subsection{A family of ideal polytopes} \label{subsec A family of ideal polytopes}

In this section we define a family of polytopes which contains the ideal polytope $\mathcal{P}_{\lambda,\boldsymbol{\mu}}$. We do this by extending the definition of $\mathcal{P}_{\lambda,\boldsymbol{\mu}}$ to general reduced expressions $\mathbf{i}$ for $\bar{w}_0$.

Firstly, recalling the construction of the ideal coordinates defined in Section \ref{sec The ideal coordinates}, we begin our generalisation by by taking an arbitrary reduced expression $\mathbf{i}=(i_1, \ldots, i_N)$ for $w_0$ and considering the map
    \begin{equation*}
    \tilde{\psi}_{\mathbf{i}} : \left(\mathbb{K}^*\right)^N \times T^{\vee} \to Z\,, \quad \left(\left(m'_1, \ldots, m'_N\right), t_R\right) \mapsto \mathbf{y}_{i_1}^{\vee}\left(\frac{1}{m'_1}\right) \cdots \mathbf{y}_{i_N}^{\vee}\left(\frac{1}{m'_N}\right) t_R.
    \end{equation*}
We again wish to work with the highest weight as opposed to the weight, that is, coordinates $(d,\boldsymbol{m}')$ instead of $(\boldsymbol{m}',t_R)$, however this requires us to develop our description of the weight map. To do so it will be better to index the coordinates $\boldsymbol{m}'$ by positive roots as follows:

We recall that any reduced expression $\mathbf{i}=(i_1, \ldots, i_N)$ for $w_0$ determines an ordering on the set of positive roots $R_+$ by
    \begin{equation} \label{eqn ordering on R plus}
    \alpha^{\mathbf{i}}_j=\begin{cases}
    \alpha_{i_1} & \text{for } j=1, \\
    s_{i_1}\cdots s_{i_{j-1}}\alpha_{i_{j}} & \text{for } j=2, \ldots N.
    \end{cases}
    \end{equation}
This ordering has the property that whenever $\alpha,  \beta \in R_+$ are positive roots such that $\alpha + \beta \in R_+$, then $\alpha + \beta$ must occur in between $\alpha$ and $\beta$.
We use the ordering on $R_+$ defined by $\mathbf{i}$ to identify $\left(\mathbb{K}^*\right)^N$ with $\left(\mathbb{K}^*\right)^{R_+}$, namely $m'_j=m'_{\alpha^{\mathbf{i}}_j}$. We will write $\left(\mathbb{K}^*\right)^N_{\mathbf{i}}$, $\left(\mathbb{K}^*\right)^{R_+}_{\mathbf{i}}$ when we need to explicitly state which reduced expression we are using.

We also recall the classic arrangement of positive roots $\alpha_{ij}=\epsilon_i-\epsilon_j$, $i<j$, similar to a strictly upper triangular matrix:
    \begin{equation} \label{eqn alpha ij arrangement}
    \begin{matrix}
    \alpha_{12} & \alpha_{13} & \cdots & \alpha_{1,n-1} & \alpha_{1n} \\
    & \alpha_{23} & \cdots & \alpha_{2,n-1} & \alpha_{2n} \\
    & & \ddots & \vdots & \vdots \\
    & & & \alpha_{n-2,n-1} & \alpha_{n-2,n} \\
    & & & & \alpha_{n-1,n}
    \end{matrix}
    \end{equation}
In particular, this takes the same form as (ideal) fillings, for which we have the natural bijective correspondence
    \begin{equation} \label{eqn rel nij alpha ij}
    n_{ij} \leftrightarrow \alpha_{ij}.
    \end{equation}
We note that if $\alpha,  \beta \in R_+$ are positive roots such that $\alpha + \beta \in R_+$, then we must have $\alpha+\beta$ appearing either to the right of $\alpha$ and above $\beta$, or to the right of $\beta$ and above $\alpha$. This is a consequence of the fact that if $\alpha=\alpha_{ij}$ and $\beta=\alpha_{kl}$ then, in order for their sum $\alpha+\beta=\alpha_{ij}+\alpha_{kl}$ to be a positive root, we must have either $i=l$ or $j=k$ (resulting in $\alpha+\beta=\alpha_{kj}$ or $\alpha+\beta=\alpha_{il}$ respectively).

\begin{defn}[Universal weight map] \label{defn universal weight}
We define a map $t_R : T^{\vee} \times \left(\mathbb{K}^*\right)^{R_+} \to T^{\vee}$ by taking $t_R(d,\boldsymbol{m}')$ to be the $n\times n$ diagonal matrix with entries
    $$ \left(t_R(d,\boldsymbol{m}')\right)_{n-j+1,n-j+1} :
        = \frac{d_{j} \prod_{l=1}^{j-1} m'_{\alpha_{lj}}}{\prod_{l=j}^{n-1} m'_{\alpha_{j,l+1}}}.
    $$
We will refer to the matrix $t_R(d,\boldsymbol{m}')$ as the universal weight matrix.
\end{defn}

\begin{ex} \label{ex univ wt matrix n4}
When $n=4$ we have
    $$t_R(d,\boldsymbol{m}')=\begin{pmatrix}
    d_4 m'_{\alpha_{14}}m'_{\alpha_{24}}m'_{\alpha_{34}} & & & \\
    & d_3 \frac{m'_{\alpha_{13}}m'_{\alpha_{23}}}{m'_{\alpha_{34}}} & & \\
    & & d_2 \frac{m'_{\alpha_{12}}}{m'_{\alpha_{23}}m'_{\alpha_{24}}} & \\
    & & & d_1 \frac{1}{m'_{\alpha_{12}}m'_{\alpha_{13}}m'_{\alpha_{14}}}
    \end{pmatrix}.
    $$
\end{ex}

Now returning to our generalisation of the ideal coordinates, and recalling the construction given at the start Section \ref{sec The ideal coordinates}, we take a reduced expression $\mathbf{i}=(i_1, \ldots, i_N)$ for $w_0$ and define the ideal chart for $\mathbf{i}$ to be
    \begin{equation*} \label{eqn defn ideal chart arbitrary i}
    \psi_{\mathbf{i}} : T^{\vee} \times \left(\mathbb{K}^*\right)^{R_+} \to Z\,, \quad
    \left(d, \left(m'_{\alpha^{\mathbf{i}}_1}, \ldots, m'_{\alpha^{\mathbf{i}}_N} \right)\right) \mapsto
    \mathbf{y}_{i_1}^{\vee}\left(\frac{1}{m'_{\alpha^{\mathbf{i}}_1}}\right) \cdots \mathbf{y}_{i_N}^{\vee}\left(\frac{1}{m'_{\alpha^{\mathbf{i}}_N}}\right) t_R(d,\boldsymbol{m}')
    \end{equation*}
where $t_R(d,\boldsymbol{m}')$ is the universal weight matrix.

\begin{prop}
The universal weight map is independent of the choice of reduced expression $\mathbf{i}$ for $w_0$.
\end{prop}

\begin{proof}
Consider the reduced expression $\mathbf{i}_0= (1,2, \ldots, n-1, 1,2 \ldots, n-2, \ldots, 1, 2, 1)$. We will begin by showing that the description of the weight matrix given in Corollary \ref{cor wt matrix in m ideal coords} is the same as the universal weight matrix under the identification $m_j=m'_{\alpha^{\mathbf{i}_0}_j}$.

It is well known that the ordering (\ref{eqn ordering on R plus}) on $R_+$ given by $\mathbf{i}_0$ is
    $$\alpha_{12}, \ \alpha_{13}, \ \ldots, \  \alpha_{1n}, \ \alpha_{23}, \ \alpha_{24}, \ \ldots, \alpha_{2n}, \ \ldots, \ \alpha_{n-2,n-1}, \ \alpha_{n-2,n}, \ \alpha_{n-1,n}.
    $$
In particular, we see that $\alpha_{ij}$ appears in the $(s_i+j-i)$-th place in this sequence, where we recall the definition
    $$s_i := \sum_{k=1}^{i-1}(n-k).
    $$
This also follows from Corollary \ref{cor nu'_i coords ideal filling for lambda} and the correspondence (\ref{eqn rel nij alpha ij}). Thus $\alpha^{\mathbf{i}_0}_{s_i+(j-i)} = \alpha_{ij}$, and so $m'_{\alpha_{ij}} = m'_{\alpha^{\mathbf{i}_0}_{s_i+(j-i)}} = m_{s_i+(j-i)}$. Then recalling the description of the weight matrix from Corollary \ref{cor wt matrix in m ideal coords}, for $j=1, \ldots, n$ we have
    $$\begin{aligned}
    \left(t_R(d,\boldsymbol{m})\right)_{n-j+1, n-j+1}
        &= \frac{d_j \prod\limits_{k=1, \ldots,j-1} m_{s_k+(j-k)}}{\prod\limits_{r=1, \ldots, n-j} m_{s_j+r}} & \text{where } r=k-j \\
        &= \frac{d_j \prod_{k=1}^{j-1} m_{s_k+(j-k)} }{\prod_{k=j}^{n-1} m_{s_j+(k+1-j)}} \\
        &= \frac{d_j \prod_{k=1}^{j-1} m'_{\alpha_{kj}} }{\prod_{k=j}^{n-1} m'_{\alpha_{j,k+1}}} \\
        &= \left(t_R(d,\boldsymbol{m}')\right)_{n-j+1,n-j+1}
    \end{aligned}
    $$
and so the two descriptions agree for $\mathbf{i}=\mathbf{i}_0$.

Next we recall that any two reduced expressions $\mathbf{i}$ and $\mathbf{i}'$ for $w_0$, are related by a sequence of transformations
    \begin{alignat}{2} \label{eqn iji jij swap}
    i,j,i &\leftrightarrow j,i,j \qquad & \text{if } |i-j|=1, \\
    i,j &\leftrightarrow j,i & \text{if } |i-j|\geq 2. \nonumber
    \end{alignat}
It suffices to show that the form of the universal weight matrix is invariant under one of the transformations of the first type, (\ref{eqn iji jij swap}).

Suppose $\mathbf{i}$ and $\mathbf{i}'$ are two reduced expressions for $w_0$ which are related by a single transformation (\ref{eqn iji jij swap}) in positions $k-1, k, k+1$. Then their respective sequences of positive roots are
    \begin{equation} \label{eqn pos root sequences iji jij}
    \begin{aligned}
    \left( \alpha^{\mathbf{i}}_j \right) &= \left( \alpha^{\mathbf{i}}_1, \ldots, \alpha^{\mathbf{i}}_{k-2}, \alpha, \alpha+ \beta, \beta, \alpha^{\mathbf{i}}_{k+2}, \ldots, \alpha^{\mathbf{i}}_N  \right), \\
    \left( \alpha^{\mathbf{i}'}_j \right) &= \left( \alpha^{\mathbf{i}}_1, \ldots, \alpha^{\mathbf{i}}_{k-2}, \beta, \alpha+ \beta, \alpha, \alpha^{\mathbf{i}}_{k+2}, \ldots, \alpha^{\mathbf{i}}_N  \right).
    \end{aligned}
    \end{equation}
If $\boldsymbol{m}'$ and $\boldsymbol{m}''$ are such that $m'_j = m'_{\alpha^{\mathbf{i}}_j}$ and $m''_j = m''_{\alpha^{\mathbf{i}'}_j}$ respectively, then as a consequence of (\ref{eqn pos root sequences iji jij}) we must necessarily have
    $$\mathbf{y}_{\mathbf{i}_{k-1}}^{\vee}\left(\frac{1}{m'_{\alpha}}\right) \mathbf{y}_{\mathbf{i}_{k}}^{\vee} \left(\frac{1}{m'_{\alpha+\beta}}\right) \mathbf{y}_{\mathbf{i}_{k+1}}^{\vee}\left(\frac{1}{m'_{\beta}} \right) = \mathbf{y}_{\mathbf{i}_{k-1}}^{\vee}\left(\frac{1}{m''_{\beta}}\right) \mathbf{y}_{\mathbf{i}_{k}}^{\vee} \left(\frac{1}{m''_{\alpha+\beta}}\right) \mathbf{y}_{\mathbf{i}_{k+1}}^{\vee}\left(\frac{1}{m''_{\alpha}} \right).
    $$
Written explicitly this gives
    $$
    \begin{pmatrix} 1 & & & & & & \\ & \ddots & & & & & \\ & & 1 & & & & \\ & & \frac{m'_{\alpha}+m'_{\beta}}{m'_{\alpha}m'_{\beta}} & 1 & & & \\ & & \frac{1}{m'_{\alpha+\beta}m'_{\beta}} & \frac{1}{m'_{\alpha+\beta}} & 1 & & \\ & & & & & \ddots & \\ & & & & & & 1 \end{pmatrix}
    =\begin{pmatrix} 1 & & & & & & \\ & \ddots & & & & & \\ & & 1 & & & & \\ & & \frac{1}{m''_{\alpha+\beta}} & 1 & & & \\ & & \frac{1}{m''_{\beta}m''_{\alpha+\beta}} & \frac{m''_{\beta}+m''_{\alpha}}{m''_{\beta}m''_{\alpha}} & 1 & & \\ & & & & & \ddots & \\ & & & & & & 1 \end{pmatrix}
    $$
which defines the coordinate change:
    \begin{equation} \label{eqn alpha beta coord change iji jij}
    m''_{\alpha} = \frac{m'_{\alpha+\beta}(m'_{\alpha}+m'_{\beta})}{m'_{\beta}}, \quad m''_{\alpha+\beta} = \frac{m'_{\alpha}m'_{\beta}}{m'_{\alpha}+m'_{\beta}}, \quad m''_{\beta} = \frac{m'_{\alpha+\beta}(m'_{\alpha}+m'_{\beta})}{m'_{\alpha}},
    \end{equation}
with $m''_{\alpha^{\mathbf{i}}_j}=m'_{\alpha^{\mathbf{i}}_j}$ for all other coordinates.
In particular we have
    \begin{equation} \label{eqn prod and quo cood change m' m''}
    m''_{\alpha}m''_{\alpha+\beta}=m'_{\alpha}m'_{\alpha+\beta}, \quad m''_{\beta}m''_{\alpha+\beta} = m'_{\beta}m'_{\alpha+\beta}, \quad \frac{m''_{\alpha}}{m''_{\beta}}=\frac{m'_{\alpha}}{m'_{\beta}}.
    \end{equation}

It remains to show that the form of the universal weight matrix is unaltered by this coordinate change. Recalling the definition of the universal weight matrix (\ref{defn universal weight}) we notice that the product in the numerator (resp. denominator) has exactly one term $m'_{\alpha_{ij}}$ for every $\alpha_{ij}$ from the $(j-1)$-th column (resp. $j$-th row) of the arrangement (\ref{eqn alpha ij arrangement}).
We recall also that the root $\alpha+\beta \in R_+$ must lie either to the right of $\alpha$ and above $\beta$, or to the right of $\beta$ and above $\alpha$ in the arrangement (\ref{eqn alpha ij arrangement}). It follows then that for every $\left(t_R(d,\boldsymbol{m}')\right)_{ii}$, at most one of $m'_{\alpha}$ or $m'_{\beta}$ can appear in each of the two products in the description of this matrix entry, and the coordinate $m'_{\alpha+\beta}$ can appear in at most one of the two products.
We note that it is impossible for $m'_{\alpha}$ or $m'_{\beta}$ to appear in one of the products in $\left(t_R(d,\boldsymbol{m}')\right)_{11}$ or $\left(t_R(d,\boldsymbol{m}')\right)_{nn}$ without $m'_{\alpha+\beta}$ also appearing.

Suppose $m'_{\alpha}$ and $m'_{\alpha+\beta}$ both appear in the numerator of $\left(t_R(d,\boldsymbol{m}')\right)_{n-j+1,n-j+1}$ (the proof starting with these terms in the denominator follows similarly). Then by definition $\alpha=\alpha_{lj}$ and $\alpha+\beta=\alpha_{kj}$ for some $k<l<j$. Since $\alpha+\beta$ is a positive root we must have $\beta=\alpha_{kl}$, and so we see that $m'_{\beta}$ cannot appear in this matrix entry. Consequently, by (\ref{eqn prod and quo cood change m' m''}), the form of this matrix entry is unaffected by the coordinate change.

Now suppose $m'_{\alpha}$ appears in the numerator of $\left(t_R(d,\boldsymbol{m}')\right)_{n-j+1,n-j+1}$, but $m'_{\alpha+\beta}$ does not (the proof starting with $m'_{\alpha}$ in the denominator follows similarly). Then by definition $\alpha=\alpha_{lj}$ for some $l$. Since $m'_{\alpha+\beta}$ does not appear in the numerator but $\alpha+\beta$ is a positive root, we must have $\alpha+\beta=\alpha_{lk}$ with $k> j$. Consequently $\beta=\alpha_{jk}$ and so $m'_{\beta}$ must appear in the denominator of this matrix entry. Thus by (\ref{eqn prod and quo cood change m' m''}), the form of this matrix entry is unaffected by the coordinate change.
\end{proof}

Since we wish to define ideal polytopes corresponding to different reduced expressions $\mathbf{i}$ for $w_0$, we use the toric chart $\psi_{\mathbf{i}}$ to generalise two of the maps given in Section \ref{subsec Constructing polytopes}, namely we take
    $$\begin{aligned}
    \phi^{\mathbf{i}}_{t^{\lambda},\boldsymbol{m}'} &: (\mathbf{K}_{>0})^{R_+} \to Z_{t^{\lambda}}(\mathbf{K}_{>0}), \\
    \mathcal{W}^{\mathbf{i}}_{t^{\lambda},\boldsymbol{m}'} &: (\mathbf{K}_{>0})^{R_+} \to \mathbf{K}_{>0}
    \end{aligned}
    $$
such that $\phi_{t^{\lambda},\boldsymbol{m}'} = \phi^{\mathbf{i}_0}_{t^{\lambda},\boldsymbol{m}'}$ and $\mathcal{W}_{t^{\lambda},\boldsymbol{m}'}=\mathcal{W}^{\mathbf{i}_0}_{t^{\lambda},\boldsymbol{m}'}$, where $\lambda$ is a dominant weight. With this notation we are ready to construct our new polytopes; given an arbitrary reduced expression $\mathbf{i}$ for $w_0$, and the associated superpotential $\mathcal{W}^{\mathbf{i}}_{t^{\lambda},\boldsymbol{m}'}$ for $GL_n/B$, we define
    $$\mathcal{P}^{\mathbf{i}}_{\lambda,\boldsymbol{\mu}'} = \left\{ \boldsymbol{\alpha} \in \mathbb{R}^N_{\boldsymbol{\mu}'} \ \big| \ \mathrm{Trop}\left(\mathcal{W}^{\mathbf{i}}_{t^{\lambda},\boldsymbol{m}'}\right)(\boldsymbol{\alpha}) \geq 0 \right\}.
    $$

For the particular reduced expression $\mathbf{i}_0$, we obtain the ideal polytope from Section \ref{subsec Constructing polytopes}, namely $\mathcal{P}^{\mathbf{i}_0}_{\lambda,\boldsymbol{\mu}'} = \mathcal{P}_{\lambda,\boldsymbol{\mu}}$. Moreover we have already seen that this polytope $\mathcal{P}^{\mathbf{i}_0}_{\lambda,\boldsymbol{\mu}'}$ is simply a linear transformation of the string polytope $\mathrm{String}_{\mathbf{i}_0}(\lambda)=\mathcal{P}_{\lambda,\boldsymbol{\zeta}}$. However in general the families of string and ideal polytopes diverge:
\begin{prop} \label{prop pos birat map between toric charts for i}
Given a reduced expression $\mathbf{i}$ for $w_0$, there is a positive birational map of tori transforming the ideal coordinate chart for $\mathbf{i}$ into the string coordinate chart for $\mathbf{i}$:
    $$\begin{tikzcd}[column sep=1.5em]
    T^{\vee}\times\left(\mathbb{K}^*\right)^{R_+} \arrow{dr}[swap]{\psi_{\mathbf{i}}} \arrow[dashed]{rr}{\vartheta} && T^{\vee}\times\left(\mathbb{K}^*\right)^{N} \arrow{dl}{\varphi_{\mathbf{i}}} \\
    & Z
    \end{tikzcd}
    $$
\end{prop}

\begin{proof}
We recall the specific reduced expression
    $$\mathbf{i}_0 := (1,2, \ldots, n-1, 1,2 \ldots, n-2, \ldots, 1, 2, 1)
    $$
for $w_0$ used earlier, and define the map $\vartheta: T^{\vee}\times\left(\mathbb{K}^*\right)^{R_+}_{\mathbf{i}} \dashrightarrow T^{\vee}\times\left(\mathbb{K}^*\right)^{N}_{\mathbf{i}}$ to be the following composition:
    $$\begin{tikzcd}[row sep=0.3em]
    T^{\vee}\times\left(\mathbb{K}^*\right)^{R_+}_{\mathbf{i}} \arrow{r}{}
        & T^{\vee}\times\left(\mathbb{K}^*\right)^{N}_{\mathbf{i}} \arrow[dashed]{r}{}
        & T^{\vee}\times\left(\mathbb{K}^*\right)^{N}_{\mathbf{i}_0} \arrow{r}{}
        & T^{\vee}\times\left(\mathbb{K}^*\right)^{N}_{\mathbf{i}_0} \arrow[dashed]{r}{}
        & T^{\vee}\times\left(\mathbb{K}^*\right)^{N}_{\mathbf{i}} \\
    \left(d,\boldsymbol{m}'_{\alpha^{\mathbf{i}}}\right) \arrow[mapsto]{r}{}
        & \left(d,\boldsymbol{m}'\right) \arrow[mapsto]{r}{}
        & \left(d,\boldsymbol{m}\right) \arrow[mapsto]{r}{}
        & \left(d,\boldsymbol{z}\right) \arrow[mapsto]{r}{}
        & \left(d,\boldsymbol{z}'\right)
    \end{tikzcd}
    $$
The first map simply describes the identification between $T^{\vee}\times\left(\mathbb{K}^*\right)^{R_+}_{\mathbf{i}}$ and $T^{\vee}\times\left(\mathbb{K}^*\right)^{N}_{\mathbf{i}}$, given by the ordering (\ref{eqn ordering on R plus}) on the set of positive roots $R_+$ defined by $\mathbf{i}$, that is $m'_j=m'_{\alpha^{\mathbf{i}}_j}$. The third map is the coordinate change given in Theorem \ref{thm coord change} between the ideal and string coordinates for $\mathbf{i}_0$.

The second and fourth maps in the above composition are the necessary coordinate changes such that $\psi_{\mathbf{i}}(d, \boldsymbol{m}') = \psi(d, \boldsymbol{m})$ and $\varphi_{\mathbf{i}_0}(d, \boldsymbol{z}) = \varphi_{\mathbf{i}}(d, \boldsymbol{z}')$. The second map is given by compositions of coordinate changes similar to (\ref{eqn alpha beta coord change iji jij}) and the fourth map is defined analogously. Both are known to be positive rational maps, but in general, not isomorphisms of tori for arbitrary reduced expressions.
\end{proof}

\begin{ex}
We let $n=4$ and take $\mathbf{i}=(1,2,3,2,1,2)$, recalling that $\mathbf{i}_0=(1,2,3,1,2,1)$. This gives the ordering
    $$\alpha^{\mathbf{i}}_1=\alpha_{12} , \quad
    \alpha^{\mathbf{i}}_2=\alpha_{13} , \quad
    \alpha^{\mathbf{i}}_3=\alpha_{14} , \quad
    \alpha^{\mathbf{i}}_4=\alpha_{34} , \quad
    \alpha^{\mathbf{i}}_5=\alpha_{24} , \quad
    \alpha^{\mathbf{i}}_6=\alpha_{23}.
    $$

The coordinate changes we require, firstly between $\boldsymbol{m}'$ and $\boldsymbol{m}$, secondly between $\boldsymbol{m}$ and $\boldsymbol{z}$ (given by Theorem \ref{thm coord change}), and thirdly between $\boldsymbol{z}$ and $\boldsymbol{z}'$ are respectively as follows:
    $$\begin{aligned}
    m'_1 &= m_1 & m_1 &= z_6 \hspace{1.5cm}& z_1 &= z'_1 \\
    m'_2 &= m_2 & m_2 &= z_4 & z_2 &= z'_2 \\
    m'_3 &= m_3 & m_3 &= z_1 & z_3 &= z'_3 \\
    m'_4 &= \frac{m_5(m_4+m_6)}{m_4} \hspace{1.5cm} & m_4 &= \frac{z_5}{z_4} & z_4 &= \frac{z'_5z'_6}{z'_4z'_6+z'_5} \\
    m'_5 &= \frac{m_4m_5}{m_4+m_6} & m_5 &= \frac{z_2}{z_1} & z_5 &= z'_4z'_6 \\
    m'_6 &= \frac{m_5(m_4+m_6)}{m_6} & m_6 &= \frac{z_3}{z_2} & z_6 &= \frac{z'_4z'_6+z'_5}{z'_6} \\
    \end{aligned}
    $$
Combining the coordinate changes we have
    $$\begin{aligned}
    m'_{\alpha_{12}} & %= m'_{\alpha^{\mathbf{i}}_1} = m'_1
    = \frac{z'_4z'_6+z'_5}{z'_6} \hspace{1.5cm}&
        m'_{\alpha_{34}} & %= m'_{\alpha^{\mathbf{i}}_4} = m'_4
        = \frac{z'_2}{z'_1} + \frac{z'_3z'_5}{z'_1z'_4(z'_4z'_6+z'_5)} \\
    m'_{\alpha_{13}} & %= m'_{\alpha^{\mathbf{i}}_2} = m'_2
    = \frac{z'_5z'_6}{z'_4z'_6+z'_5} &
        m'_{\alpha_{24}} & %= m'_{\alpha^{\mathbf{i}}_5} = m'_5
        = \frac{z'_3z'_4(z'_4z'_6+z'_5)}{z'_2z'_4(z'_4z'_6+z'_5)+z'_3z'_5} \\
    m'_{\alpha_{14}} & %= m'_{\alpha^{\mathbf{i}}_3} = m'_3
    = z'_1 &
        m'_{\alpha_{23}} & %= m'_{\alpha^{\mathbf{i}}_6} = m'_6
        = \frac{z'_2z'_4z'_6}{z'_1z'_3} + \frac{z'_5z'_6}{z'_1(z'_4z'_6+z'_5)} \\
    \end{aligned}
    $$
\end{ex}

It follows from this example that, in general, given some dominant weight $\lambda$ and two reduced expressions $\mathbf{i}, \mathbf{i}'$ for $w_0$, the two polytopes $\mathcal{P}^{\mathbf{i}}_{\lambda,\boldsymbol{\mu}'}$ and $\mathcal{P}^{\mathbf{i}'}_{\lambda,\boldsymbol{\mu}''}$ are related by a piecewise-linear map. However in contrast, the respective tropical critical points $p_{\lambda, \boldsymbol{\mu}'}^{\mathrm{trop}}$ and $p_{\lambda, \boldsymbol{\mu}''}^{\mathrm{trop}}$ coincide:

\begin{prop} \label{prop trop crit pt independent of i}
For a given dominant weight $\lambda$, the tropical critical point $p_{\lambda, \boldsymbol{\mu}'}^{\mathrm{trop}}$ is independent of the choice of reduced expression $\mathbf{i}$ for $w_0$.
\end{prop}

\begin{proof}
It suffices to consider two reduced expressions $\mathbf{i}, \mathbf{i}'$ for $w_0$ that are related by a single transformation (\ref{eqn iji jij swap})
in positions $k-1,k,k+1$. Then the respective sequences of positive roots are given by (\ref{eqn pos root sequences iji jij}) and the coordinate change is given by (\ref{eqn alpha beta coord change iji jij}). We see that the tropical coordinate change is given by
    \begin{equation} \label{eqn trop coord change mu' mu''}
    \begin{aligned}
    \mu''_{\alpha} &= \mu'_{\alpha+\beta} + \min\{ \mu'_{\alpha}, \mu'_{\beta}\} - \mu'_{\beta}, \\
    \mu''_{\alpha+\beta} &= \mu'_{\alpha}+\mu'_{\beta} - \min\{ \mu'_{\alpha}, \mu'_{\beta}\}, \\
    \mu''_{\beta} &= \mu'_{\alpha+\beta} + \min\{ \mu'_{\alpha}, \mu'_{\beta}\} - \mu'_{\alpha}. \\
    \end{aligned}
    \end{equation}

Recall that the positive root $\alpha+\beta$ must appear either to the right of $\alpha$ and above $\beta$, or to the right of $\beta$ and above $\alpha$ in the arrangement (\ref{eqn alpha ij arrangement}). Consequently $\mu'_{\alpha+\beta}$ must appear either to the right of $\mu'_{\alpha}$ and above $\mu'_{\beta}$, or to the right of $\mu'_{\beta}$ and above $\mu'_{\alpha}$ in the filling
    $$\left\{ n_{ij}=\mu'_{\alpha_{ij}}=\mathrm{Val}_{\mathbf{K}}\left(m'_{\alpha_{ij}}\right)\right\}_{1\leq i < j \leq n}$$
(c.f. Corollary \ref{cor nu'_i coords ideal filling for lambda}), and similarly for $\mu''_{\alpha+\beta}$ in the respective filling. If we suppose that our filling is an ideal filling (for $\lambda$), then it follows that
    $$\mu'_{\alpha+\beta} = \max\{ \mu'_{\alpha}, \mu'_{\beta} \}.
    $$
Thus, by considering the tropical coordinate change (\ref{eqn trop coord change mu' mu''}), we see that at a critical point
    $$\begin{aligned}
    \mu''_{\alpha} &= \mu'_{\alpha+\beta} + \min\{ \mu'_{\alpha}, \mu'_{\beta}\} - \mu'_{\beta} = \mu'_{\alpha} , \\
    \mu''_{\alpha+\beta} &= \mu'_{\alpha}+\mu'_{\beta} - \min\{ \mu'_{\alpha}, \mu'_{\beta}\} %= \max\{\mu'_{\alpha},\mu'_{\beta}\}
        = \mu'_{\alpha+\beta} , \\
    \mu''_{\beta} &= \mu'_{\alpha+\beta} + \min\{ \mu'_{\alpha}, \mu'_{\beta}\} - \mu'_{\alpha} = \mu'_{\beta}. \\
    \end{aligned}
    $$
It follows that if we index our coordinates by positive roots then the tropical critical point is independent of the choice of reduced expression for $w_0$.
\end{proof}

\addcontentsline{toc}{section}{References}
\bibliographystyle{plain}
\bibliography{Ideal_Polytopes_for_Representations_of_GLnC_Teresa_Ludenbach}

\end{document}